%% file: H3D1st.tex
\documentclass[12pt]{amsart}
\usepackage{enumerate,amssymb,amsfonts,latexsym,psfrag}
\usepackage{amscd,amsmath}
\usepackage[dvips]{graphicx,epsfig}
\usepackage{epstopdf}
\usepackage[all]{xy}
\usepackage{fullpage}
\usepackage{subfig}
\usepackage{multirow,array}
\usepackage{color}
\usepackage{mathtools}
\usepackage{amsthm}

\usepackage{filecontents}

\input{macrosH3D1st.tex}
\usepackage[titletoc,title,header]{appendix}

\usepackage{titlesec}

\titleformat{\section}
  {\Large\bfseries\center}{\thesection.}{0.5em}{}
  \titlespacing*{\section} {0pt}{1em}{2.3ex plus .2ex}
\titlelabel{\thetitle.\quad}
\titleformat{\subsection}
  [runin]{\normalfont\bfseries}{\thesubsection.}{0.5em}{}[.]
\titlespacing*{\subsection}{0pt}{1em}{2.3ex plus .2ex}

\usepackage{titletoc}
\titlecontents{section}[1.5em]
  {\normalfont\bfseries}
  {\contentslabel{1.5em}}
  {\hspace*{-2.3em}}
 {\titlerule*[1pc]{}\contentspage}

\allowdisplaybreaks[1]

\begin{document}

\bigskip\bigskip

\title[H\'enon renormalization]{Renormalization of three dimensional H\'enon map I $ \colon $ Reduction of ambient space }

\address {Hong-ik University }
\date{August 7, 2014}
\author{Young Woo Nam}
\thanks{College of Science and Technology, Hongik University at Sejong, Korea. \newline 
Email $ \colon $ namyoungwoo\,@\,hongik.ac.kr}

\begin{abstract}

Three dimensional analytic H\'enon-like map
$$ F(x,y,z) = (f(x) - \eps(x,y,z),\, x,\, \de(x,y,z)) $$ 
and its {\em period doubling} renormalization is defined. If $ F $ is infinitely renormalizable map, Jacobian determinant of $ n^{th} $ renormalized map, $ R^nF $ has asymptotically universal expression
$$ \Jac R^nF = b_F^{2^n}a(x)(1 + O(\rho^n)) $$
where $ b_F $ is the average Jacobian of $ F $. 
The toy model map, $ F_{\mod} $ is defined as the map satisfying $ \di_z \eps \equiv 0 $. The set of toy model map is invariant under renormalizaton. Moreover, if $ \| \di_z \de \| \ll \| \di_y \eps \| $, then there exists the continuous invariant plane field over $ \OO_F $ with dominated splitting. Under this condition, three dimensional H\'enon-like map 
is dynamically decomposed into two dimensional map with contraction along the strong stable direction. Any invariant line field on this plane filed over $ \OO_{F_{\mod}} $ cannot be continuous. 

\end{abstract}

\maketitle


\thispagestyle{empty}

\setcounter{tocdepth}{1}

\tableofcontents


\newpage
\renewcommand{\labelenumi} {\rm {(}\arabic{enumi}{)}}

\section{Introduction}
Universality of infinitely renormalizable unimodal maps in one dimensional dynamical system was discovered by Feigenbaum and independently by Coullet and Tresser in the mid 1970's. Hyperbolicity at the fixed point of renormalization operator is finally proved in the one dimensional holomorphic dynamical systems by Lyubich using quadratic-like maps in \cite{Lyu}. Universality of higher dimensional maps which are strongly dissipative and close to the one dimensional maps has been expected. For example, see \cite{CEK}. 
The {\em period doubling} renormalization of two dimensional H\'enon-like map is introduced in \cite{CLM} with universal limit of renormalized maps. 
But geometry of Cantor attractor has different from one dimensional maps. In this paper, H\'enon renormalization theory is extended for three dimensional maps. This extension has in general two goals. 
\begin{itemize}
\item Finding the same or similar results of two dimensional theory in three dimension. \ssk
\item Finding new phenomena which appear only on three or higher dimensional maps.
\end{itemize} \ssk
\nin This paper concentrate on the first part of above goals. The content is the modified version of the first part of my thesis in \cite{Nam}. 

\msk
\subsection{Statement of result}

Three dimensional H\'enon-like map $F$ from the cubic box $ B $ to $ \R^3 $ is defined as follows
$$ F \colon (x,y,z) \mapsto (f(x) - \eps(x,y,z),\, x,\, \de(x,y,z)) $$
where $ f(x) $ is a unimodal map. Let us assume that $ \| \:\! \eps \|_{C^3}, \| \:\! \de \|_{C^3} \leq \bar \eps $ for a sufficiently small positive $ \bar \eps $.  
\ssk \\
We need the non linear coordinate change map which is called the horizontal-like diffeomorphism $ H $ for universal limit of renormalized maps as follows
$$ H : (x,y,z) \mapsto (f(x) - \eps(x,y,z),\, y,\, z - \de(y,f^{-1}(y),0))  . $$
Then H\'enon renormalization is extendible to three dimensional H\'enon-like maps with same definitions. Renormalized map $ RF $ of three dimensional H\'enon-like map, $F$ is defined as
$$ RF = \La \circ H \circ F^2 \circ H^{-1} \circ \La^{-1} $$
where $ \La $ is the appropriate dilation. 
$ R^nF $ converges to the universal limit as $ n \ra \infty $. Furthermore, $ F_* $ is the hyperbolic fixed point of the renormalization operator, $ R : F \mapsto RF $.  
\ssk \\
Let $ \II_B(\bar \eps) $ be the set of infinitely renormalizable H\'enon-like maps on the domain $ B $ with the norms $ \| \eps \| $ and $ \| \de \| $ bounded by small enough number, $ \bar \eps >0 $. 
Average Jacobian is defined on the critical Cantor set
$$ b_F = \exp \int_{\OO_F} \log \Jac F \,d\mu $$
where $ \mu $ is the unique ergodic probability measure on the Cantor attractor $ \OO_F $. Then 
for $ F \in \II_B(\bar \eps) $, universality of Jacobian is generalized for three dimensional map 
$$ \Jac R^nF = b^{2^n}a(x)(1+ O(\rho^n))$$
where $ b = b_F $ is the average Jacobian of $ F $ and $ a(x) $ is the universal function for $ \rho \in (0,1) $ (Theorem \ref{Universality of the Jacobian}).
\ssk \\
\nin However, universality of Jacobian determinant does not seem to imply any universal expression of three dimensional map, $ R^nF $.  
Thus instead of constructing universal theory of all three dimensional maps in $ \II (\bar \eps) $, we would find an invariant space under H\'enon renormalization operator and search geometric properties of Cantor attractor of maps in this class. 
Let H\'enon-like maps with the condition $ \di_z \eps \equiv 0 $ be {\em toy model maps}, say $ F_{\mod} $. 
Thus H\'enon renormalization of toy model map is a skew product of renormalization of two dimensional map with third coordinate. In other words, the following is true for every $ n \in \N $
$$ \pi_{xy} \circ R^nF_{\mod} = R^nF_{2d} $$
where $ \pi_{xy} $ is the projection from $ \R^3 $ to $ \R^2 $ and $ F_{2d} \colon (x,y) \mapsto (f(x) - \eps(x,y),\; x) $ is two dimensional H\'enon-like map. $ b_F = b_1 b_2 $ where $ b_1 $ is the average Jacobian of $ F_{2d} $ and $ b_2 \asymp \di_z \de $. Universality of two dimensional map and the fact that $  \Jac R^nF_{\mod} = \di_y \eps_n \cdot \di_z \de_n $ implies that universality of toy model maps.
$$ R^nF_{\mod} = (f_n(x) + b_1^{2^n}a(x)\,y\, (1+O(\rho^n)),\ x,\ b_2^{2^n}\!z \, (1+O(\rho^n))) $$
Let us assume that $ b_2 \ll b_1 $. 
Then there exists the continuous invariant plane field on the critical Cantor set under $ DF_{\mod} $. It is $ C^1 $ robust. In particular, $ \| \di_z \eps \| \ll b_F $ also implies the existence of continuous invariant plane field over $ \OO_{F_{\mod}} $. Moreover, there is no continuous invariant line field on this continuous invariant plane field. 
\msk \\
In the forthcoming paper, we would find another invariant space under H\'enon renormalization which does not require any invariant plane field. But geometric properties of Cantor attractor is the same as those of Cantor attractor of two dimensional H\'enon-like maps. 

\msk
\begin{acknowledgement}
I would like to thank specially my advisor, Marco Martens who introduced this topic and gave much inspirations through the discussion. Tangent bundle splitting on Section \ref{Tangent bundle splitting} originated from his insight. I also thank Misha Lyubich for his interest of this topic and his kindness.
\end{acknowledgement}

\bsk

\section{Preliminaries} \label{preliminaries} 
Let us introduce two dimensional H\'enon-like maps and its renormalization defined in \cite{CLM}. Many topological properties of two dimensional renormalizable H\'enon-like map are well adapted to three dimensional H\'enon-like maps.
\subsection{Notations}
For the given map $ F $, if a set $ A $ is related to $F$, then we denote it to be $ A(F) $ or $ A_F $ and $ F $ can be skipped if there is no confusion without $ F $. The domain of $ F $ is denoted to be $ \Dom(F) $. 
If $ F(B) \subset B $, then we call $ B $ is an (forward) invariant set under $ F $. 
\ssk \\
For three dimensional map, let us the projection from $ \R^3 $ to its $ x- $axis, $ y- $axis and $ z- $axis be $ \pi_x $, $ \pi_y $ and $ \pi_z $ respectively. Moreover, the projection from $ \R^3 $ to $ xy- $plane be $ \pi_{xy} $ and so on.
\ssk \\
Let $ C^r(X) $ be the Banach space of all real functions on $ X $ for which the $ r^{th} $ derivative is continuous. The $ C^r $ norm of $ h \in C^r(X) $ is defined as follows
$$ \| h \|_{C^r} = \max_{ 1 \leq \,k \leq \,r } \left\{ \| h \|_0,\ \| D^kh \|_0 \right\} . $$
For analytic maps, since $ C^0 $ norm bounds $ C^r $ norm for any $ r \in \N $, we often use the norm, $ \| \cdot \| $ instead of $ \| \cdot \|_0 $ or $ \| \cdot \|_{C^k} $. 
$ W^s(p) $ and $ W^u(p) $ are stable and unstable manifold at a point, $ p $. If $ W^s(p) $ is not connected in the given region, local stable manifold, $ W^s_{\loc}(p) $ is defined as the component of $ W^s(p) $ containing $ p $. If unstable manifold is one dimensional, then we express the curve connecting two points along the unstable manifold in $ X $ as follows
$$ [ \,p,q\, ]^u_{w} \subset W^u(w) . $$
The square bracket means that $ [ \,p,q\, ]^u_{w} $ is the homeomorphic image of a closed interval, for example, $ [-1,1] $ under a continuous map from $ \R $ to $ X $. The points $ p $ and $ q $ are the end points of curve. 
$ A = O(B) $ means that there exists a positive number $ C $ such that $ A \leq CB $. Moreover, $ A \asymp B $ means that there exists a positive number $ C $ which satisfies $ \dfrac{1}{C} B \leq A \leq CB$.
\msk

\subsection{Renormalization of one dimensional unimodal map}
Let $f : I \ra I$ be a $C^3$ or smoother unimodal map with non-degenerate critical point $c \in I$ and its Schwarzian derivative is negative on $I$. $f$ is called (period doubling) \textit{renormalizable} map if there exists a closed interval $ J \subset \inter I$ such that $J \cap f(J) =\varnothing $ and $f^2(J) \subset J$ which contains the critical point of $ f $. Then $f^2 : J \ra J$ is also a unimodal map on $J$. Then we can choose the minimal disjoint intervals $J_c = [f^4(c), f^2(c)]$ and $J_v = [f^3(c), f(c)] $ which are invariant under $f^2$. Thus {\em renormalization} $R_cf$ at the critical point as $R_c f$ is defined as $ s f^2(s^{-1}x) $ where $ x \ra sx $ is the affine rescaling from $J_c$ to $I$, for $ s<-1 $. 
The domain of renormalizable map contains the critical point, the critical value and one repelling fixed point whose eigenvalue is negative. 
\ssk \\
Suppose that $f$ is an infinitely renormalizable map. 
Then there is the fixed point $f_*$ of the renormalization operator $R_c$ 
with universal scaling factor $\si = 1/2.6 \ldots $. The scaling factor of $n^{th}$ renormalized map converges to $\si$ exponentially fast as $n \rightarrow \infty$. 
Let the critical point of $ f_* $ be $ c_* $ and the interval be $ I = [-1,1] $. Also assume that $f_*(c_*) =1$ and $f^2_*(c_*) = -1$. Then the intervals, $J_c^* = [-1, f^4_*(c_*)]$ and $J_v^*= f_*(J^*_c) = [f^3_*(c_*),1]$ are the smallest invariant ones under $f_*^2$ around the critical point and the critical value respectively. Observe that the critical point $ c_* $ is in $ J_c^* $ and $ f_*(J_v^*) = J_c^* $. 
Let onto map $s \colon J^*_c \ra I$ be the orientation reversing affine rescaling. Thus $s \circ f_* \colon J^*_v \ra [-1, 1]$ is an expanding diffeomorphism.  Let us consider the inverse contraction
$$ g_* \colon I \ra J^*_v, \quad g_*=f^{-1}_* \circ s^{-1}
$$
where $f^{-1}_*$ is the branch of the inverse function which maps $J^*_c$ onto $J^*_v$. The map $g_*$ is called the {\em presentation function} and it has the unique fixed point at 1. 
By definition of $g_*$ implies that 
$$ f^2_* | J^*_v = g_* \circ f_* \circ (g_* )^{-1} .
$$
Then by appropriate rescaling of the presentation function, $g_*$, we can define {\em renormalization} at the {\em critical value}, 
$R^n_v f_*$. 
Inductively, $g^n_*$ is defined on the smallest interval $J^*_v(n)$ containing critical value, 1 with period $2^n$. Let $G^n_* \colon I \ra I$ be the diffeomorphism which is the appropriately rescaled map of $g_*^n$.
\msk \\
Then the fact that $g_*$ is a contraction implies the existence of the following limit,
$$ u_* = \lim_{n \ra \infty} G^n_* \colon I \ra I . $$
Moreover, the convergence is exponentially fast in $C^3$ topology.
Let us see the following lemmas in \cite{CLM}.

 \begin{lem} [Lemma 7.1 in \cite{CLM}]   \label{renormalization at critical value}
 For every $n \geq 1$ \msk
 \begin{enumerate}
 \item  $J^*_v(n) = g^n_*(I)$ \msk
 \item  $ R^n_v f_* = G^n_* \circ f_* \circ (G^n_*)^{-1}$ \msk
 \item  $ u_* \circ f_* = f^* \circ u_*$ \ssk
 \end{enumerate}
 \end{lem}
 
 \begin{lem} [Lemma 7.3 in \cite{CLM}]    \label{composition of the presentation function}
 Assume that there is the sequence of smooth functions  $g_k \colon I \ra I, \ \ k=1,2, \ldots n$ such that $\|g_k - g_*\|_{C^3} \leq C\rho^k$ where the $g_* = \lim_{k \ra \infty} g_k$ \, for some constant $C>0$ and $\rho \in (0,1)$. Let $g^n_k = g_k \circ \cdots \circ g_n$ and let $G^n_k = a^n_k \circ g^n_k \colon I \ra I$, where  $a^n_k$ is the affine rescaling of $\Im g^n_k$ to $I$. Then $\| G^n_k - G^{n-k}_* \|_{C^1} \leq C_1\rho^{n-k}$, where $C_1$ depends only on $\rho$ and $C$.
\end{lem}

\msk
\subsection{Two dimensional H\'enon-like map}
Let $B$ be a square region whose center is the origin and which contains the fixed points $ \beta_0 $ and $ \beta_1 $. Thus $B= I^h \times I^v$ where horizontal and vertical axes, say $I^h$ and $I^v$ are the (appropriately extended) symmetric intervals at zero of the one-dimensional renormalizable map $f$. 

\begin{figure}[htbp]
\begin{center}
\centering
\subfloat[A parabolic-like curve of the degenerate map as $  W^u(\beta_0) $]{
\psfrag{b0}[c][c][0.7][0]{\Large $\beta_0$}
\psfrag{b1}[c][c][0.7][0]{\Large $\beta_1$}
\label{paracurve}
\includegraphics[width=0.45\textwidth]{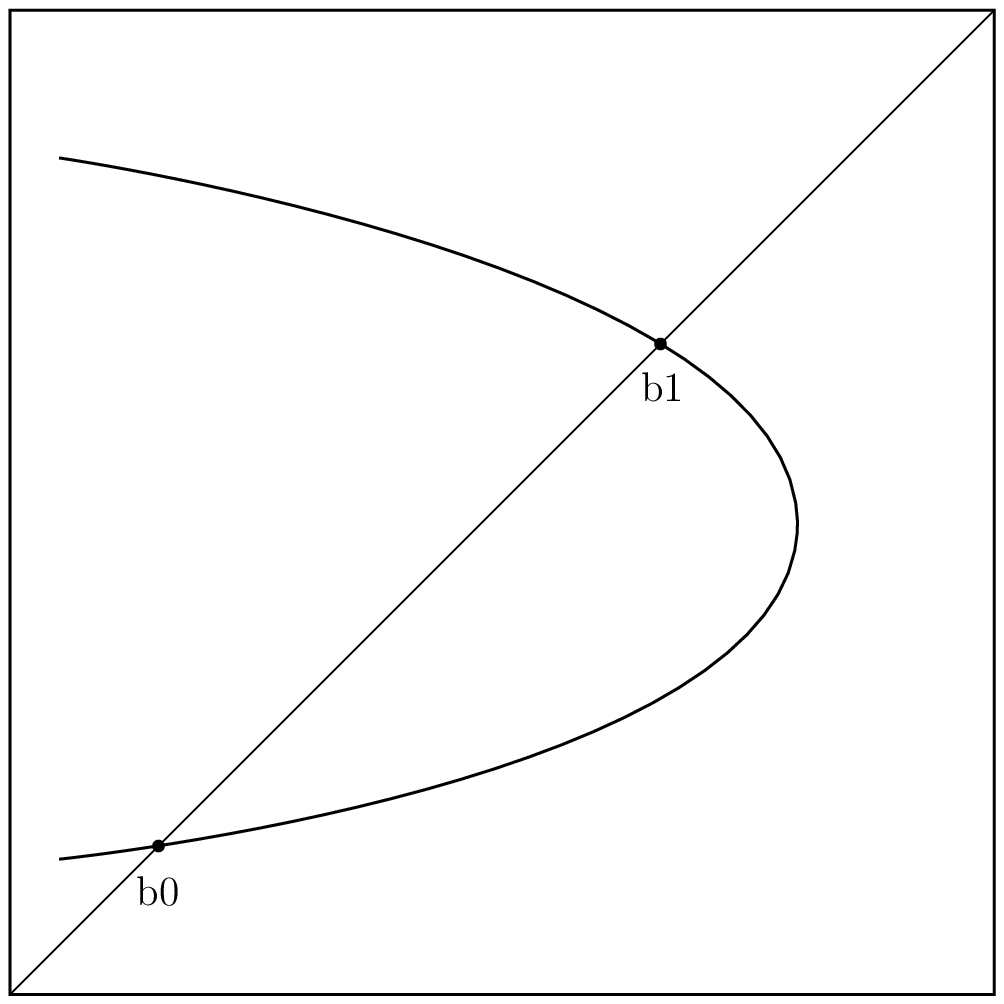} }
\subfloat[Stable and unstable manifolds of H\'enon-like map]{
\psfrag{Wb1}[c][c][0.7][0]{\Large $W^s_{\loc}(\beta_1)$}
\psfrag{Wb0}[c][c][0.7][0]{\Large $W^s_{\loc}(\beta_0)$}
\psfrag{b0}[c][c][0.7][0]{\Large $\beta_0$}
\psfrag{p0}[c][c][0.7][0]{\Large $p_0$}
\psfrag{p1}[c][c][0.7][0]{\Large $p_1$}
\psfrag{p2}[c][c][0.7][0]{\Large $p_2$}
\psfrag{b1}[c][c][0.7][0]{\Large $\beta_1$}
\includegraphics[width=0.45\textwidth]{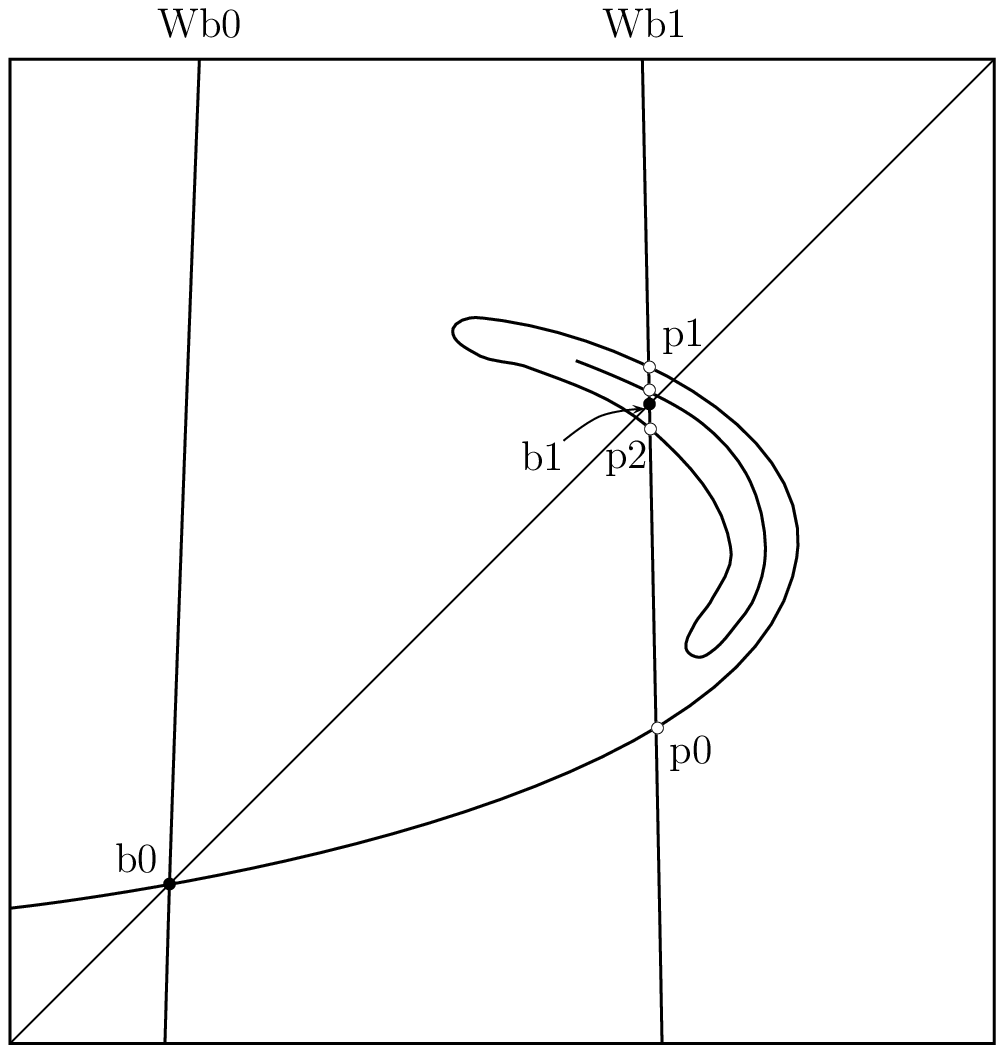} }
\end{center}
\caption{Degenerate map and H\'enon-like diffeomorphism}
\label{fig:unstable manifolds}
\end{figure}

\nin Two dimensional map $F : B \longrightarrow \R^2$ is called {\em H\'enon-like} map if the image of vertical line and horizontal line is a horizontal line and a parabolic-like curve respectively. Then H\'enon-like map $F$ is as follows
$$ F(x,y) = (f(x) - \eps(x,y),\ x).
$$
Additionally, if Jacobian determinant of $F$ is non-zero at every point in $ \Dom(F) $, then $F$ is called {\em H\'enon-like diffeomorphism}. 
We assume that two dimensional H\'enon-like map, $F$ has two saddle fixed points $\beta_0$ with positive eigenvalues, {\em flip saddle} and $\beta_1$ with negative eigenvalues, {\em regular saddle}. 
Period doubling renormalization of two dimensional analytic H\'enon-like map was defined in \cite{CLM}. Orientation preserving H\'enon-like map is called \textit{renormalizable} if the unstable manifold of $\beta_0$, $ W^u(\beta_0) $ intersects the stable manifold of $\beta_1$, $ W^s(\beta_1) $, at the single orbit of an intersection point, say $ \Orb_{\Z}(w) $. Let $p_0 \in \Orb_{\Z}(w) \cap W^s_{\loc}(\beta_1) $ be the {\em unique} point satisfying the following conditions. \msk
\begin{enumerate}
\item Every forward image of $ p $, namely, $p_k = F^k(p_0)$ for $ k \geq 0 $, is in $ W^s_{\loc}(\beta_1) $. \msk
\item Each backward images of $ p $ in $ \Dom(F) $ is disjoint from $ W^s_{\loc}(\beta_1) $. \ssk
\end{enumerate}
If $ \| \eps \| $ is small enough, then 
$ W^s_{\loc}(p_{-n}) $ is pairwise disjoint where $ n \leq 0 $. Moreover, $ W^s(\beta_1)$ and $ W^s_{\loc}(p_{-n}) $ converges to $ W^s(\beta_1)$ because $ p_{-n} $ converges to $\beta_0 $ as $ n \rightarrow +\infty $.

\begin{figure}[htbp]
\begin{center}

\psfrag{Z1}[c][c][0.7][0]{\large $Z_1$}
\psfrag{Z2}[c][c][0.7][0]{\large $Z_2$}
\psfrag{A0}[c][c][0.7][0]{\large $A_0$}
\psfrag{A-2}[c][c][0.7][0]{\large $A_{-2}$}
\psfrag{A-3}[c][c][0.7][0]{\large $A_{-3}$}
\psfrag{M1}[c][c][0.7][0]{\large $M_1$}
\psfrag{M-1}[c][c][0.7][0]{\large $M_{-1}$}
\psfrag{M-2}[c][c][0.7][0]{\large $M_{-2}$}
\psfrag{M-3}[c][c][0.7][0]{\large $M_{-3}$}
\psfrag{D}[c][c][0.7][0]{\large $D$}
\psfrag{F(D)}[c][c][0.7][0]{\large $F(D)$}

\psfrag{Wb0}[c][c][0.7][0]{\large $W^s_{\loc}(\beta_0)$}
\psfrag{M0}[c][c][0.7][0]{\large $M_0$}
\psfrag{b0}[c][c][0.7][0]{\large $\beta_0$}
\psfrag{p0}[c][c][0.7][0]{\large $p_0$}
\psfrag{p1}[c][c][0.7][0]{\large $p_1$}
\psfrag{p2}[c][c][0.7][0]{\large $p_2$}
\psfrag{p-1}[c][c][0.7][0]{\large $p_{-1}$}
\psfrag{p-2}[c][c][0.7][0]{\large $p_{-2}$}
\psfrag{p-3}[c][c][0.7][0]{\large $p_{-3}$}
\psfrag{b1}[c][c][0.7][0]{\large $\beta_1$}
\includegraphics[scale =0.8]{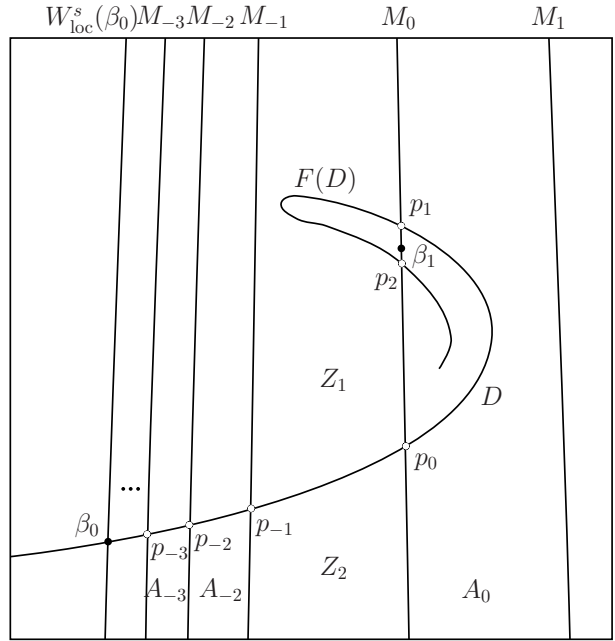} 
\end{center}
\caption{Regions between local stable manifolds}
\label{fig:invariant domain}
\end{figure}

\nin Denote $ W^s_{\loc}(p_{-n}) $ by $M_{-n}$ for every $n \geq 0$. For instance, $M_0$ is $ W^s_{\loc}(\beta_1)$. Moreover, define $M_1$ as the component of $ W^s(\beta_1)$ whose image under $F$ is contained in $M_{-1}$ such that it does not have any point of $\Orb_{\Z}(w)$. $ M_{1} $ is on the opposite side of $M_{-1}$ from $M_0$. We may assume that $M_1$ is a curve connecting the up and down sides of the square domain $B$ inside. Then we can easily check the curves $ [p_0, p_1]^u_{\beta_0} $ and  $ [p_1, p_2]^u_{\beta_0} $ which are parts of $ W^u(\beta_0) $ does not intersect $M_1$ and $M_{-1}$ respectively if $F$ is renormalizable.
\ssk \\
On the domain $B$, the dynamical region for renormalizable H\'enon-like maps is the closure of the component of $ B \setminus W^s(\beta_0) $ containing $\beta_1$, say $B_{\bullet}$ because it is an (forward) invariant region under $F$. Let each region between $M_{-n}$ and $ M_{-n+1} $ be $ A_{-n} $ for every $n \geq 0$. Since $ F(M_{-n}) \subset  M_{-n+1} $ for each $n \geq 0$, we can see $ F(A_{-n}) \subset  A_{-n+1} $ for each $n \geq 0$. But the image of $A_0$ under $F$ is contained on $A_{-1}$, that is, $ F(A_0) \subset A_{-1} $. In other words, $ W^s_{\loc}(\beta_1)$ intersects $ W^u(\beta_0) $ at $p_1$ transversally. 

\nin Let the region above the curve $ [p_{-1},p_0]^u_{\beta_0} $ in $A_{-1}$ be $Z_1$ and the region below the same curve in $A_{-1}$ be $Z_2$.  Let the domain enclosed by two curves $ [p_0, p_1]^u_{\beta_0} $ and $ [p_0, p_1]^s_{\beta_1} $ be $D$. Thus for the renormalizable H\'enon-like map, 
 $ F^2(A_0) \subset D $. Then $D$ is invariant under $F^2$ and furthermore, any neighbourhood of $D$ in $A_0$ is also invariant under $F^2$. 
\begin{lem}
Let $F$ be the renormalizable H\'enon-like map. Let the region between two local stable manifolds at $ \beta_1 $, $M_0$ and $M_1$ be $A_0$. Then $F^2(A_0) \subset D$. In particular, any open neighbourhood of $D$ in $A_0$ is invariant under $F^2$.
\end{lem}

\subsection{Renormalization operator of two dimensional H\'enon-like maps} \label{2d renorm operator}
The domain $D$ defined on the previous subsection is invariant under $F^2$. However, $F^2$ is not H\'enon-like map because the image of vertical line 
under $F^2$ is not a horizontal line. 
Then we need non-linear coordinate change map to define renormalization of H\'enon-like maps.  
Let $H$ 
be the \textit{horizontal diffeomorphism} defined as follows
$$ H(x,y) = (f(x) - \eps(x,y),\ y) $$
\nin The map $ H $ preserves each horizontal lines. Then by 
Lemma 3.4 in \cite{CLM}, 
$ H \circ F^2 \circ H^{-1} $ is a H\'enon-like map and this map is called {\em pre-renormalization} of $F$ and is denoted to be $ PRF $. Analytic definition of renormalization of F by \msk
\begin{equation} \label{eq-analytic definition of renormalization}
\begin{aligned}
RF = \La \circ PRF \circ \La^{-1}
\end{aligned} \msk
\end{equation}

\nin where $ \La $ is the dilation $ \La(x, y) = ( sx, sy) $ for the appropriate number $ s < -1 $. For instance, if the degenerate map $ F_{\bullet}(x, y) = (f(x), \; x) $ is renormalizable with its horizontal diffeomorphism $ H_{\bullet}(x, y) = (f(x), \; y) $, then 
$$ H_{\bullet}^{-1} \circ F_{\bullet}^2 \circ H_{\bullet} = (f^2(x),\; x) . $$
Then the renormalization of $ F_{\bullet} $ is 
$$ RF_{\bullet} = \La_{\bullet} \circ PRF_{\bullet} \circ \La^{-1}_{\bullet} = (s_{\bullet}f^2(s^{-1}_{\bullet} x),\ x) $$
where $ \La_{\bullet}(x) = s_{\bullet}x $ is a dilation of the unimodal renormalizable map $ x \ra f(x) $. Thus if $ \| \:\! \eps \| $ is small enough, then the dilation $ \La $ in the equation \eqref{eq-analytic definition of renormalization} is a $ \eps- $perturbation of $ \La_{\bullet} $. The pre-renormalization is defined on the region $ \La^{-1}(B) $. Let $ U $ be the interval, $ \pi_x( \La^{-1}(B)) $. Thus $ \La^{-1}(B) $ is extendible to $ U \times I^v $ with the full height. $ \Dom(H) $ is the region enclosed by curves $ f(x) - \eps(x,y) = \const . $ and $ y = \const . $ such that the image under $ H $ is $ \La^{-1}(B) $. Define $ V $ as the interval $ \pi_x( \Dom(H) \cap \{ y = 0 \}) $.

 \comm{**********************************
 \begin{figure}[htbp]
\begin{center}
\psfrag{M1}[c][c][0.7][0]{\large $M_1$}
\psfrag{M-1}[c][c][0.7][0]{\large $M_{-1}$}

\psfrag{M0}[c][c][0.7][0]{\large $M_0$}
\psfrag{p0}[c][c][0.7][0]{\large $p_0$}
\psfrag{p1}[c][c][0.7][0]{\large $p_1$}
\psfrag{p-1}[c][c][0.7][0]{\large $p_{-1}$}
\psfrag{b1}[c][c][0.7][0]{\large $\beta_1$}

\psfrag{B1c}[c][c][0.7][0]{\Large $B^1_c$}
\psfrag{B1v}[c][c][0.7][0]{\Large $B^1_v$}
\psfrag{Ga-(B)}[c][c][0.7][0]{\large $\La^{-1}(B)$}
\includegraphics[scale =0.8]{pieces-renor} 
\end{center}
\caption{Restricted pieces for renormalization}
\label{fig:Pieces for renor}
\end{figure}
**************************************}

\ssk
\begin{lem}[Lemma 3.4 on ~\cite{CLM}]
Let $F$ be analytic renormalizable H\'enon-like map with the small norm of $\eps $, $ \| \:\! \eps \| \leq \bar \eps$, then
 $$ H \circ F^2 \circ H^{-1} = (f_1(x) - \eps_1(x,y),\ x)
 $$
 for some unimodal map $f_1$ on $V$ such that $\| f^2- f_1 \|_V \leq C \bar \eps $ for some $C > 0$ and $ \| \eps_1 \| = O( \bar \eps^2)$.
\end{lem} \ssk
%

\nin Suppose H\'enon-like map $F$ is infinitely renormalizable. Then $ R^nF $ converges to the degenerate map $F_* = (f_*(x),\ x)$ exponentially fast as $ n \ra \infty $ where $f_*$ is the fixed point of the renormalization operator of unimodal maps. Hyperbolicity of renormalization operator of analytic unimodal maps was proved in \cite{Lyu}. The renormalization operator has codimension one stable manifold and one dimensional unstable manifold at the fixed point $f_*$. Moreover, exponential convergence of $ R^nF $ to the one dimensional fixed point $ (f_*(x),\,x) $ and super-exponential decay of $ \eps_n $ of $ R^nF $ implies the vanishing spectrum of $DR$, the derivative of renormalization operator. Hence, the unstable manifold at the fixed point of H\'enon renormalization operator is the same as that of renormalization operator of unimodal maps. See Section 4 in \cite{CLM}.

\bsk

\section{Renormalization of three dimensional H\'enon-like maps} \label{3D maps}

\subsection{H\'enon-like maps in three dimension}
Let $B$ be the three dimensional box domain, namely, $B = B_{2d} \times [-c,c]$ for some $c>0$ where $ B_{2d} $ be the domain of two dimensional H\'enon-like map. Moreover, denote $B $ as $ I^x \times  {\bf I}^v $ where $I^x$ is the line parallel to $x$-axis and ${\bf I}^v = I^y \times I^z$ where $I^y$ and $I^z$ are lines parallel to $y$-axis and $z$-axis respectively. Let us define {\em three dimensional H\'enon-like map} on the cube $B$ 
as follows 
\begin{equation}  \label{3d Henon map}
\begin{aligned}
F(x,y,z) = (f(x) - \eps(x,y,z), \ x, \ \de(x,y,z))
\end{aligned}
\end{equation}
where $f : I^x \rightarrow I^x $ is a unimodal map. 
\ssk 
\nin Observe that the image of the plane, $ \lbrace x = C \rbrace $ parallel to $yz$-plane under $F$ is contained in the plane, $ \lbrace y = C \rbrace  $ which is parallel to $ xz $-plane.

\begin{figure}[htbp]
\begin{center}
\psfrag{M1}[c][c][0.7][0]{\Large $F $}

\includegraphics[scale=1.0]{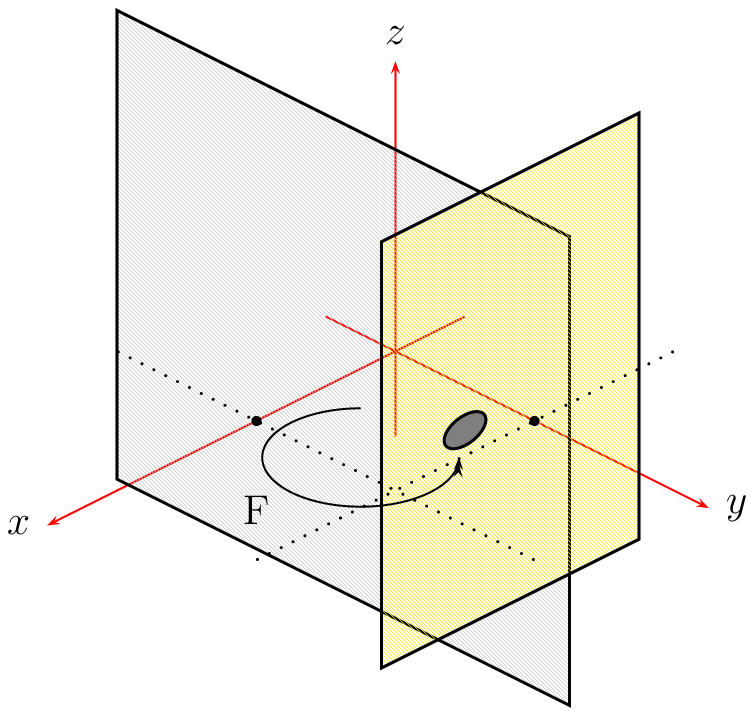}
\end{center}
\caption{Image of $ \lbrace x= \const . \rbrace$ under three dimensional H\'enon-like map}
\label{fig:imageofplane}
\end{figure}

\nin Let $ \bar \eps $ and $ \bar \de $ be small enough positive numbers. Assume that  $ \| \eps \|_{C^3} \leq \bar \eps $ and $\| \de \|_{C^3} \leq \bar \de$. If the unimodal map $ f(x) $ has two fixed points, then $F$ has also only two saddle fixed points, say $\beta_0$ and $\beta_1$, by contraction mapping theorem of the third coordinate map. Moreover, if the product of all eigenvalues of $ DF $ at each $ \beta_i $ for $ i=0,1 $ are close enough to zero, then each fixed point has 
one dimensional unstable manifold. Orientation preserving three dimensional H\'enon-like map is called \textit{renormalizable} if $ W^u(\beta_0) $ and $W^s(\beta_1) $ intersects in the single orbit of a point.
 \ssk  \\
Topological properties of renormalizable two dimensional H\'enon-like map are well extended to the renormalizable three dimensional one. 
See Figure \ref{fig:invariant domain} in page \pageref{fig:invariant domain} for the adaptation of three dimensional objects. Let $B_{\bullet}$ be the component of $B \setminus W^s(\beta_0) $ containing $ \beta_1 $, which is invariant under $F$. Definitions of $ M_i $ for $ i \leq 1 $, $ A_j $ for $ j \leq 0 $ is the same as those for two dimensional H\'enon-like maps.
$ W^s_{\loc}(\beta_1) $ is (forward) invariant under $F$ and it is the common boundary of the regions $ A_{-1} $ and $ A_0 $. Then $ F(A_{-1}) \subset A_0 $ and $ F(A_0) \subset A_{-1} $. In particular, $ A_0 $ is invariant under $ F^2 $ and $ F^2(A_0) $ contains a small neighborhood of $ [p_0, p_1]^u_{\beta_0} $ in $ A_0 $ and its boundary is disjoint from $M_1$. 
Then the following properties are the same as those of two dimensional H\'enon-like maps. 
\msk

\begin{enumerate}
\item  $M_0$ is invariant under $F$. \ssk
\item  $ F(M_{-n}) \subset  M_{-n+1} $ for each $n \geq 0$. \ssk
\item  $F(M_1) \subset M_{-1} $. \ssk
\item  $ F(A_{-n}) \subset  A_{-n+1} $ for each $n \geq 0$. In particular, $ F(A_{-1}) \subset  A_0 $. \ssk
\item  Let the region on the right side of $M_1$ be $A_1$. Then $F(A_1) \subset A_{-2}$. \ssk
\item  $ W^u(\beta_0) $ intersects $ W^s_{loc}(\beta_1)$ at $p_0$, $p_1$ and $p_2$ transversally. 
\end{enumerate}
\msk
\nin For three dimensional H\'enon-like map, define the cylindrical region $ D $ as follows.
\begin{equation} \label{eq-D invariant region under F2}
\begin{aligned}
D \equiv F(A_{-1})
\end{aligned} \msk
\end{equation}
Then $ D $ is invariant under $F^2$ in $A_0$.

\begin{rem}
The region $ D_{2d} $ for two dimensional H\'enon-like map is enclosed by two curves $ W^s_{\loc}(\beta_1) $ and $ W^u(\beta_0) $. $ D_{2d} $ depends only on the given H\'enon-like map. However, this definition is not valid for higher dimension. The region $ D \equiv F(A_{-1}) $ for three dimensional H\'enon-like map is a small neighborhood of the curve $ [p_0, p_1]^u \subset W^u(\beta_0) $ in the region $ A_0 $. 
\end{rem}

\ssk
\begin{lem} \label{attracting region in B}
Let $F$ be renormalizable three dimensional H\'enon-like map. Then $\overline B_{\bullet}$ is invariant under $F$. Moreover, for every point $w \in B_{\bullet}$, there exists $k \in \N$ such that $F^k(w) \in \overline D$.
\end{lem}
\begin{proof}
$  W^s(\beta_0) $ is invariant under $F$ and every $M_{-n} $ for some $-n \leq -1$ are components of the stable manifold $  W^s(\beta_0) $. Then we see that $ F^{n-1}(M_{-n}) \subset M_{-1} $ where $-n \leq -2$ and $ F(M_1) \subset M_{-1} $. Furthermore, by the definition of $D$, we see $ F(M_{-1}) \subset \di D $ and $ F(M_{0}) \subset \di D $. Then we can take $k = n$ where $ -n \leq -1 $, $ k=1 $ where $ n=0 $ and $ k=2 $ where $ n=1 $.
\ssk \\
Now let us take a point $w \notin \bigcup_{n \leq 1} M_{n} $. Let us show that $F^k(w) \in \overline D$ for some $k \geq 0$. 
We may assume that $w$ is contained in some region $A_{-n}$ for some $-n \leq 1$ because each region $A_{-n}$ is separated by $M_{-n}$ and $B_0$ is the union of $ M_{-n}$ and $ A_{-n}$. If $w \in A_{-n}$ where $-n \leq -1$, $F^{n-1}(w)$ is on $A_{-1}$. Let $w' $ be $ F^{n-1}(w)$. Then by the definition of $D$, $ F(w') \in D $. Moreover, if $ w \in A_0 $, then $F^2(w) \in D$ by the invariance of $D$ under $F^2$. If $ w \in A_1 $, then $ F(w) \in A_{-2} $. Hence, we can choose $ k = n $ where $ -n \leq -1 $, $ k=2 $ where $ n=0 $ and $ k = 3 $ where $ n=1 $.
\end{proof}

\ssk
\begin{cor} 
Let $F$ be renormalizable three dimensional H\'enon-like map. Let the region between two local stable manifolds $M_0$ and $M_1$ be $A_0$. Then $F^2(A_0) \subset D$. In particular, any open neighbourhood of $D$ in $A_0$ is invariant under $F^2$.
\end{cor}
\begin{proof}
Let us take any neighborhood of $D$ in $A_0$, say $D'$. Then we get the following set inclusion relation
$$ F^2(D) \subset F^2(D') \subset F^2(A_0) \subset F(A_{-1}) = D \subset D'  .
$$
Hence, $ F^2(D') \subset D' $.
\end{proof}

\nin As the result, we can choose arbitrary region $ D' \subset A_0 $ containing $D$ as an invariant domain under $F^2$. Let us take a region containing $ D $ such that $ \pi_{xy}(D) $ contains (relatively) compactly $ D_{2d} $ in $A_0$ where two dimensional region $ D_{2d} $ is enclosed by curves, $ [p_0, p_1]^{s}_{\beta_1} $ and $ [p_0, p_1]^{u}_{\beta_0} $. 
Express this extended region $ D' $ to be also $D$ unless it makes any confusion.


\begin{figure}[htbp]

\begin{center}
\psfrag{b0}[c][c][0.7][0]{\Large $\beta_0 $}
\psfrag{b1}[c][c][0.7][0]{\Large $\beta_1 $}
\psfrag{Ws} [c][c][0.7][0]{\Large $W_{\loc}^s$}
\includegraphics[scale=1.0]{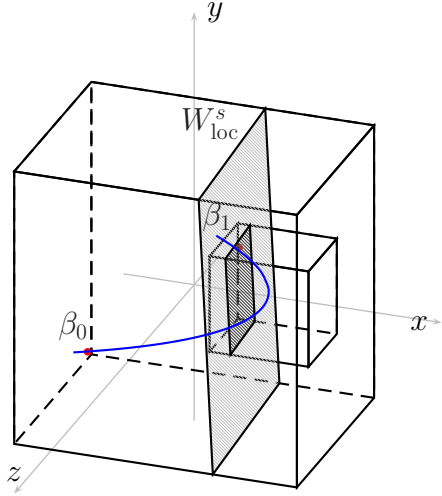}
\end{center}
\caption{The local stable manifold of $ \beta_1 $, $W_{\loc}^s(\beta_1)$ and the unstable manifold of $ \beta_0 $, $ W^u(\beta_0)  $ }
\label{fig:domain3D}
\end{figure}

\msk

\subsection{Renormalization of three dimensional H\'enon-like maps} \label{def of renormal}
In this section we construct {\em period doubling} renormalization operator of three dimensional analytic H\'enon-like maps. 
However, $F^2$ is not H\'enon-like map because the image of the plane, $ \lbrace x=C \rbrace $ in $B$ under $F^2$ is not a part of the plane, $ \lbrace y=C \rbrace $. Thus in order to construct renormalization operator, we need non-linear coordinate change map. Define the \textit{horizontal-like diffeomorphism} as follows
\begin{align} \label{horizontal-like diffeo}
  H(x,y,z) = (f(x) - \eps (x,y,z), \ y, \ z - \de (y, f^{-1}(y), 0)) .
\end{align}
\nin  Recall that $ J_c $ and $ J_v $ is the minimal invariant intervals under $ f^2 $ containing the critical point and the critical value of $ f $ respectively. Let $ V $ be a closed interval invariant under $ f^2 $ disjoint from $ J_v $. Suppose that $ V $ contains a small neighborhood of every $ J_c $ if the given unimodal maps are $ f(x) - \eps(x,\,y_0,\,z_0) $ for every $(y_0, z_0) \in {\bf I}^v $. Recall also that $ {\bf I}^v $ has the full length to $ y- $axis and $ z- $axis direction. 
\ssk \\
Let us take $ \Dom(H) $ as the region 
of which image under $ H $ is $ V \times {\bf I}^v $. Let us take the region in $ \Dom(H) $, say $ P $, of which faces satisfy the equations
$$ \{f(x) - \eps(x,y,z) = \const .\},\quad \{y = \const . \},\quad \{z = \const . \} $$
as boundary surfaces such that $ P $ is the minimal region invariant under $ F^2 $. Thus the ratio of each side of $ H(P) $ parallel to $ x $, $ y $ and $ z $ axis is 
$$ 1 \;\colon 1 + O(\bar \eps) \;\colon O(\bar \de) $$
Let us extend the minimal region in order to construct the cube as its image under $ H $. This extended region is called $ B^1_v $. Observe that $ B^1_v $ is compactly contained in $ \inter(\overline{A_{-1} \cup A_0}) $ and $ B^1_v $ is invariant under $ F^2 $.
\ssk \\
The inverse map $ H^{-1} $ is expressed as follows\footnote{ The first coordinate map of $ H^{-1}(w) $, $ \phi^{-1}(x,y,z) $ is {\em not} the {\em inverse function} of any one dimensional map. However, $ \phi^{-1}(w) $ is a $ \eps - $perturbation of $ f^{-1}(x) $ satisfying the following equation 
$$ f \circ \phi^{-1}(w) - \eps \circ H^{-1}(w) = x . $$  } 
\begin{align*}
H^{-1}(x,y,z) & \equiv  (\phi^{-1}(x,y,z),\ y,\ z + \de(y,\,f^{-1}(y),\,0)) .
\end{align*}
where $ \phi^{-1} $ is the straightening map satisfying $ \phi^{-1} \circ H(x,y,z) = x $. \\
Denote $ (x,y,z) $ by the point $ w $ in the three dimensional region.
\nin Let $ \UU_{\,V} $ be the space of unimodal maps defined on the set $V$. 
In this paper, every H\'enon-like map is {\em analytic} H\'enon-like map. H\'enon-like map means three dimensional H\'enon-like map unless any confusion appears. 

\ssk
\begin{prop} \label{preconv}
Let $F(w) = (f(x) - \eps(w), \  x, \ \de(w)) $ be a H\'enon-like map and $H$ be the horizontal-like diffeomorphism in \eqref{horizontal-like diffeo}. Suppose that $\| \:\! \eps \|_{C^2} \leq C\bar\eps$ and $\| \:\! \de\|_{C^2} \leq C\bar\de$ 
for some $ C>0 $. 
Then there exists a unimodal map $f_1 \in \UU_{\,V}$ such that $\| f_1 -f^2 \|_V < C\bar\eps$ and the map $H \circ F^2 \circ H^{-1}$ is the H\'enon-like map $(x,y,z) \mapsto (f_1(x) - \eps_1(x,y,z), \ x, \ \de_1(x,y,z))$ on $ V \times {\bf I}^v $ with the norm, $\| \:\! \eps_1 \| = O(\bar\eps^2 + \bar\eps \bar\de) $ and $\| \:\! \de_1 \| = O(\bar\eps \bar\de + \bar\de^2)$.
\end{prop}

\begin{proof}

By the straightforward calculation, we obtain the expression of $H \circ F^2 \circ H^{-1} $ as follows \ssk
\begin{equation*}
\begin{aligned}
& \ (x,y,z) \\
\mapsto & \ (f(f(x) - \eps \circ F \circ H^{-1}(w)) - \eps \circ F^2 \circ H^{-1}(w),\ x,\ \de \circ F \circ H^{-1}(w) - \de(x, f^{-1}(x),0)) .
\end{aligned} \msk
\end{equation*} 
Thus the first coordinate function of $H \circ F^2 \circ H^{-1}$ is 
$$ f(f(x) - \eps \circ F \circ H^{-1}(w)) - \eps \circ F^2 \circ H^{-1}(w) . $$

\nin Denote $ v(x) $ by $ \eps(x, f^{-1}(x),0) $. Thus $ v \circ f(x) = \eps(f(x),x,0) $. By the linear approximation, we obtain \ssk
\begin{equation*}
\begin{aligned}
& f(f(x) - \eps \circ F \circ H^{-1}(w)) - \eps \circ F^2 \circ H^{-1}(w) \\[0.3em]
= \ & f^2(x) - f'(f(x)) \cdot \eps \circ F \circ H^{-1}(w) - \big[\,\eps(f(x),\,x,\, 0) + \di_x \eps \circ (f(x),\,x,\, 0) \cdot \eps \circ F \circ H^{-1}(w) \\
\quad & + \di_z \eps \circ (f(x),\,x,\, 0) \cdot \de \circ F \circ H^{-1}(w)\,\big] + h.\,o.\,t. \\[0.3em]
= \ & f^2(x) - v \circ f(x) - [\,f'(f(x)) - \di_x \eps \circ (f(x),\,x,\, 0)\, ] \cdot v(x) \\
\quad & - \big[\,f'(f(x)) - \di_x \eps \circ (f(x),\,x,\, 0)\,\big] \cdot \big[\,\di_y \eps \circ (x, f^{-1}(x), 0) \cdot (f^{-1})'(x) \cdot \eps \circ H^{-1}(w) \\
\qquad & + \di_z \eps \circ (x, f^{-1}(x), 0) \cdot \de \circ H^{-1}(w) \,\big]
  - \di_z \eps \circ (f(x),\,x,\, 0) \cdot \de \circ F \circ H^{-1}(w) + h.\,o.\,t. 
\end{aligned} \msk
\end{equation*}
Let us choose the unimodal map of the first component of $H \circ F^2 \circ H^{-1}$, say $ f_1(x)$ as follows
$$ f^2(x) - v \circ f(x) - \big[\,f'(f(x)) - \di_x \eps \circ (f(x),\,x,\, 0)\,\big] \cdot v(x) . $$ 
Then $ \| f_1(x) - f^2(x) \| = O(\| \:\!\eps \|)$ and the norm of $ \eps_1(w) $ is $ O \big(\| \:\!\eps \|^2 + \| \:\! \eps \| \, \| \:\! \de \| \,\big) $. 
Let us estimate the third coordinate of $H \circ F^2 \circ H^{-1}$. 
\msk
\begin{equation*}
\begin{aligned}
& \de \circ F \circ H^{-1}(w) - \de(x,\,f^{-1}(x),\,0) \\
= \ & \de \big(x,\ \phi^{-1}(w),\ \de \circ H^{-1}(w) \big) - \de(x,\,f^{-1}(x),\,0) \\
= \ & \di_y \de \circ (x,\,f^{-1}(x),\,0) \cdot (\phi^{-1}(w) - f^{-1}(x)) + \di_z \de \circ (x,\,f^{-1}(x),\,0) \cdot \de \circ H^{-1}(w) +  h.\,o.\,t. \\
= \ & \di_y \de \circ (x,\,f^{-1}(x),\,0) \cdot (f^{-1})'(x) \cdot \eps \circ H^{-1}(w) + \di_z \de \circ (x,\,f^{-1}(x),\,0) \cdot \de \circ H^{-1}(w) +  h.\,o.\,t. 
\end{aligned} \msk
\end{equation*}
Then $ \| \:\!  \de_1\| $ is $ O(\| \:\!  \eps \| \, \| \:\!  \de \| + \| \:\!  \de \|^2) $.
\end{proof}

\nin Define {\em pre-renormalization} of $F$ as $ PRF \equiv H \circ F^2 \circ H^{-1}$ on $ H(B^1_v) $. 
Since $ H(B^1_v) $ is the cubic region, the domain $ B $ is recovered as the image of $ H(B^1_v) $ under the appropriate linear expanding map $\La(x,y,z) = (sx, sy, sz)$ for some $s<-1$, that is, $ \Dom(PRF) $ is $\La^{-1}(B)$.

\comm{***************
\ssk \\
Notify that the map $ H^{-1} $ from $\La^{-1}(B)$ to $B^1_v$ preserves each plane parallel to $ xz- $plane and the map $ F \circ  H^{-1} $ from $\La^{-1}(B)$ to $B^1_c$ preserves each plane parallel to $ yz- $plane. Let $B^1_c$ be 
$ F(B^1_v) $. 
\begin{equation*}
 \begin{aligned}
 H^{-1} : & \ \La^{-1}(B) \lra B^1_v ,  & (x,y,z) & \mapsto (\phi^{-1}(x,y,z) ,\ y,\ z + \de(y, f^{-1}(y),0)) & \\
 F \circ  H^{-1} : & \ \La^{-1}(B) \lra B^1_c,  &  (x,y,z) &\mapsto  (x,\ \phi^{-1}(x,y,z),\ \de \circ H^{-1}(w)) &
 \end{aligned} \ssk
 \end{equation*}
***********************}

\msk
\begin{defn}[Renormalization]
Let  $V$ be the (minimal) closed subinterval of $I^x$ such that  $V \times {\bf I}^v$ is invariant under $ H \circ F^2 \circ H^{-1}$ and let $s \colon V \ra I $ be the orientation reversing affine rescaling. With the rescaling map $\La(x,y,z) = (sx, sy, sz)$, {\it Renormalization} of three dimensional H\'enon-like map is defined on the domain $ B \equiv I^x \times \bf{I}^v $ as follows
\begin{align*}
RF = \La \circ H \circ F^2 \circ H^{-1} \circ \La ^{-1} .
\end{align*}
If 
$F$ is $n$ times renormalizable, then the $n^{th}$ renormalization is defined successively
\begin{align*}
 R^{n}F =  \La_{n-1} \circ H_{n-1} \circ (R^{n-1}F)^2 \circ H^{-1}_{n-1} \circ \La_{n-1} ^{-1} .
\end{align*}
where $ R^{n-1}F $ is $(n-1)^{th}$ renormalization of $F$ for $n \geq 1 $.
\end{defn}

\nin 
Denote the set of infinitely renormalizable H\'enon-like maps on the region $ B $ by $ \II_B $. In particular, if all H\'enon-like maps in $ \II_B $ satisfies that $ \max \{ \| \eps \|, \| \de \| \} \leq \bar \eps $, then this set is denoted to be $ \II_B(\bar \eps) $. The set of infinitely renormalizable unimodal maps on the interval $ I^x $ is expressed as $ \II_{I^x} $.

\ssk
\begin{lem} \label{exponential convergence to 1d map} 
Let $F \in \II_B(\bar \eps) $ 
with small enough $ \| \eps\| $ and $ \| \de \| $ bounded by $ C\bar \de$ for some $ C>0 $. 
Then for all sufficiently big $ n \geq 1 $, $ R^nF $ converge to the degenerate map $F_* = (f_*(x), x, 0)$ exponentially fast as $ n \ra \infty $ where $ f_* $ is the fixed point of one dimensional renormalization operator. 
\end{lem}

\begin{proof}
Let the degenerate map 
be $ F_{f_N} = (f_N ,\, x,\, 0) $ where $ R^NF = (f_N - \eps_N,\, x,\, \de_N) $ and let $F_{R^N_cf} = (R^N_cf ,\, x,\, 0)$ where $ R^N_cf $ is the $ N^{th} $ renormalized map of $ f $ for $ N \geq 1$. Then for big enough $N$, we obtain the following estimation
\begin{align*}
\| R^NF - F_* \| \leq \ &  \| R^NF - F_{f_N} \| + \|  F_{f_N} - F_{R^N_cf} \| + \| F_{R^N_cf} - F_* \|  \\
= \ & \| (\eps_N,\,0,\,\de_N) \| + \| f_N - R_c^Nf \| + \|  R_c^Nf - f_* \|   \\
\leq \ & C_2 \big( \bar \eps + \bar \de \big)^{2^N} + \| f_N - R_c^Nf \| + C_0 \rho_0
\end{align*}
for some $ C_0,\, C_2 >0 $ and $ 0 < \rho_0 < 1 $. From the theory of renormalization of unimodal maps, $ R_c^{Nk}f $ converges to $ f_* $ exponentially fast as $ k \ra \infty $ for sufficiently large $N$. Using the adapted metric in \cite{PS}, we can take $ N = 1 $. Then for every $ n \geq 1 $, we obtain 
\begin{align*}
\| R^nF - F_* \| \leq C_2 \big( \bar \eps + \bar \de \big)^{2^n} + \| f_n - R_c^nf \| + C_0 \rho_0^n
\end{align*}
for some $ C_0,\, C_1 >0 $ and $ 0 < \rho_0 < 1 $. 
Moreover,
\begin{align*}
\| f_n - R_c^nf \|  \leq &\, \| f_n - R_c f_{n-1} \| + \| R_c f_{n-1} - R_c^2 f_{n-2} \| +  \| R_c^2 f_{n-2} - R_c^3 f_{n-3} \| + \cdots \\
 & +  \| R_c^{m-1} f_{n-m+1} - R_c^m f_{n-m} \| + \| R_c^{m} f_{n-m} - R_c^{m+1} f_{n-m-1} \| + \cdots \\
 & + \| R_c^{n-1} f_1 - R_c^n f \| .
\end{align*}
For sufficiently large $ m $ and $ n-m $, by Lemma 8 in \cite{dMP} on the space of quadratic-like maps, we have $ C_0 $ distance contraction. Moreover, by Main Theorem on \cite{AMdM}, we obtain $ C^r $ contractions for $ r \geq 3 $ .\footnote{The theorems of on \cite{dMP} and \cite{AMdM} assumed that the maps are infinitely renormalizable with bounded combinatorics. On \cite{AMdM}, infinitely renormalized unimodal maps have the same bounded type. However, we assume that every renormalizable maps have the type of {\em period doubling} on this article. This fixed and bounded single combinatorics is much simpler than the actual hypothesis on \cite{dMP} or \cite{AMdM}.}
\begin{align} \label{bounds of sum eps 1}
& \| R_c^{m} f_{n-m} - R_c^{m+1} f_{n-m-1} \| + \cdots + \| R_c^{n-1} f_1 - R_c^n f \| 
\leq C_m \rho_m^{n-m} + \cdots + C_n \rho_n^n
\end{align} 
for some $ 0 < C_i = O(\bar \eps^{2^i}) $ and $ 0< \rho_i < 1$ where $ i =m, m+1, \ldots , n $. Every $ C_i $ and $ \rho_i $ are independent of $ n $. Thus the sum \eqref{bounds of sum eps 1} is bounded above by $ C_1 \rho_1^{n-m} $ for some $ C_1 >0 $ and $ 0 < \rho_1 <1 $. Moreover, by the direct calculations of each terms, we obtain
\begin{equation} \label{bounds of sum eps 2}
\begin{aligned} 
& \| f_n - R_c f_{n-1} \| + \| R_c f_{n-1} - R_c^2 f_{n-2} \| 
 + \cdots + \| R_c^{m-1} f_{n-m+1} - R_c^m f_{n-m} \| \\
& \ \ \leq C_n \bar \eps^{2^{n-1}} + C_{n-1}^2 \bar \eps^{2^{n-2}} + \cdots + C_{n-m}^m \bar \eps^{2^{n-m}}
\end{aligned}
\end{equation}
for some $ 0 < C_i $, $ i = n-m, \ldots , n $. For sufficiently big $ n-m $, the sum \eqref{bounds of sum eps 2} is $ O(\bar \eps_0^{2^{n-m}}) $ for $ \bar \eps_0 < \bar \eps $.
 Then $ \| f_n - R_c^nf \| \leq  C_1 \rho_1^{n-m} + O(\bar \eps_0^{2^{n-m}}) $. Hence,
\begin{align*}
\| R^nF - F_* \| \leq C_2 \big( \bar \eps + \bar \de \big)^{2^n} + C_1 \rho_1^{n-m} + O(\bar \eps_0^{2^{n-m}}) + C_0 \rho_0^n \leq C\rho^n
\end{align*}
for some $ C>0 $ and $ 0<\rho <1 $. Therefore, $R^nF$ converges to  $F_*$ exponentially fast.
\end{proof}

\nin On the following sections, we suppress the bound of small norms of $\eps $ and $ \de $ to be $ \bar \eps $, that is, we express $ \bar \eps = \max \{\bar \eps, \bar \de \} $. 
\msk

\subsection{Hyperbolicity of renormalization operator}
Hyperbolicity of renormalization operator at its fixed point was proved by Lyubich in \cite{Lyu} using quadratic-like maps. 
Moreover, Lyubich proved that renormalization operator at the fixed point has one dimensional unstable manifold 
 in the complex sense. The same thing is true in the real sense by Theorem 2.4 and Theorem 3.9 in \cite{dFdMP}. 
 \footnote{Lyubich pointed out the uniform norm of analytic operator bounds the norm of all derivatives of the same operator.}
\ssk \\
The renormalization operator of one dimensional maps is embedded under the natural inclusion to the renormalization operator of degenerate maps in the space of renormalizable H\'enon-like maps. Moreover, since the space of one dimensional infinitely renormalizable maps is a closed subset of the space of renormalizable H\'enon-like maps, the quotient space $ \II_B(\bar \eps) / \II_{I^x} $ is defined with the quotient norm. 
Thus the fact that super exponential convergence of both $ \| \:\!\eps_n \|$ and $ \| \de_n \| $ is $ O(\bar \eps^{2^n}) $ by Proposition \ref{preconv} implies that the renormalization fixed point is the map $ (x,y,z) \ra (f_*(x), x, 0) $. Furthermore, it also implies vanishing spectrum of the quotient space at the renormalization fixed point and hyperbolicity of the renormalization operator of two and three dimensional H\'enon-like maps. See Section 4 in \cite{CLM}. 

\bsk

\section{Critical Cantor set} \label{critical cantor set}
We study the minimal attracting Cantor set for infinitely renormalizable H\'enon-like map $ F $. $ F $ acts as a dyadic adding machine on this Cantor set. Topological construction of the invariant Cantor set of three dimensional H\'enon-like map is exactly the same as that of two-dimensional one (Corollary \ref{dyadic adding machine} below). 
Thus we use the same definitions and notions of the two dimensional case in this section for the sake of completeness. 
\subsection{Branches} \label{branch}
Let $\Psi^1_v = \psi^1_v \equiv H^{-1} \circ \La^{-1} $ be the non linear scaling map which conjugates $F^2$ to $RF$ on $\Psi^1_v (B)$, 
and let $\Psi^1_c = \psi^1_c \equiv F \circ \psi_v$. The subscripts $v$ and $c$ are associated to the {\it critical value} and the {\it critical point} respectively.
Similarly, let $\psi^2_v$ and $\psi^2_c$ be the non linear scaling maps conjugating $RF$ to $R^2F$. Let
$$
 \Psi^2_{vv} = \psi^1_v \circ \psi^2_v, \ \  \Psi^2_{cv} = \psi^1_c \circ \psi^2_v, \ \  \Psi^2_{vc} = \psi^1_v \circ \psi^2_c,  \ldots
$$
Successively, we can define the non linear scaling map of the $n^{th}$ level for any $n \in \N$ as follows 
$$
 \Psi^n_{\bf w} = \psi^1_{w_1} \circ \cdots \circ \psi^n_{w_n},  \ \ {\bf w}=(w_1, \ldots , w_n) \in \lbrace v, c \rbrace^n 
$$
where ${\bf w}=(w_1, \ldots , w_n)$ is the word of length $n$ and $W^n = \lbrace v, c \rbrace^n$ is the $n$-fold Cartesian product of $\lbrace v, c \rbrace$. 
\msk
\begin{lem} \label{diameter}
Let $F \in \II_B(\bar{\eps})$ for $ n \geq 1$ with small enough $ \bar \eps > 0 $. The derivative of $\Psi^n_{\bf w}$ is exponentially shrinking for $ n \in \N $ with $ \si $, that is, $\| D\Psi^n_{\bf w} \| \leq C \si^n$ for every words ${\bf w} \in W^n$ for some $C>0$ depending only on $B$ and $\bar{\eps}$.
\end{lem}

\begin{proof}
By the definition of the first coordinate map of $ H^{-1} $ and the trivial identity, $ H \circ H^{-1} = \id $, we have
$$ \phi^{-1}(w) = f^{-1}(x + \eps \circ H(w)) $$

\nin 
Thus 
\comm{************
each partial derivatives of $ \phi^{-1} $ is as follows \msk
\begin{equation*}
\begin{aligned}
\di_x \phi^{-1}(w) &= \ (f^{-1})'(x + \eps \circ H^{-1}(w)) \cdot \big[\, 1 + \di_x \eps \circ H^{-1}(w) \cdot \di_x \phi^{-1}(w) \,\big] \\[0.4em]
\di_y \phi^{-1}(w) &= \ (f^{-1})'(x + \eps \circ H^{-1}(w)) \\[-0.3em]
& \quad \cdot \big[\, \di_x \eps \circ H^{-1}(w) \cdot \di_y \phi^{-1}(w) + \di_y \eps \circ H^{-1}(w) + \di_z \eps \circ H^{-1}(w) \cdot \frac{d}{dy}\, \de(y,f^{-1}(y),0) \,\big] \\[0.2em]
\di_z \phi^{-1}(w) &= \ (f^{-1})'(x + \eps \circ H^{-1}(w)) \cdot \big[\, \di_x \eps \circ H^{-1}(w) \cdot \di_z \phi^{-1}(w) + \di_z \eps \circ H^{-1}(w) \,\big] .
\end{aligned} \msk
\end{equation*}
******}
we obtain
\begin{equation} \label{partial derivatives of phi-1}
\begin{aligned}
\di_x \phi^{-1}(w) &= \ \frac{(f^{-1})'(x + \eps \circ H^{-1}(w))}{1 - (f^{-1})'(x + \eps \circ H^{-1}(w)) \cdot \di_x \eps \circ H^{-1}(w)} \\
\di_y \phi^{-1}(w) 
&= \ \di_x \phi^{-1}(w) \cdot \big[\,\di_y \eps \circ H^{-1}(w) + \di_z \eps \circ H^{-1}(w) \cdot \frac{d}{dy}\, \de(y,f^{-1}(y),0)\,\big] \\
\di_z \phi^{-1}(w) 
&= \ \di_x \phi^{-1}(w) \cdot \di_z \eps \circ H^{-1}(w) \ .
\end{aligned} \bsk
\end{equation}

\nin Let us estimate $ \| \:\! D \phi^{-1} \| $. The above equation \eqref{partial derivatives of phi-1} implies that $ \| \:\! \di_x \phi^{-1} \| \asymp \| (f^{-1})' \| $ \ and furthermore, $  \phi^{-1}(x,y_0,z_0) \asymp f^{-1}(x) $ for every $ (y_0,z_0) \in {\bf I}^v $. The fact that $ \phi^{-1} \colon \La^{-1}(B) \ra \pi_x(B^1_v) $ implies that the domain of $ f^{-1} $ is $ \pi_x(\La^{-1}(B)) $. Then $ \| (f^{-1})' \| $ is away from one and $ \| \:\! \di_y \phi^{-1} \| $ and $ \| \:\! \di_z \phi^{-1} \| $ are $ O(\bar \eps) $ by the equation \eqref{partial derivatives of phi-1}. 
\ssk \\
The norm of derivatives of $ \phi^{-1}_n(w) $ for each $ n $ has the same upper bounds because $ f_n \ra f_* $ exponentially fast and $ \| \:\! \di_y \phi^{-1}_n \| $ and $ \| \:\! \di_z \phi^{-1}_n \| $ is bounded by $ O(\bar \eps^{2^n}) $. The dilation, $ \La^{-1} $ contracts by the factor $ \si(1 + O(\rho^n)) $ where $ \rho = \dist(F, F_*) $. 
The above equation implies that $ \| DH^{-1}_n \| $ and $ \| D(F\circ H^{-1}_n) \| $ are uniformly bounded and the upper bounds are independent of $ n $. Thus $ \| \psi^n_w \| \leq C\si $ for $ w =v, c $. Hence, the composition of $ \psi^k_w $ for $ k = 1,2,\ldots, n $ contracts by the factor $ \si^n $, that is, $ \| D\Psi_{\bf w} \| \leq C\si^n $ for some $ C>0 $.
\end{proof}
\msk

\subsection{Pieces}
Definitions on Section \ref{def of renormal} implies that $ B_v^1(F) = \psi^1_v(B) $ and $ B_c^1(F) = F \circ \psi^1_v(B) $. Observe that $F(B_c^1) \subset B_v^1 $. 
Similarly, define $ B_v^1(R^nF) $ and $ B_c^1(R^nF) $ as $ \psi^{n+1}_v(B) $ and $F_n \circ \psi^{n+1}_v(B) $ respectively for each $n \geq 1$. Observe that the piece $B^1_c(F_*)$ is a part of the parabolic-like curve of $ \{ \, (x,y) \; | \; x=f_*(y) \,\} $ and $B^1_v(F_*)$ is a rectangular box. \ssk \\ 
Let us call the set $B^n_{\bf w} \equiv B^n_{\bf w}(F) = \Psi^n_{\bf w}(B)$ the {\em pieces} of $n^{th}$ level or $n^{th}$ {generation} where ${\bf w} \in W^n$. 
Moreover, $W^n$ can be a additive group under the following correspondence from $W^n$ to the numbers with base 2 of  mod $2^n$
$$ 
{\bf w} \mapsto \sum^{n-1}_{k=0} w_{k+1} 2^k \qquad  (\mod 2^n)
$$
where the symbols $v$ and $c$ are corresponding to $0$ and $1$ respectively.  Let $P \colon W^n \ra W^n$ be the operation of adding 1 in this group. Lemma 5.3 in \cite{CLM} involves the following lemma in three dimension. 

\begin{lem} \label{pieces}
\begin{enumerate}
\item[{\rm (1)}] The pieces for the above maps are nested :
$$
B^n_{{\bf w} \nu} \subset B^{n-1}_{\bf w}, \quad {\bf w} \in W^{n-1}, \ \nu \in W.
$$
\item[{\rm (2)}] The pieces $B^n_{\bf w}, \ {\bf w} \in W$ 
are pairwise disjoint. \msk

\item[{\rm (3)}] the pieces under $F$ are permuted as follows. $F(B^n_{\bf w}) = B^n_{P({\bf w})}$ unless $P({\bf w}) = v^n$. If $P({\bf w}) = v^n$, then $F(B^n_{\bf w}) \subset B^n_{v^n}$.
\end{enumerate}
\end{lem}


\begin{figure}[htbp]
\begin{center}
\psfrag{Psiv}[c][c][0.7][0]{\Large $\psi^1_v $}
\psfrag{Psi2}[c][c][0.7][0]{\Large $\psi^2_v $}
\psfrag{Psin}[c][c][0.7][0]{\Large $\psi^n_v $}
\psfrag{RnF}[c][c][0.7][0]{\Large $R^nF $}
\psfrag{F}[c][c][0.7][0]{\Large $F$}
\psfrag{RF}[c][c][0.7][0]{\Large $RF$}
\includegraphics[scale=0.65]{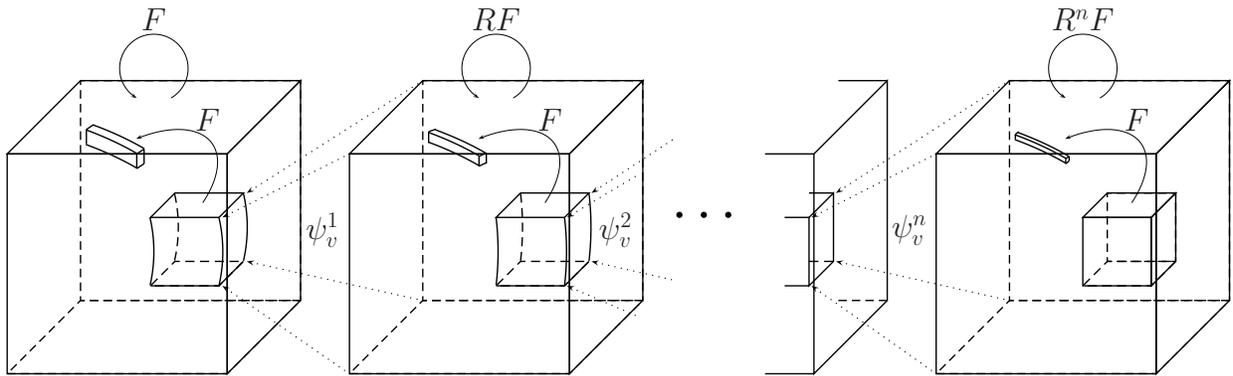}
\end{center}
\caption{Coordinate change map, $\psi^n_{v}$ at each level}
\label{fig:coordiren}
\end{figure}

\comm{*************
\nin Then the following diagram is commutative. 
\begin{displaymath}
\xymatrix @C=1.5cm @R=1.5cm
 {
B  \ar[d]_*+ {\Psi^n_{\bf w}}
  \ar[r]^*+{R^nF} & B \ar[d]^*+{\Psi^n_{\bf w}}   \\
 \Psi^n_{\bf w}(B) \ar[r]^*+{F^{2^n}}           & \Psi^n_{\bf w}(B)  }
\end{displaymath}
***********}
\nin Furthermore, Lemma~\ref{diameter} implies the following corollary. Recall that dyadic numbers and dyadic adding machine corresponds to {\em period doubling} renormalization. 

\begin{cor} \label{diameter 2}
The diameter of each piece shrinks exponentially fast for each $ n \geq 1 $, that is, $\diam (B_{\bf w}^n) \leq C\si^n$ for all ${\bf w} \in W^n$ where the constant $C > 0$ depends only on $B$ and $\bar \eps$.
\end{cor}
\nin Define the minimal invariant Cantor set of infinitely renormalizable H\'enon-like map $F$ as follows
$$
\OO \equiv \OO_F = \bigcap_{n=1}^{\infty} \bigcup_{{\bf w} \in W^n} B^n_{\bf w}.
$$
Since each $\Psi^n_{\bf w}$ is a diffeomorphism from $B^n_{\bf w}$ to its image, passing the limit with the result of Lemma \ref{pieces}, the Cantor set $ \OO_F $ is invariant under $F$. 
Let us consider the inverse limit of $W^n$, say $\displaystyle {W^{\infty} \equiv \lim_{\longleftarrow}W^n} $. The element of this set is the infinite sequence $(w_1 w_2 \ldots)$ of symbols. Thus $ W^{\infty} $ is the set of formal power series of dyadic numbers when $v$ and $c$ corresponds to 0 and 1 respectively.
$$
{\bf w} \mapsto \sum^{\infty}_{k=0} w_{k+1}2^k
$$
Then $W^{\infty}$ is {\it dyadic} group and it is also Cantor set with topology induced by the following metric
$$
\sum^{\infty}_{i=0} \frac{| \;\! v_i - w_i |}{2^i}
$$
where $ v_i $ and $ w_i $ are $ i^{th} $ letters of $ (v_1v_2v_3 \ldots)$ and $ (w_1w_2w_3 \ldots)$ respectively for every $ i \in \N $. For the detailed construction of dyadic group as Cantor set, see ~\cite{BB}. With this Cantor set, the {\it adding machine} $P \colon W^{\infty} \ra W^{\infty}$ is the operation of adding 1 in this group. Non negative integers with base 2 are embedded as the set of finite numbers in the dyadic group. Moreover, $F$ acts on the critical Cantor set as an adding machine of dyadic group by the following Corollary.

\begin{cor} \label{dyadic adding machine}
The map $F|_{\;\!\OO} $ is topologically conjugate to the dyadic adding machine $P 
$ on $ W^{\infty} $. The conjugacy is the following homeomorphism $h \colon W^{\infty} \lra \OO$
$$
h \colon {\bf w} = (w_1\,w_2 \ldots) \mapsto \bigcap ^{\infty}_{n=1} B^n_{w_1 \ldots\, w_n} .
$$
Furthermore, there exists the unique invariant probability measure $\mu$ of which support is the Cantor set $\OO$.
\end{cor}

\begin{proof}
Consider the following diagram.
\begin{displaymath}
\xymatrix @C=1.5cm @R=1.5cm
 {
W^{\infty}  \ar[d]_*+ {h}
  \ar[r]^*+{P} & W^{\infty} \ar[d]^*+{h}   \\
  \OO \ar[r]^*+{F}          & \OO  }
\end{displaymath}
\ssk \\
Take a word ${\bf w} \in W$. Let ${\bf w}_i = (w_1w_2w_3 \ldots w_i)$ be the first consecutive $i$ concatenations of the word ${\bf w}= (w_1w_2w_3 \ldots)$. Then by Lemma \ref{pieces}, $F(B^i_{{\bf w}_i}) = B^{i}_{{\bf w}_i +1}$ if $ {\bf w}_i \neq v^n$. Otherwise, $F(B^i_{{\bf w}_i}) \subset  B^{i}_{{\bf w}_i +1}$. Each domain $B^i_{{\bf w}_i}$ shrinks to a point of $\OO_F$ when $ i \ra \infty$. Then passing the limit
$$
F \left(\, \bigcap_{i=1}^{\infty} B^i_{{\bf w}_i} \right) = \bigcap_{i=1}^{\infty} B^{i}_{{\bf w}_i +1}
$$
where $ {\bf w}_i +1 $ is the image of $ {\bf w}_i $ under the adding machine with finite length $ (\mod 2^i) $. Thus $F(h({\bf w})) = h ({\bf w}+1)$, that is, the above diagram is commutative. If two words ${\bf v}$ and ${\bf w}$ have the different $i^{th}$ letter but not before, then $B^i_{{\bf v}_i}$ and $B^i_{{\bf w}_i}$ are disjoint from each other. Moreover, every point of $\OO$ has its word and two different points of $\OO$ have different words by construction of the critical Cantor set. Hence, $h$ is the bijection. The metric of dyadic group implies the (uniform) continuity of $h$. Furthermore, same topological structure and continuous bijection implies that $h$ is a homeomorphism between two compact spaces.
\end{proof}

\begin{rem}
The formal power series of numbers with base 2 are involved with the combinatorics of renormalization operator. If the combinatorics of renormalization is not period doubling but $p$-tupling by a constant $ p $, then we can construct {\it p-adic} additive group of numbers with base $p$ using the same notions. Compare \cite{HLM} for the p-tupling renormalization of two dimensional H\'enon-like maps.
\end{rem}

\nin We will call the set $\OO_F$ constructed above the {\it critical Cantor set} of $F$.

\msk

\section{Average Jacobian}
\nin Let us consider the average Jacobian of infinitely renormalizable map and show that the biggest Lyapunov exponent is 0 in Theorem \ref{max exponent is 1} below. Definitions and lemmas in this section are the same as those for two dimensional H\'enon-like map in \cite{CLM}. 
\ssk \\
Denote the Jacobian determinant of $ F $ at $w$ by $ \Jac F(w) $. Thus
$$
\log \left| \frac{\Jac F(y)}{\Jac F(z)} \right| \leq C \quad \text{for any} \ \ y,z \in B
$$
by some constant $C$ which is not depending on $y$ or $z$. Moreover, Lemma ~\ref{diameter} says the diameter of the domain $B^n_{\bf w} $ converges to zero exponentially fast. 

\begin{lem}[Distortion Lemma]  \label{distortion}
There exist a constant $C$ and the positive number $\rho <1$ satisfying the following estimate
$$
\log \left| \frac{\Jac F^k(y)}{\Jac F^k(z)} \right| \leq C \rho^n \quad \textrm{for any} \ \ y,z \in B^n_{\bf w}
$$
where $k= 1, 2, 2^2, \ldots , 2^n$.
\end{lem}

\nin Existence of the unique invariant probability measure, say $\mu$, on $\OO_F$ enable us to define the average Jacobian.
$$
b_F \equiv b = \exp \int_{\OO_F} \log \Jac F \; d\mu
$$
On each level $ n $, the measure $\mu$ on $\OO_F$ satisfies that $\mu (B^n_{{\bf w}_n} \cap \OO_F) = 1/ 2^n$ for every ${\bf w}_n$ where ${\bf w}_n$ is a word of length $ n $. 
\msk
\begin{cor} \label{average}
For any piece of $B^n_{{\bf w}}$ on the level $ n $ and any point $ w \in B^n_{{\bf w}} $,
$$
\Jac F^{2^n} \!(w) = b^{2^n} (1 + O(\rho ^n))
$$
where $ b $ is the average Jacobian of $ F $ for some positive $ \rho < 1 $.
\end{cor}

\begin{proof}
Since
$$
\int _{B^n_{\bf w}} \log \Jac F^{2^n} d\mu = \int _{\OO} \log \Jac F \; d\mu = \log b,
$$
there exists a point $\eta \in B^n_{\bf w}$ such that  $ \log \Jac F^{2^n} \!(\eta) = \dfrac{\log b}{\mu(B^n_{\bf w})}  = 2^n \log b$ \\
For any $ w \in B^n_{\bf w}$, $\log \, \Jac F^{2^n} \! (z) \leq C \rho^n + \log \, \Jac F^{2^n} \!(\eta)$, and $ O(\rho^n) = \log (1 + O(\rho^n))$ for a fixed constant $\rho$. Then
\begin{align*}
\log \Jac F^{2^n} (w) &= \log (1 + O(\rho^n)) + \log \Jac F^{2^n} (\eta)\\
                            &= \log (1 + O(\rho^n)) \cdot  b^{2^n} 
\end{align*}
\textrm{Therefore}, \ $ \displaystyle{ \Jac F^{2^n} (w) = b^{2^n} (1 + O(\rho^n))} $.
\end{proof}

\nin Three Lyapunov exponents $\chi_0, \chi_1$ and $ \chi_2$ exist for three dimensional map. Let $\chi_0$ be the maximal one. 
Since $F$ is ergodic with respect to the invariant finite measure $\mu$ on the critical Cantor set, we have the following inequality.
$$
| \: \! \mu |\, \chi(x) \leq \int_{\OO_F} \log \| DF(x) \| \; d\mu(x)
$$
where $| \:\!\mu| $ is the total mass of $ \mu $ on $ \OO_F $.
\msk
\begin{thm} \label{max exponent is 1}
The maximal Lyapunov exponent of $F$ on $\OO_F$ is 0.
\end{thm}
\begin{proof}
\comm{*******************
Let $\mu_n$ be \,$2^n \mu | _{B^n_w}$, an invariant measure under $F^{2^n}$ and let $\nu_n$ be the (unique) invariant measure on $R^nF |_{\;\!\OO_{R^nF}}$. Then
$$
2^n \chi_0(F, \mu) = \chi_0(F^{2^n} |_{B^n_{v^n}}, \mu_n) = \chi_0(R^nF, \nu_n) \leq \int_{B^n_w} \log \| D(R^nF) \|\, d\nu_n \leq C 
$$
 for every $ n \in \N$, where $C$ is a constant independent of $n$. The last inequality comes from the uniformly bounded $C^1$ norm of derivative of $R^nF$. Then the maximal Lyapunov exponent $\chi_0 \leq 0$. If $\chi_0 <0$, then the support of $\mu$ contains some periodic cycles by Pesin's theory. But $\OO_F$ does not contain any periodic cycle because $ F $ acts on $ \OO_F $ as a dyadic adding machine. Therefore, $\chi_0 =0$ and the sum of the other exponents, $\chi_1  + \chi_2$, is $\log b$.
 ***********************}
See the proof of Theorem 6.3 in \cite{CLM}. 
\end{proof}
\nin Observe that $ \log b $ is the sum of Lyapunov exponents except the maximal one.
\bsk

\section{Universal expression of Jacobian determinant}  \label{universality of Jacobian}

\nin Universality of average Jacobian is involved with asymptotic behavior of the non linear scaling map $\Psi^n_{v^n}$ between the renormalized map $F_n \equiv R^nF$ and $F^{2^n}$ for each $n \in \N$. $ \Psi^n_{v^n}$ conjugate $F^{2^n}$ to $F_n$. Thus using the chain rule and Corollary \ref{average}, $ \Jac F_n $ is the product of the average Jacobian of $F^{2^n}$ and the ratio of $ \Jac \Psi^n_{v^n} $ at $ w $ and $ F_n(w) $ as follows
\begin{equation} \label{chain rule}
\begin{aligned}
\Jac F_n(w) &= \Jac F^{2^n}(\Psi^n_{v^n}(w)) \frac{\Jac \Psi^n_{v^n}(w)}{\Jac \Psi^n_{v^n}(F_n(w))}  \\
                  &= b^{2^n} \frac{\Jac \Psi^n_{v^n}(w)}{\Jac \Psi^n_{v^n}(F_n(w))} (1+ O(\rho^n)).
\end{aligned} \msk
\end{equation}
\nin Then in Theorem \ref{Universality of the Jacobian} below, universality of Jacobian of $ \Psi^n_{v^n}$ implies that of $ \Jac F_n $. The asymptotic of non-linear part of $ \Psi^n_{v^n}$ is essential to the universal expression of $ \Jac \Psi^n_{v^n} $.

\msk 
\subsection{Asymptotic of $\Psi^n_k$ for fixed $ k^{th} $ level } \label{asymptotic coordinate}
For every infinitely renormalizable H\'enon-like map $F$, we define the {\em tip}
\begin{align}
\{\tau \} \equiv \{ \tau_F \} = \bigcap_{n \ge 0} B^n_{v^n}
\end{align}
where the pieces $B^n_{v^n}$ are defined as $ \Psi^n_{v^n}(B(R^nF)) $. The tip of $ R^kF $ is denoted by $\tau_k =\tau (R^k F)$ for each $ k \in \N $. Since every $ B^n_{v^n}(F) $ contains $ \tau_F $, let us condense the notation $ \Psi^n_{v^n} $ into $ \Psi^n_{\tip} $. Moreover, in order to simplify notations and calculations, let the tip move to the origin as a fixed point of each $\Psi^1_v (R^kF)$ for every $ k \in \N $ by the conjugation of appropriate translations. Let us define $  \Psi_k^{k+1} $ in this section \footnote{If we need to distinguish the scaling maps, $ \Psi^n_k $ around tip from its composition with translations, then we use the notation, $ \Psi^n_{k,\, \tip} $. }
\ssk
\begin{align} \label{origin as tip}
    \Psi_k(w) \equiv \Psi_k^{k+1}(w) = \Psi^{k+1}_v (w + \tau_{k+1}) - \tau_k 
\end{align}
where $w=(x, y, z)$. 
\nin Denote the derivative of $ \Psi_k $ 
at the origin by $D_k \equiv D_k^{k+1}$.
\begin{equation*}
\begin{aligned}   
  \quad D^{k+1}_k \equiv D_k = D\Psi^{k+1}_k(0) &= D(\Psi^1_v(R^k F))(\tau_{k+1} ) & \\
       &= D(T_k \circ \Psi^1_v (R^kF) \circ T^{-1}_{k+1})(0) &
\end{aligned} \ssk
\end{equation*}       
where \ $T_j : w \mapsto  w-\tau_j$  for $ j= k,\, k+1$.
Then we can decompose $ D_k $ into the matrix of which diagonal entries are ones and  the diagonal matrix. 
\ssk
\begin{align}   \label{decomposition}
 D_k= 
 \left(    \begin{array}{ccc}
      1 & t_k & u_k \\
        & 1 & \\
        & d_k & 1
    \end{array}      \right)
 \left(    \begin{array}{ccc}
     \alpha_k &  & \\
              & \si_k & \\
              &  & \si_k
    \end{array}       \right) 
=
 \left(    \begin{array} {ccc}
     \alpha_k & \si_k \:\! t_k & \si_k \:\!u_k \\
              & \si_k &  \\
              & \si_k \:\! d_k &\si_k
    \end{array}       \right)         
\end{align} 

\msk
\nin Recall that $\si_k = - \si \left(1+O(\rho^k) \right)$.
Moreover, we can express $\Psi^{k+1}_k$ with the linear and non-linear parts.
\begin{align} \label{Psi on k level}
\Psi^{k+1}_k \equiv \Psi_k(w) = D_k \circ (\id+ {\bf s}_k)(w)
\end{align}
where ${\bf s}_k (w)= (s_k(w),\; 0,\; r_k(y)) $ and $ \|{\bf s}_k (w) \|= O(\| w \|^{\:\!2})$\ near the origin. 
Comparing the derivative of $ H^{-1}_k \circ \La^{-1}_k $ at the tip and $D_k$ and by Corollary \ref{diameter 2}, we obtain the following estimations
\begin{equation} \label{bounds at 0}
\begin{aligned}
t_k &= \di_y \phi_k^{-1} (\tau_{k+1}) =  \di_x \phi_k^{-1}(\tau_{k+1}) \cdot \big[\, \di_y \eps_k (\tau_k) +  \di_z \phi_k^{-1}(\tau_{k+1}) \cdot d_k \,\big] \\[0.2em]
u_k &= \di_z \phi_k^{-1} (\tau_{k+1}) = \di_x \phi_k^{-1}(\tau_{k+1}) \cdot \di_z \eps_k(\tau_k) \\
 \textrm{and}\quad d_k &= \frac{d}{dy}\,\de_k \left(\pi_y(\tau_{k+1}), f^{-1}_k(\pi_y(\tau_{k+1})),0 \right) 
\end{aligned}
\end{equation}
where $ \phi^{-1}_k(w) = \pi_x \circ H^{-1}_k(w) $. Since the norm, $ \| \:\! \di_x\phi^{-1}_k(\tau_k) \| $ 
exponentially converges to $\si$ as $k \ra \infty$, we have the estimation, $\alpha_k = \si^2 \left(1+O(\rho^k) \right)$ for some $\rho \in (0,1)$. The above constants $ | \;\! t_k |, | \:\! u_k | $ and $ | \:\! d_k | $ are $ O \big(\bar \eps^{2^k} \big) $ because $ \| \:\! \eps \|_{C_1} $ and $ \| \:\! \de \|_{C^1}  $ are $ O \big(\bar \eps^{2^k} \big) $.  

\msk 
\begin{lem} \label{bounds on domain}
Let $s_k$ be the function defined on \eqref{Psi on k level}. For each $k \in \N$, \msk
\begin{enumerate}
\item
$ | \, \partial_xs_k| \; = O(1), \qquad \ \ | \, \partial_y s_k|  \ = O(\bar \eps^{2^k}), \qquad | \, \partial_z s_k| \ = O(\bar \eps^{2^k})  $ \ssk
  \item 
$ | \, \partial^2_{xx} s_k| = O(1), \qquad  \ \ \! | \, \partial^2_{xy} s_k| = O(\bar \eps^{2^k}), \qquad | \, \partial^2_{yy} s_k| = O(\bar \eps^{2^k}) $ \ssk
 \item
$ | \, \partial^2_{yz} s_k| = O(\bar \eps^{2^k}), \qquad | \, \partial^2_{zx}s_k| = O(\bar \eps^{2^k}), \qquad | \, \partial^2_{zz} s_k| = O(\bar \eps^{2^k}) $ \ssk
 \item
$  | \, r_k(y)| = O(\bar \eps^{2^k}), \qquad   | \, r'_k(y)| = O(\bar \eps^{2^k}), \qquad | \, r_k''(y)| = O(\bar \eps ^{2^k}) $ \msk
\end{enumerate}
\end{lem}

\begin{proof}
The map $\Psi_k $ has the two expressions, $D_k \circ (\id+ {\bf s}_k)$ and $T_k \circ H^{-1}_k \circ \La_k \circ T_{k+1}^{-1}$, that is,
\begin{align*} 
\Psi_k &= D_k \circ (\id+ {\bf s}_k)(w) \\
&=T_k \circ H^{-1}_k \circ \La_k^{-1} \circ T_{k+1}^{-1}(w) = H^{-1}_k \circ \La_k^{-1} (w+ \tau_{k+1}) - \tau_k
\end{align*}
Let $ \tau_k = (\tau_k^x, \tau_k^y, \tau_k^z) $ for each $ k \geq 1 $.  
Firstly, let us compare the third coordinates of two expressions of $\Psi_k $. 
\begin{align*}
\si_k (d_ky + z+ r_k(y)) 
&= \pi_z \big(H^{-1}_k \circ \La_k^{-1} (w+ \tau_{k+1}) - \tau_k \big) \\
&= \si_k(z + \tau_{k+1}^z) + \de_k \big(\, \si_k (y+\tau_{k+1}^y), \ f^{-1}_k (\si_k (y+\tau_{k+1}^y)), \ 0 \big) - \tau_k^z 
\end{align*}
Thus we have the following equation
\[ \si_k  r_k(y) = - \si_k d_k \:\! y + \de_k \big( \si_k (y+\tau_{k+1}^y), \; f^{-1}_k (\si_k (y+\tau_{k+1}^y)), \; 0 \big) +\si_k \tau_{k+1}^z - \tau_k^z .
\]
Then $ | \:\! r_k(y)| \leq C \big( \,| \:\! d_ky | + \| \:\! \de_k \|_{C^0} \big)$ for some $C>0$. Since $ \Dom(\Psi_k) $ is bounded and $ \| \:\! \de_k \| $ is $ O(\bar \eps ^{2^k}) $, we have $| \;\! r_k| = O(\bar \eps ^{2^k}) $. Moreover, 
\[ r_k'(y) = -d_k + \frac{d}{dy}\,\de_k \big( \si_k (y+\tau_{k+1}^y), \; f^{-1}_k (\si_k (y+\tau_{k+1}^y), \; 0 \big)
\]
Thus $ | \:\! r_k'| $ is bounded by $ \| \:\!\de \|_{C^1} $. Similarly, the second derivative $ | \;\! r_k''| $ is also controlled by $ \| \:\! \de \|_{C^2} $. Then $| \;\! r_k'| = O(\bar \eps ^{2^k}) $ and $ | \:\! r_k''| = O(\bar \eps ^{2^k}) $.
\msk \\
Secondly, compare first coordinates using \eqref{decomposition} and \eqref{Psi on k level}. Thus
\begin{align}  \label{comparison}
\alpha_k x + \alpha_k \cdot s_k(w) + \si_k t_k \:\! y + \si_k u_k (z + r_k(y)) = \phi_k^{-1}(\si_k w + \si_k \tau_{k+1}) - \pi_x(\tau_k).
\end{align}
\ssk
It implies the following equations
\begin{equation} \label{bounds of sk}
\begin{aligned}
\alpha_k \cdot \di_x s_k &= \si_k \cdot \di_x \phi^{-1}_k - \alpha_k  \\
\alpha_k \cdot \di_y s_k &= \si_k \cdot \di_y \phi^{-1}_k - \si_k t_k -\si_k \cdot u_k r'_k(y)  \\
\alpha_k \cdot \di_z s_k &= \si_k \cdot \di_z \phi^{-1}_k - \si_k u_k .
\end{aligned}
\end{equation}
 
\nin 
Then by the equation \eqref{partial derivatives of phi-1}\ssk , $ \| \;\!\di_x \phi^{-1}_k \| = O(1)$, $ \| \;\! \di_y \phi^{-1}_k \| = O \big(\bar \eps ^{2^k} \big) $ and $ \| \;\! \di_z \phi^{-1}_k \| = O\big(\bar \eps ^{2^k}\big)$. By the equation \eqref{bounds at 0}, $ | \;\!t_k | $ and $ | \:\! u_k | $ is $ O\big(\bar \eps ^{2^k}\big) $. Hence, $ \| \:\! \partial_xs_k \| \, = O(1) $,  $  \| \:\! \partial_y s_k \| \, = O \big(\bar \eps^{2^k} \big)$ and $ \| \:\! \partial_z s_k \| \; = O \big(\bar \eps^{2^k} \big) $. By the above equation \eqref{bounds of sk}, each second partial derivatives of $s_k$ are comparable with the second partial derivatives of $\phi^{-1}_k$ over the same variables because $  | \, r_k''(y)| = O \big(\bar \eps ^{2^k} \big) $. 
\ssk \\
Let us estimate some second partial derivatives of $ \phi^{-1}_k $. Recall that 
\begin{equation*}
\begin{aligned}
\phi^{-1}_k(w) &= \ f^{-1}_k(x + \eps_k \circ H^{-1}_k(w)) \\[0.2em]
\eps_k \circ H^{-1}_k(w) &= \ \eps_k(\phi^{-1}_k(w),\,y,\,z + \de_k (y,f^{-1}(y),0)) .
\end{aligned}
\end{equation*}
Thus
\begin{equation*}
\begin{aligned}
\di_x \phi^{-1}_k(w) &= \ (f^{-1}_k)'(x + \eps_k \circ H^{-1}_k(w)) \cdot \big[\, 1 + \di_x (\eps_k \circ H^{-1}_k(w))\,\big] \\[0.4em]
\di_x( \eps_k \circ H^{-1}_k(w)) &= \ \di_x \eps_k \circ H^{-1}_k(w) \cdot \di_x \phi^{-1}_k(w) \\[0.4em]
\di_{xx}( \eps_k \circ H^{-1}_k(w)) &= \ \di_x( \eps_k \circ H^{-1}_k(w)) \cdot \di_x \phi^{-1}_k(w) + \di_x \eps_k \circ H^{-1}_k(w) \cdot \di_{xx} \phi^{-1}_k(w) \\
&= \ \di_x \eps_k \circ H^{-1}_k(w) \cdot \big[\,\di_x \phi^{-1}_k(w)\,\big]^2 + \di_x \eps_k \circ H^{-1}_k(w) \cdot \di_{xx} \phi^{-1}_k(w) .
\end{aligned} \msk
\end{equation*}
\nin Moreover, $ \| \:\! \eps_k \|_{C^2} $ and $ \| \:\! \de_k \|_{C^2} $ bounds the norm of every second derivatives of $ \| \:\! \phi^{-1}_k \| $ except $ \| \di_{xx} \phi^{-1}_k(w) \| $. Let us estimate $ \di_{xx} \phi^{-1}_k(w) $
\msk
\begin{equation*}
\begin{aligned}
\di_{xx} \phi^{-1}_k(w) &= \ (f^{-1}_k)''(x + \eps_k \circ H^{-1}_k(w)) \cdot \big[\, 1 + \di_x (\eps_k \circ H^{-1}_k(w))\,\big] \\
& \quad \ \ + (f^{-1}_k)'(x + \eps_k \circ H^{-1}_k(w)) \cdot \di_{xx} (\eps_k \circ H^{-1}_k(w)) .
\end{aligned} \ssk
\end{equation*}
\nin Recall that $ \| \:\! \eps_k \|_{C^2} $ and $ \| \:\! \de_k \|_{C^2} $ are $ O \big(\bar \eps^{2^k} \big) $. Since both $ \| (f^{-1})' \| $ and $ \| (f^{-1})'' \| $ are $ O(1) $, so is $ \| \:\! \di_{xx}\phi^{-1}_k \| $. Any other second derivative of $ \| \:\! \phi^{-1}_k \| $ is bounded by $ O\big( \bar \eps^{2^k} \big) $. For example, the following estimation,
\begin{equation*}
\begin{aligned}
\di_{yx}\phi^{-1}_k(w) &= \di_{xx}\phi^{-1}_k(w) \cdot \left[\, \di_y \eps_k \circ H^{-1}_k(w) + \di_z \eps_k \circ H^{-1}_k(w) \cdot \frac{d}{dy}\,\de_k(y,f^{-1}_k(y),0)\,\right] \\
& \quad \ \ + \di_x \phi^{-1}_k(w) \cdot \left[\, \di_x (\di_y \eps_k \circ H^{-1}_k(w)) + \di_x (\di_z \eps_k \circ H^{-1}_k(w)) \cdot \frac{d}{dy}\,\de_k(y,f^{-1}_k(y),0)\,\right]
\end{aligned}
\end{equation*}
implies that $ \| \;\! \di_{yx} \phi^{-1}_k \| $ is bounded by $ O\big( \bar \eps^{2^k} \big) $. The norm estimation of other second partial derivatives of $ \phi^{-1}_k $ is left to the reader.
\end{proof}
\msk

\subsection{The estimation of non linear part $S^n_k$ from level $k$ to the fixed level $n$}
We consider the behavior of 
non linear scaling map from $k^{th} $ level to $n^{th}$ level. 
Let
$$
\Psi^n_k = \Psi_k \circ \cdots \circ \Psi_{n-1}, \quad B^n_k = \Im \Psi^n_k
$$
By Lemma \ref{diameter}, 
$$ \diam(B^n_k) = O(\si^{n-k}) \qquad \textrm{ for} \quad k<n
$$
Then combining Lemma \ref{diameter} and Lemma \ref{bounds on domain}, we have the following corollary.

\begin{cor}  \label {bounds at k level}
 For all points $w = (x,y,z) \in B^n_{k}$ and where $k<n$, we have
\begin{align*}
& |\; \! \di_x s_k(w) | = O(\si^{n-k}) \qquad |\; \! \di_y s_k(w) | = O  \big(\bar \eps^{2^k} \si^{n-k} \big)
 \qquad |\; \! \di_z s_k(w) | = O \big( \bar \eps^{2^k} \si^{n-k} \big) \\
 &|\; \! r'_k(y)|=O \big(\bar \eps^{2^k} \si^{n-k} \big) \qquad \quad |\; \! r''_k(y)|=O \big(\bar \eps^{2^k} \si^{n-k} \big)
\end{align*}
\end{cor}
\begin{proof}
By definition, $ s_k(w) $ is quadratic and higher order terms at the tip, $ \tau_k $. Similarly, $ r'_k(y) $ only contains quadratic and higher order terms at the tip. Using Taylor's expansion and the fact that $ \diam(B^n_k) = O (\si^{n-k}) $, we obtain the result of corollary.
\end{proof}
\msk
\nin Since the origin is the fixed point of every $ \Psi_j $, 
derivative of $\Psi^n_k$ at the origin is the composition of consecutive $ D_i $s for $ k \leq i \leq n-1 $
\begin{equation*}
 D^n_k = D_k \circ D_{k+1} \circ \cdots \circ D_{n-1}. 
\end{equation*}
Moreover, we can decompose $D^n_k$ into two matrices, the matrix whose diagonal entries are ones and the diagonal matrix by reshuffling. 

\begin{rem}
The notations $t_{n+1,\,n}, u_{n+1,\,n}, d_{n+1,\,n} $ are simplified as $t_n, u_n, d_n$ like the notations used in \eqref{decomposition}. Similarly, $ \alpha_{n+1,\,n}, \si_{n+1,\,n}$ are abbreviated as $\alpha_n, \si_n$ respectively. For instance, we let $\alpha_n = \si^2(1+ O(\rho^n)),\ \si_n = -\si (1+ O(\rho^n))$. Using the similar abbreviation, $D_n$ denote $D^{n+1}_n$ and $s_n$ does $s^{n+1}_n$.
\end{rem}

\begin{lem} \label{decomposition of derivative}
The derivative of $ \Psi^n_k $ at the origin, $ D^n_k $ is decomposed into the dilation and other parts as follows
\begin{align*}
D_k^n=
\left ( \begin{array} {ccc}
1 & t_{n,\,k} & u_{n,\, k} \\[0.2em]
   & 1          & \\
   & d_{n,\, k}& 1 
\end{array} \right )
\left ( \begin{array} {ccc}
\alpha_{n,\, k} &                &  \\
                & \si_{n,\,k}    &  \\
                &                & \si_{n,\,k}
\end{array} \right ) 
\end{align*}
\msk \\
where $\alpha_{n,\,k} = \si^{2(n-k)}(1+O(\rho^k))$ and $\si_{n,\,k}= (-\si)^{n-k}(1+O(\rho^k)) $ for some $\rho \in (0,1)$. Each $ t_{n,\,k} $, $ u_{n,\,k} $ and $ d_{n,\,k} $ are comparable with $ t_{k+1,\,k} $, $ u_{k+1,\,k} $ and $ d_{k+1,\,k} $ respectively and converges to the numbers $ t_{*,\,k} $, $ u_{*,\,k} $ and $ d_{*,\,k} $ respectively super exponentially fast as $ n \ra \infty $.
\end{lem}
\begin{proof}
Using the definition of each derivatives of $ \Psi_j $ on the equation \eqref{decomposition} at the origin, we obtain 
\begin{align*}
D^n_{k} = \prod^{n-1}_{j=k} D_j = \prod^{n-1}_{j=k}
\begin{pmatrix}
\alpha_j & \si_j \;\! t_j &  \si_j \;\! u_j\, \\
              & \si_j &  \\
              & \si_j \;\! d_j \,&\si_j
\end{pmatrix} .
\end{align*}
By the straightforward calculation, we have following expressions,

\msk
\begin{equation} \label{scaling from kth to nth level}
\begin{aligned}
\si_{n,\,k} &= \prod^{n-1}_{j=k} \si_j = \prod^{n-1}_{j=k} (-\si) (1+ O(\rho^j)) = (-\si)^{n-k}  (1+ O(\rho^k))& \\
\alpha_{n,\,k} &= \prod^{n-1}_{j=k} \alpha_j = \prod^{n-1}_{j=k} \si^2 (1+ O(\rho^j)) = \si^{2(n-k)}  (1+ O(\rho^k))& .
\end{aligned}
\end{equation}
By the definition of $ d_{n,\,k} $ and \eqref{scaling from kth to nth level}, each components of the diffeomorphic part and the scaling part are separated
\begin{equation} \label{sum of d, u and t respectively}
\begin{aligned}
d_{n ,\: k} =& \sum^{n-1}_{j=k} d_j \\
u_{n,\: k} =&  \sum_{j=k}^{n-1} (-\si)^{j-k}u_j \, (1+ O(\rho^k)) \\
t_{n ,\: k} =&  \sum_{j=k}^{n-2} (-\si)^{j-k} \left[ u_j \sum_{i=j}^{n-2} d_{i+1} + t_j + t_{n-1} \right] (1+ O(\rho^k)) .
\end{aligned}
\end{equation}
Since $ | \:\! d_j | = O\big(\bar \eps^{2^j} \big) $, $ | \:\! u_j | = O\big(\bar \eps^{2^j} \big) $ and $ | \:\! t_j| = O\big(\bar \eps^{2^j} \big) $ for each $ j \in \N $, each terms of the series 
in \eqref{sum of d, u and t respectively} shrink super exponentially fast. Then the sum $ d_{n,\,k} $, $ u_{n,\,k} $ and $ t_{n,\,k} $ are comparable with the first terms of each series. Moreover, $ d_{n,\,k} $, $ u_{n,\,k} $ and $ t_{n,\,k} $ converge to some numbers $ d_{*,\,k} $, $ u_{*,\,k} $ and $ t_{*,\,k} $ as $ n \ra \infty $ super exponentially fast respectively.
\end{proof}
\msk
\nin After reshuffling $\Psi^n_k$, we can factor out $ D^n_k $ from the map $\Psi^n_k$. Then we have
\begin{align} \label{coor change from k to n}
 \Psi_k^n=D_k^n \circ (\id + {\bf S}^n_k)
\end{align}
where  ${\bf S} ^n_k=(S_k^n(w),\ 0,\ R_k^n(y)) 
$. Observe that $R_k^n $ depends only on $ y $ 
by the direct calculation of $H^{-1}_k \circ \La^{-1}_k \circ \cdots \circ H^{-1}_{n-1} \circ \La^{-1}_{n-1}$.
\msk
\begin{prop} \label{bounds of R}
The third coordinate of\, ${\bf S} ^n_k  $, $ R_k^n(y) $ has the following norm estimations
\begin{align*}
|R_k^n| = O \big(\bar \eps^{2^k} \big), \quad
| \: \!(R_k^n)'| = O\big(\bar \eps^{2^k} \si^{n-k} \big). 
\end{align*}
for all $ w \in B(R^nF) $ and for all $ k<n $.
\end{prop}

\begin{proof}
The proof is involved with the recursive formula between each partial derivatives of $S^n_k$ and $S^n_{k+1}$. Thus 
we need some intermediate calculations. 
Denote $ \Psi_{k+1}^n(w) $ by $ w_{k+1}^n = (x_{k+1}^n, y_{k+1}^n, z_{k+1}^n) \in  B^n_{k+1} $. 

\nin By the equation \eqref{coor change from k to n}, we have
\msk
\begin{equation*}
\begin{aligned}
\left ( \begin{array} {c}
x_{k+1}^n\\ [0.3em]
y_{k+1}^n\\ [0.3em]
z_{k+1}^n
\end{array} \right )=
\left ( \begin{array} {c l l}
\alpha_{n,\,k+1}     & \si_{n,\,k+1}  \cdot t_{n,\,k+1}  & \si_{n,\,k+1} \cdot u_{n,\,k+1} \\ [0.3em]
                     & \si_{n,\,k+1}              &                 \\ [0.3em]
                     & \si_{n,\,k+1} \cdot d_{n,\,k+1}  &  \si_{n,\,k+1}
\end{array} \right )
\left ( \begin{array} {c}
x+ S _{k+1}^n(w) \\ [0.3em]
y \\ [0.3em]
z+R_{k+1}^n(y)
\end{array} \right )  .            
\end{aligned} \bsk
\end{equation*}
Then each coordinate of $w^n_{k+1}$ is
\msk
\begin{equation}  \label{the image of the Psi from nth level to k+1th level}
\begin{aligned} 
x_{k+1}^n =& \ \alpha_{n,\,k+1}(x+S_{k+1}^n(w))+\si_{n,\,k+1} t_{n,\, k+1} \cdot y + \si_{n,\,k+1} u_{n,\,k+1}(z+R_{k+1}^n(y)) \\[0.2em]
y_{k+1}^n =& \ \si_{n,\,k+1} \cdot y \\[0.2em]
z_{k+1}^n =& \ \si_{n,\,k+1}\, d_{n,\,k+1} \cdot y+ \si_{n,\,k+1}(z+R_{k+1}^n(y)) . 
\end{aligned} \msk
\end{equation}
For any fixed $k<n$, the recursive formula of $ \Psi^n_k$ is
\msk
\begin{equation}      \label{recursive form}
\begin{aligned} 
D_k^n \circ (\id + {\bf S}_k^n) &= \Psi _k^n = \Psi _k \circ \Psi_{k+1}^n = D_k \circ (\id + {\bf s}_k) \circ \Psi _{k+1}^n  \\[0.2em]
 &= D_k^n \circ (\id + {\bf S}_{k+1}^n) + D_k \circ {\bf s}_k \circ \Psi_{k+1}^n  \\[0.2em]
\text {Thus} \qquad  \Psi_k^n(w) &= D_k^n \circ (\id + {\bf S}_{k+1}^n)(w) + D_k \circ {\bf s}_k (w_{k+1}^n)
\end{aligned}
\end{equation}
and note that
\begin{equation*}
 D_k \circ {\bf s}_k (w_{k+1}^n) =
\left (
       \begin{array} {crr}
    \alpha_k  &  \si_k t_k  &  \si_k u_k  \\[0.2em]
              &  \si_k &         \\[0.2em]
              &  \si_k d_k  & \si_k
       \end{array}
\right )
\left (
       \begin{array}{c}
s_k(w_{k+1}^n) \\[0.3em]
0 \\[0.1em]
r_k (y_{k+1}^n)
        \end{array}
\right ) . 
\end{equation*}
\ssk \\
\comm{**********
Moreover, the first partial derivatives of each coordinate are as follows
\msk
\begin{equation}  \label{1st partial}
\begin{aligned}
\frac {\partial x_{k+1}^n}{\partial x} &= \alpha_{n,\,k+1} \left (1+ \frac{\partial S_{k+1}^n}{\partial x}(w) \right) \\[0.3em]
\frac {\partial x_{k+1}^n}{\partial y} &= \alpha_{n,\,k+1} \frac{\partial S_{k+1}^n}{\partial y}(w) + \si_{n,\,k+1} t_{n,\,k+1} +\si_{n,\,k+1}u_{n,\,k+1} (R_{k+1}^n)'(y) \\[0.3em]
\frac {\partial x_{k+1}^n}{\partial z} &= \alpha_{n,\,k+1} \frac{\partial S_{k+1}^n}{\partial z}(w) + \si_{n,\,k+1} u_{n,\,k+1} \\[0.3em]
\frac {\partial y_{k+1}^n}{\partial y} &= \frac {\partial z_{k+1}^n}{\partial z} = \si_{n,\,k+1} \\[0.3em]
\frac {\partial z_{k+1}^n}{\partial y} &= \si_{n,\,k+1} d_{n,\,k+1}+ \si_{n,\,k+1} \cdot (R_{k+1}^n)'(y) \\[0.3em]
\frac {\partial y_{k+1}^n}{\partial x} &= \frac {\partial y_{k+1}^n}{\partial z} = \frac {\partial z_{k+1}^n}{\partial x} = 0 .
\end{aligned} \msk
\end{equation}
*******************}
In order to estimate of $R^n_k(y)$, compare the third coordinates of functions in \eqref{recursive form}. Recall $\si^{-1} = \la $. Then
\begin{equation*}
\begin{aligned}
z_{k}^n =& \ \si_{n,\,k}\, d_{n,\,k} \cdot y+ \si_{n,\,k}(z+R_{k}^n(y)) \\[0.2em]
=& \ \si_{n,\,k} \,d_{n,\,k} \cdot y+ \si_{n,\,k}(z+R_{k+1}^n(y)) +\si_k \cdot r_k (y_{k+1}^n) 
\end{aligned} 
\end{equation*}
Then
\begin{equation*}
\begin{aligned}
 R_k^n(y)  =& \ R_{k+1}^n(y) +\si_{n,\,k}^{-1} \cdot \si_k \cdot r_k(y_{k+1}^n)   
\end{aligned} \msk
\end{equation*}
where $ \si_{n,\,k}^{-1} \cdot \si_k $ is $ (-\lambda)^{n-k-1}(1+O(\rho^k)) $. \ssk
By 
the recursive relation between $R^n_k(y)$, $  R_{k+1}^n(y) $ and the bounds of $ r_k(y_{k+1}^n) $, we obtain that 
\msk
\begin{align*}
 R^n_k(y) &= R_{k+1}^n(y) + O\big( (-\lambda)^{n-k-1} r_k(y_{k+1}^n) \big), \quad  
 (R^n_k)'(y) &= (R_{k+1}^n)'(y) + O\big( r_k'(y_{k+1}^n) \big)      
\end{align*} 
Hence, by the equation \eqref{the image of the Psi from nth level to k+1th level}, we have 
\begin{equation*}
\begin{aligned}
|R_k^n| &\leq |R_{k+1}^n| + K_0 \bar \eps^{2^k}, \quad  
| \: \! (R_k^n)'| &\leq  |(R_{k+1}^n)'| + K_1 \bar \eps^{2^k}\si^{n-k} 
\end{aligned} \msk
\end{equation*}
 for all $ k<n  $. Then, 
\begin{align*}
|R_k^n| = O(\bar \eps^{2^k}), \quad | \: \!(R_k^n)'| = O(\bar \eps^{2^k} \si^{n-k} ) 
\end{align*}
 for all $ k<n  $.
\end{proof}

\msk
\begin{lem} \label{asymptotics of non-linear part 1}
For $k<n$, we have \msk
\begin{enumerate}
\item $ | \; \! \partial_x S^n_k| \; = O(1), \qquad  \qquad | \; \! \partial_y S^n_k|  \ = O(\bar \eps^{2^k}), \qquad \qquad \ | \; \! \partial_z S^n_k| \ = O(\bar \eps^{2^k})  $ \ssk
\item $ | \; \! \partial^2_{xy} S^n_k| = O(\bar \eps^{2^k} \si^{n-k}), \quad \; | \; \! \partial^2_{xz} S^n_k| = O(\bar \eps^{2^k} \si^{n-k})$ \ssk
\item  $| \; \! \partial^2_{yz}S^n_k| = O(\bar \eps^{2^k} ),  \qquad \quad \ | \; \! \partial^2_{zz} S^n_k| = O(\bar \eps^{2^k} ) $ . 
\end{enumerate}
\end{lem}
\ssk
\begin{proof}
Compare the first coordinates of $\Psi^n_k$ in \eqref{recursive form}. Thus 
\begin{align*}
x^n_k &= \alpha_{n,\,k}(x+S_k^n(w)) + \si_{n,\,k}\,t_{n,\,k} \cdot y + \si_{n,\,k}\,u_{n,\,k}\big(z+R_k^n(y) \big)  \\
      &= \alpha_{n,\,k}(x+S_{k+1}^n(w)) + \si_{n,\,k}\,t_{n,\,k} \cdot y + \si_{n,\,k}\,u_{n,\,k} \big(z+R_{k+1}^n(y) \big) + \alpha_k \cdot s_k(w_{k+1}^n) \\
      & \quad \ \ + \si_k u_k \cdot r_k(y_{k+1}^n) .
\end{align*}
Then the recursive formula for $S^n_k$ is as follows
\begin{equation*}
\begin{aligned}
S_k^n(w) &=  S_{k+1}^n(w)+\alpha_{n,\,k}^{-1} \alpha_k \cdot s_k(w_{k+1}^n) +\alpha_{n,\,k}^{-1} \si_{n,\,k} \, u_{n,\,k} \big(R_{k+1}^n(y)-R_k^n(y) \big) \\
 & \quad \ \ +\alpha_{n,\,k}^{-1}\si_k \, u_k \cdot r_k(y^n_{k+1}) .
\end{aligned} \ssk
\end{equation*}
\nin Let us take the first partial derivatives of the above equation in order to have the recursive formulas of each first partial derivatives of $S_k^n(w)  $. Then we obtain that \msk
\begin{align*}
\frac{\di S^n_k}{\di x} = \ &   \frac{\di S^n_{k+1}}{\di x}\left(1+  \frac{\di s_k}{\di x^n_{k+1}}\right)+  \frac{\di s_k}{\di x^n_{k+1}}\\[0.4em]
\frac{\partial S_k^n}{\partial y} 
= \ & \left (1 + \frac{\partial s_k}{\partial x^n_{k+1}} \right ) \frac{\partial S_{k+1}^n}{\partial y} + 
K_1 \la^{n-k-1} \left[ \Big (t_{n,\,k+1} + u_{n,\,k+1} \cdot (R_{k+1}^n)'(y) \Big )  \frac{\partial s_k}{\partial x^n_{k+1}}  \right. \\[0.2em]
& \quad \left.  + \,\frac{\partial s_k}{\partial y^n_{k+1}} +  \Big ( d_{n,\,k+1} + (R_{k+1}^n)'(y) \Big ) \frac{\partial s_k}{\partial z^n_{k+1}} \; \right]  \\[0.2em]
& \quad + 
K_1\la^{n-k-1} u_{n,\,k} \Big ( (R_{k+1}^n)'(y) - (R_k^n)'(y) \Big) +
 K_2 \la^{n-k} u_k \cdot r_k'(y_n^{k+1}) \\[0.4em]
\frac{\partial S_k^n}{\partial z}  
= \ & \left( 1 + \frac{\partial s_k}{\partial x^n_{k+1}} \right) \frac{\partial S_{k+1}^n}{\partial z} + K_1 \la^{n-k-1} \left[ u_{n,\,k+1}  \frac{\partial s_k}{\partial x^n_{k+1}} +\frac{\partial s_k}{\partial z^n_{k+1}} \; \right] \\[-1em]
\end{align*} 
where \,$ \alpha_{n,\,k}^{-1} \;\! \alpha_k \;\! \si_{n,\,k+1} = K_1 (-\la)^{n-k-1} $\, and 
 $ \alpha_{n,\,k}^{-1} \;\! \si_{n,\,k+1} = K_2 (-\la)^{n-k+1} $. 
By Corollary \ref{bounds at k level} and Proposition \ref{bounds of R}, $| \; \! \di s_k / \di x^n_{k+1} |$ is $ O(\si^{n-k}) $ and $| \; \! \di s_k / \di y^n_{k+1} | $ and $| \; \! \di s_k / \di z^n_{k+1} | $ is $ O(\bar \eps^{2^k} \si^{n-k})$. Moreover, $| \: \! t_{n,\,k} |,\, | \: \! u_{n,\,k}|$ and $| \: \!d_{n,\,k}|$ are $ O(\bar \eps^{2^k}) $. With all these facts, each partial derivatives of $ S^n_k $ has bounds as follows \msk
\begin{align*}
\left| \frac{\di S^n_k}{\di x} \right|  \leq & ( 1+ O(\rho^{n-k}) ) \left| \frac{\di S^n_{k+1}}{\di x} \right|+  C \si^{n-k}, \qquad  
\left |\dfrac{\partial S_k^n}{\partial y} \right | \leq & \big (1 + O(\rho^{n-k}) \big ) \left |\dfrac{\partial S_{k+1}^n}{\partial y} \right | + C \bar \eps^{2^k} \\[0.2em]
 \left |\dfrac{\partial S_k^n}{\partial z} \right | \leq & \big (1 + O(\rho^{n-k}) \big ) \left| \dfrac{\partial S_{k+1}^n}{\partial z} \right| + C \bar \eps^{2^k} \\[-1em]
\end{align*} 
for some constant $ C>0 $ and $ \rho \in (0,1) $. 
Hence, using above recursive formulas we have
\msk
\begin{equation*}
\begin{aligned}
\left| \frac{\di S^n_k}{\di x} \right| =O(\si),  \quad \left | \dfrac{\partial S_k^n}{\partial y} \right | =  O( \bar \eps^{2^k}) \quad \text{and} \quad \left| \dfrac{\partial S_k^n}{\partial z} \right| = O( \bar \eps^{2^k})
\end{aligned} \msk
\end{equation*}
for all $ k<n $.
The second partial derivatives of $S^n_k$ are as follows 
\msk
\begin{align*}
\frac{\partial^2 S_k^n}{\partial xy}
&=  \left( 1+ \frac{\partial s_k}{\partial x^n_{k+1}} \right) \frac{\partial^2 S_{k+1}^n}{\partial xy} + \alpha_{n,\,k+1} \left( 1 + \frac{\partial S_{k+1}^n}{\partial x} \right)\frac{\partial^2 s_k }{\partial (x^n_{k+1})^2} \frac{\partial S_{k+1}^n}{\partial y}\\
 & \quad + \si_{n,\, k+1} \left( 1 + \frac{\partial S_{k+1}^n}{\partial x} \right)  \left[ \Big( t_{n,\,k+1} + u_{n,\,k+1}(R_{k+1}^n)'(y) \Big) \frac{\partial^2 s_k }{\partial (x^n_{k+1})^2} + \frac{\partial^2 s_k }{\partial x^n_{k+1}y^n_{k+1}} \right.       \phantom{***}  \\
 & \quad \left. + \Big( d_{n,\,k+1} + (R_{k+1}^n)'(y) \Big) \frac{\partial^2 s_k }{\partial x^n_{k+1}z^n_{k+1}} \; \right] \\[0.3em]
 \frac{\partial^2 S_k^n}{\partial xz} 
& = \left( 1+ \frac{\partial s_k}{\partial x^n_{k+1}} \right) \frac{\partial^2 S_{k+1}^n}{\partial xz} + \alpha_{n,\,k+1} \left( 1 + \frac{\partial S_{k+1}^n}{\partial x} \right)\frac{\partial^2 s_k }{\partial (x^n_{k+1})^2} \frac{\partial S_{k+1}^n}{\partial z} \hspace*{1.25in}\\
 & \quad + \si_{n,\, k+1} \left( 1 + \frac{\partial S_{k+1}^n}{\partial x} \right)  
  \cdot \left[  u_{n,\,k+1} \frac{\partial^2 s_k }{\partial (x^n_{k+1})^2} +  \frac{\partial^2 s_k }{\partial x^n_{k+1}z^n_{k+1}} \; \right] \\[0.3em]
\frac{\partial^2 S_k^n}{\partial yz}  
& =  \left (1+\frac{\partial s_k}{\partial x^n_{k+1}} \right) \frac{\partial^2 S_{k+1}^n}{\partial yz} + \left[ \alpha_{n,\,k+1} \frac{\partial S_{k+1}^n}{\partial z} \frac{\partial S_{k+1}^n}{\partial y} + \si_{n,\, k+1} u_{n,\,k+1} \frac{\partial S_{k+1}^n}{\partial y} \right.  \\
   & \quad + \left. \si_{n,\, k+1} \Big( t_{n,\,k+1} + u_{n,\,k+1}(R_{k+1}^n)'(y) \Big) \Big( \frac{\partial S_{k+1}^n}{\partial z} + K_1 (-\la)^{n-k-1} u_{n,\,k+1} \Big) \right]    \frac{\partial^2 s_k}{\partial (x^n_{k+1})^2}   \\
& + \left( \si_{n,\, k+1} \frac{\partial S_{k+1}^n}{\partial z} + K_4 u_{n,\,k+1} \right) \frac{\partial^2 s_k}{\partial x^n_{k+1}y^n_{k+1}}  \\
& + \si_{n,\, k+1} \left[\, \frac{\partial S_{k+1}^n}{\partial y} + \big(d_{n,\,k+1} + (R_{k+1}^n)'(y) \big) \frac{\partial S_{k+1}^n}{\partial z} \right.  \\
   & \quad \bigg. + K_4 \left( t_{n,\,k+1} +u_{n,\,k+1}\, d_{n,\,k+1} + 2u_{n,\,k+1}(R_{k+1}^n)'(y) \right) \bigg] \frac{\partial^2 s_k}{\partial x^n_{k+1}z^n_{k+1}}  \\
& + K_4 \frac{\partial^2 s_k}{\partial y^n_{k+1}z^n_{k+1}} + K_4\Big( d_{n,\,k+1} + (R_{k+1}^n)'(y) \Big) \frac{\partial^2 s_k}{\partial (z^n_{k+1})^2} \\[0.3em]
\frac{\partial^2 S_k^n}{\partial z^2}  
=  & \left (1+\frac{\partial s_k}{\partial x^n_{k+1}} \right) \frac{\partial^2 S_{k+1}^n}{\partial z^2} \\
& + \left(\alpha_{n,\,k+1} \left( \frac{\di S^n_{k+1}}{\di z} \right)^2 + \si_{n,\,k+1} (1 + u_{n,\,k+1}) \frac{\di S^n_{k+1}}{\di z} + K_4  u_{n,\,k+1}^2 \right) \frac{\di^2 s_k}{\di (x^n_{k+1})^2} \phantom{****} \\
& + 2\left( \si_{n,\,k+1} \frac{\di S^n_{k+1}}{\di z} + K_4 u_{n,\,k+1} \right) \frac{\di^2 s_k}{\di x^n_{k+1}z^n_{k+1}} + K_4 \frac{\di^2 s_k}{\di (z^n_{k+1})^2}
\end{align*}
where $ K_4 = \alpha_{n,\,k}^{-1} \, \alpha_k \, \si_{n,\,k+1}^2 = O(1)$.
\ssk \\
By Lemma \ref{bounds of sk}, Corollary \ref{bounds at k level}, 
and Proposition \ref{bounds of R}, the bounds of $| \;\! \di^2 s_k / \di (x^n_{k+1})^2 | $ is $ O( \si^{n-k})$ and 
 $| \;\! \di^2 s_k / \di uv | $ is $ O(\bar \eps^{2^k} \si^{n-k})$ where $ u, v = x^n_{k+1}, y^n_{k+1}, z^n_{k+1} $ except that both $ u $ and $ v $ are not $ x^n_{k+1} $ simultaneously. The norm of the first and the second partial derivatives of $ s_k $ and the estimation of $| \;\!t_{n,\,k}|, | \;\!u_{n,\,k}|$ and $| \;\! d_{n,\,k}|$ imply the bounds of norm of the second partial derivatives of $ S_k^n $ as follows.
\msk 
\begin{align*}
\left |\dfrac{\partial^2 S_k^n}{\partial xy} \right | & \leq \big (1 + O(\rho^{n-k}) \big ) \left |\dfrac{\partial^2 S_{k+1}^n}{\partial xy} \right | + C \bar \eps^{2^k}\si^{n-k}  \\[0.2em]
 \left |\dfrac{\partial^2 S_k^n}{\partial xz} \right | & \leq \big (1 + O(\rho^{n-k}) \big ) \left |\dfrac{\partial^2 S_{k+1}^n}{\partial xz} \right | + C \bar \eps^{2^k}\si^{n-k} \\[0.2em]
  \left |\dfrac{\partial^2 S_k^n}{\partial yz} \right | & \leq \big (1 + O(\rho^{n-k}) \big ) \left |\dfrac{\partial^2 S_{k+1}^n}{\partial yz} \right | + C \bar \eps^{2^k} \\[0.2em]
  \left |\dfrac{\partial^2 S_k^n}{\partial z^2} \right | & \leq \big (1 + O(\rho^{n-k}) \big ) \left |\dfrac{\partial^2 S_{k+1}^n}{\partial z^2} \right | + C \bar \eps^{2^k} \ .\\[-1em]
\end{align*} 
Hence, \
$ | \:\!\partial^2_{xy} S^n_k| = O(\bar \eps^{2^k}\si^{n-k}), \  | \:\!\partial^2_{xz} S^n_k| = O(\bar \eps^{2^k}\si^{n-k})$, \ 
 $| \:\!\partial^2_{yz}S^n_k| = O(\bar \eps^{2^k}),$ \ and   \ $| \:\!\partial^2_{zz} S^n_k| = O(\bar \eps^{2^k}) $.
\end{proof}
\msk

\subsection{Universal properties of the coordinate change map $ \Psi^n_k $} \label{universal functions}
On the following Lemma \ref{asymptotics of S for k}, we would show that the non-linear part of $ \id + S(x,y,z) $ is a small perturbation of the one-dimensional universal function. 
\ssk \\
Let us normalize the maps, $u_*$ and $g_*$ in Lemma \ref{renormalization at critical value} and Lemma \ref{composition of the presentation function}. Let the fixed point move to the origin and let the derivatives at the origin is one. Define the map $ v_*(x) $ as follows
$$ v_*(x) = \frac{u_*(x+1)-1}{u_* ' (1)} 
$$
Abusing notation, denote the normalized function of $g_*(x)$ to be also the $g_*(x)$ in the following lemma.

\begin{lem} \label{asymptotics of S for k}
There exists a positive constant $\rho <1$ such that for all $k<n$ and for every $y \in I^y $ and $ z \in   I^z$
\begin{align*}
  | \id + S^n_k(\, \cdot \,, y,z) - v_*(\,\cdot \,) \: \!| =& \ O (\bar \eps^{2^k} y + \bar \eps^{2^k}z + \rho^{n-k}) \\
 \textrm{and} \quad  | 1+ \di_x S^n_k (\,\cdot \,, y, z) - v_*'(\,\cdot \,)\: \! | =& \ O (\rho^{n-k}) .
\end{align*}
\end{lem}

\begin{proof}
The map \,$ \id + S^n_k(\,\cdot \, , y,z)$ is the normalized function of $\Psi^n_k$ such that the derivative at the origin is the identity map, and $v_*(\, \cdot \,)$ is also the normalized map of $u_*$, which is the conjugation of the renormalization fixed point at the critical point and the critical value in Lemma \ref{renormalization at critical value}. Thus the normalized map, $ \id + S^n_k(\,\cdot \, , 0,0)$ and the one dimensional map, $G_*^n$ converge to the same function $v_*(\, \cdot \,)$ as $ n \ra \infty $ because the critical value of $f$ and the tip of $F$ moved to the origin as the fixed point of each function $g_*^n$ by the appropriate affine conjugation. 
\ssk \\
By Lemma \ref{asymptotics of non-linear part 1} we have
\begin{align*}
| \; \! \partial_y S^n_k|  \ = O(\bar \eps^{2^k}), \qquad  | \; \! \partial_z S^n_k| \ = O(\bar \eps^{2^k}) 
\end{align*}
and moreover,
 $$ | \;\! \partial^2_{xy} S^n_k| = O(\bar \eps^{2^k}\si^{n-k}), \qquad  | \;\! \partial^2_{xz} S^n_k| = O(\bar \eps^{2^k}\si^{n-k}) . $$
Thus the proof of asymptotic along the section parallel to $ x- $axis is enough to prove the whole lemma. By Lemma \ref{exponential convergence to 1d map}, 
$$  \dist_{C^3} (\, \id + s_k(\,\cdot \, , 0,0),\ g_*(\,\cdot \,) ) = O(\rho^k)
$$
and by Lemma \ref{composition of the presentation function}, we obtain 
\begin{align} \label{C1 convergence of coordinate change}
\dist_{C^1} (\, \id + S^n_k(\,\cdot \, , 0,0),\ G_*^{n-k}(\,\cdot \,) ) = O(\rho^{n-k}) .
\end{align}
Since the $G_* ^{n} \ra v_*$ exponentially fast, we have the exponential convergence of $\id + S^n_k(\,\cdot , 0,0)$ to $v_*(\,\cdot \,)$. 
Hence, the above asymptotic and the exponential convergence at the origin prove the first part of the lemma. Furthermore, $ C^1 $ convergence of \eqref{C1 convergence of coordinate change} implies that 
$$ | \, 1+ \di_x S^n_k (\,\cdot \, , 0, 0) - v_*'(\,\cdot \,) | = O (\rho^{n-k})$$
 where $\rho \in (0,1)$. 
\end{proof}
\msk

\subsection{Estimation of the quadratic part of $S^n_k$ for $n$} \label{asymptotic coordinate for n}
We estimate the asymptotic of $S^n_k$ using the estimation of partial derivatives and recursive formulas. Then it implies the asymptotic of non-linear part of $\Psi^n_k$ for $ n $. 
In order to simplify notations, we would treat the case $ k =0$ and consider the behavior of $S^n_0$ instead of $S^n_k$. 

\begin{lem} \label{asymptotics of non linear part}
The following asymptotic is true
$$ | \: \! [\; x + S^n_0(x,y,z) ] - [\, v_*(x) + a_{F,\,1}\,y^2+ a_{F,\,2}\,yz + a_{F,\,3}\,z^2] | = O(\rho^n)
$$
where constants $ | \:\! a_{F,\,1}|, | \:\! a_{F,\,2}|$  $ | \:\! a_{F,\,3}| $ are $ O(\bar \eps) $ for some $\rho \in (0,1)$.
\end{lem}
\begin{proof}
For any fixed $k \geq 0$, the recursive formula for $n > k$ comes from the $\Psi^{n+1}_k = \Psi^n_k \circ \Psi^{n+1}_n$. Thus by the equation \eqref{recursive form}, we have 
\begin{align} \label{recursive formula for n}
{\bf S}^{n+1}_k (w) = {\bf s}_n (w) + D^{-1}_n \circ {\bf S}^n_k \circ D_n \circ (\id + {\bf s}_n) (w) .
\end{align}
Let $k=0$ for simplicity, and compare each coordinates of the both sides. Then \ssk
\begin{equation*}
\begin{aligned}
& \  (S^{n+1}_0(w),\ 0 ,\ R^{n+1}_0(y))  \\[0.3em]
 = & \ (s_n(w),\ 0,\ r_n(y)) 
 +  \left (
   \begin{array} {ccc}
    \alpha_n^{-1}  & \alpha_n^{-1}(-t_n +d_n u_n)  & - \alpha_n^{-1} u_n \\ [0.3em]
              &\ \  \si_n^{-1} &         \\ [0.3em]
              & -\si_n^{-1} d_n  & \ \  \si_n^{-1}
       \end{array}
\right )     
\left (
       \begin{array}{c}
S^n_0(w) \\ [0.3em]
0 \\ [0.3em]
R^n_0(y)
        \end{array}
\right ) \\[0.8em]
& \qquad \qquad \qquad \qquad \qquad \qquad \qquad \qquad \circ  
\left( 
       \begin{array} {ccc}
 \alpha_n & \si_n t_n & \si_n u_n  \\ [0.3em]
               & \si_n       &      \\ [0.3em]
               & \si_n d_n & \si_n
       \end{array}  
\right)       
\left(
       \begin{array} {c}
       x+ s_n(w) \\ [0.3em]
       y \\ [0.3em]
       z + r_n(y)
       \end{array}  
\right) .
\end{aligned} \msk
\end{equation*}
\nin Thus we obtain the following equation by the straightforward calculations.
\msk
\begin{align*}
& \  (S^{n+1}_0(w),\ 0 ,\ R^{n+1}_0(y))  \\[0.2em]
 = & \ (s_n(w),\ 0,\ r_n(y))  
  +  \left( \dfrac{1}{\alpha_n}S^n_0(w) - \dfrac{1}{\alpha_n}u_n R^n_0(y),\ 0 ,\ \dfrac{1}{\si_n} R^n_0(y) \right) \circ \\
  & \quad  \Big( \alpha_n (x+ s_n(w)) + \si_n t_n \:\! y + \si_n u_n (z + r_n(y)), \  \si_n y, \  \si_n d_n \:\! y + \si_n (z+r_n(y)) \Big) \\[0.5em]
 = & \  (s_n(w),\ 0,\ r_n(y))  \\
 & + \left(  \dfrac{1}{\alpha_n}S^n_0 \Big( \alpha_n (x+ s_n(w)) + \si_n t_n \:\! y + \si_n u_n (z + r_n(y))  , \  \si_n y, \  \si_n d_n \:\! y + \si_n (z+r_n(y)) \Big) \right. \\
 & \qquad  \left. - \:\! \dfrac{1}{\alpha_n} u_n R^n_0(\si_n y) ,\ 0 , \   \dfrac{1}{\si_n} R^n_0(\si_n y) \right) . \\[-1em]
\end{align*}  

\nin Firstly, let us compare the third coordinates of each side of the above equation. Using Taylor's expansion and Lemma \ref{bounds on domain}, we obtain
\begin{align*}
R^{n+1}_0(y) & = r_n(y) + \dfrac{1}{\si_n}R^n_0(\si_n y) \\
 &= \dfrac{1}{\si_n}\,R^n_0(\si_n y) + c_ny^2 +O(\bar \eps^{2^n} y^3)  \ \ \textrm{where} \ c_n = O(\bar \eps^{2^n}) .
\end{align*}
\ssk
Then we have the following form of $R^n_0(y)$.
\begin{align*}
R^n_0(y) & = a_n y^2 + A_n(y) y^3 \\
R^{n+1}_0(y) & = \dfrac{1}{\si_n} \Big( a_n (\si_n y)^2 + A_n(\si_n y) \cdot (\si_n y)^3 \Big) 
+ c_ny^2 + O(\bar \eps^{2^n} y^3) .
\end{align*}
Thus $a_{n+1} = \si_n a_n + c_n $  and $ \| A_{n+1} \| \leq \| \;\!\si_n \|^2   \| A_n  \| + O(\bar \eps^{2^n}) $. 
Hence,  $A_n \ra 0$ \ and \ $ a_n \ra 0$ \ exponentially fast as $n \ra \infty$.
 The image of the vertical plane $(y,z) \ra (0,y,z)$ under the map, $\id + {\bf S}^n_0$ is the graph of the function $\xi_n \colon {\bf I}^v \ra \R$ defined as 
$$ \xi_n (y,z) = (S^n_0 (0,y,z),\ 0,\ R^n_0(y)) .
$$
Since $R^n_0(y)$ is vanished exponentially fast, $ | \;\!\xi_n (y,z)| = | \;\!S^n_0 (0,y,z)| + O(\rho^n)$. Moreover, the second part of Lemma \ref{asymptotics of S for k} implies the following equation
\begin{align} \label{exponential convergence of S}
 \big| \: \! [\,x + S^n_0(x,y,z) ] - [ \,v_*(x) + S^n_0 (0,y,z) ] \big| = O(\rho^n) .
\end{align}
\msk
Secondly, compare the first coordinates of the equation \eqref{recursive formula for n} at $(0,y,z)$
\begin{align*}
 & S^{n+1}_0(0,y,z) \\
  = & \ s_n(0,y,z)  
  + \dfrac{1}{\alpha_n}\,S^n_0 \Big( \alpha_n s_n(0,y,z) + \si_n t_n y + \si_n u_n (z + r_n(y))  , \  \si_n y, \  \si_n d_n y + \si_n (z+r_n(y)) \Big)   \\
 &    -  \dfrac{1}{\alpha_n}\, u_n R^n_0(\si_n y) .
\end{align*}
The estimation of\, $| \: \! \partial^2_{xy} S^n_k|,\,  |  \:\! \partial^2_{xz} S^n_k|$ \,and\, $|\: \! \partial^2_{yz}S^n_k|,\, | \; \! \partial^2_{zz} S^n_k| $ in Lemma \ref{asymptotics of non-linear part 1} implies that 
$$ \frac{\di S^n_0}{\di x} (0,y,z) = O(\si^{n} y + \si^n z) \quad \textrm{and} \quad
 \frac{\di S^n_0}{\di z} (0,y,z) = O( y + z)
$$
respectively. The order of $t_n, u_n, r_n$ and Taylor's expansion of\, $S^n_0$\, at $(0, \si_n y, \si_n z)$ implies that 
\begin{align*}
 & S^{n+1}_0(0,y,z) \\
  = & \ s_n(0,y,z)   \\
  \quad & + \dfrac{1}{\alpha_n} \left[ S^n_0(0,\,\si_n y,\, \si_n z) +  \frac{\di S^n_0}{\di x} (0,\,\si_n y,\, \si_n z) \cdot \Big(\alpha_n s_n(0,y,z) + \si_n t_n y + \si_n u_n (z + r_n(y)) \Big) \right. \\
  \quad & \left. + \frac{\di S^n_0}{\di z} (0,\,\si_n y,\,\si_n z) \cdot \Big( \si_n d_n y + \si_n r_n(y) \Big) \right] -  \dfrac{1}{\alpha_n}\, u_n R^n_0(\si_n y) + O\left( \bar \eps^{2^n} \sum^3_{j=0} y^{3-j}z^j \right)   \\
  = & \ \dfrac{1}{\alpha_n}\, S^n_0(0,\,\si_n y,\, \si_n z) + \sum^2_{i=0} e_{n,\,i}\,y^{2-i}z^i +  O\left( \bar \eps^{2^n} \sum^3_{j=0} y^{3-j}z^j \right) 
\end{align*} 
where $e_{n,\,i} = O(\bar \eps^{2^n})$ \ for \ $i=0,1,2$. 
Then we can express $S^n_0(0,y,z)$ as the quadratic and higher order terms, 
\begin{align*}
 S^n_0(0,y,z) &=  a_{n,\,1}\,y^2 + a_{n,\,2}\,yz + a_{n,\,3}\,z^2 + A_n(y,z) \left( \sum^3_{j=0} c_j\, y^{3-j}z^j \right) .
\end{align*} 
The recursive formula for $S^n_0(0,y,z)$ implies that
\begin{align*}
 \quad & S^{n+1}_0(0,y,z) \\
=& \   \dfrac{1}{\alpha_n} \left[ a_{n,\,1}(\si_n y)^2 + a_{n,\,2}(\si_n y \: \si_n z) +  a_{n,\,3}(\si_n z)^2
+ A_n(\si_n y, \si_n z) \left(\sum^3_{j=0} c_j\,(\si_n y)^{3-j}(\si_n z)^j \right) \right] \\
\quad & +  \sum^2_{i=0} e_{n,\,i} \: y^{2-i}z^i +  O\left( \bar \eps^{2^n} \sum^3_{j=0} y^{3-j}z^j \right) .
\end{align*}
Hence, $a_{n+1,\,i} = \dfrac{\,\si^2}{\,\alpha_n} \, a_{n,\,i} + \displaystyle\sum^2_{j=0} e_{n,\,j}$ for $i=0,1,2$ \, and moreover,
$ \|A_{n+1} \| \leq \|A_n\| \cdot \dfrac{\ | \;\! \si_n|^3}{| \;\!\alpha_n|} + O(\bar \eps^{2^n})$. 
It implies that $a_{n,\,i} \ra a_{F,\,i}$ \  for $i=0,1,2 $\, and\ $\| A_n\| \ra 0$ exponentially fast as $n \ra \infty$. The exponential convergence of $ S^n_0(0,y,z) $ to the quadratic function of $y$ and $z$ and the equation \eqref{exponential convergence of S} show the asymptotic of $S^n_0(x,y,z)$.
\end{proof}
\ssk
\begin{rem}
The above Lemma can be generalized for $S^n_k$ as follows 
$$ \big| \: \! [\; x + S^n_k(x,y,z) ] - [\, v_*(x) + a_{F,\,1}\,y^2+ a_{F,\,2}\,yz + a_{F,\,3}\,z^2] \big| = O(\rho^{n-k}) .
$$
The constants $ | \:\!a_{F,\,i}| $ for $i=1,2,3$ \,of $S^n_k$ are $O(\bar \eps^{2^k})$.
\end{rem}

\msk

\subsection{Universality of Jacobian determinant, $ \Jac R^nF $}
Let the $n^{th}$ renormalized map of $F$ be $R^nF \equiv F_n = (f_n -\eps_n,\;x,\; \de_n)$. Recall that $\Psi^n_{\tip} \equiv \Psi^n_{v^n}$ from $n^{th}$ level to $0^{th}$ level and the tip $\tau_F $ is contained in $ B^n_{v^n}$ for all $n \in \N $. Thus $\Psi^n_{\tip}$ is the original coordinate change map rather than the normalized function $\Psi^n_0$ conjugated by translations $T_n$. 
Recall the equation \eqref{chain rule} again
\begin{align*}
\Jac F_n(w) = \Jac F^{2^n}(\Psi^n_{\tip} (w)) \frac{\Jac \Psi^n_{\tip} (w)}{\Jac \Psi^n_{\tip}(F_nw)}  
                  = b^{2^n} \frac{\Jac \Psi^n_{\tip}(w)}{\Jac \Psi^n_{\tip}(F_nw)} (1+ O(\rho^n)).
\end{align*} 
\msk
\begin{thm} [Universal expression of Jacobian determinant] \label{Universality of the Jacobian}
Let $F $ be the three dimensional H\'enon-like diffeomorphism $ \II_B (\bar \eps)$ for sufficiently small $\bar \eps >0$, we obtain that
$$ \Jac F_n = b^{2^n}_F a(x) \: (1+ O(\rho^n))
$$
where $b_F$ is the average Jacobian of \;$F$ and $a(x)$ is the universal positive function for some $\rho \in (0, 1)$.
\end{thm}
\begin{proof}
Let us consider the affine maps 
$$ T \colon w \mapsto w- \tau, \qquad T_n \colon w \mapsto w - \tau_n $$
where $\tau_n$ is the tip of $R^nF$.      
Then we can consider the map
$$ L^n \colon w \mapsto (D^n_0)^{-1}(w - \tau)  $$
as the local chart of $B^n$. On these local charts, we write maps with the boldfaced letters if the maps are conjugated by its local charts in this proof.
$$ {\bf F}_n = T_n \circ F_n \circ T_n^{-1}, \qquad \id + \ {\bf S}^n_0 = L^n \circ \Psi^n_{\tip} \circ T_n^{-1}
$$
By the definition of coordinate change map, $\Psi^n_{\tip}$ and the normalized map, $\Psi^n_0$, the following diagram is commutative.
\begin{displaymath}
\xymatrix @C=1.5cm @R=1.5cm
 {
T_n(B)  \ar[d]_*+{{\Psi^n_0}}& B \ar[l]_*+{{T_n}} \ar[r]^*+{{F_n}}  \ar[d]_*+{{\Psi^n_{\tip}}} & F_n(B) \ar[r]^*+{{T_n} } \ar[d]^*+{{\Psi^n_{\tip}}} &(T_n \circ F_n)(B)  \ar[d]^*+{{\Psi^n_0}} \\
 T(B_n)                   & B^n \ar[l]_*+{T}  \ar[r]^*+{{F^{2^n}}}            &F^{2^n}(B^n) \ar[r]^*+{T}         & (T \circ F^{2^n})(B^n)  }
\end{displaymath}
\ssk \\
Since any translation does not affect Jacobian determinant, the ratio of Jacobian determinant of coordinate change maps is as follows
\ssk
\begin{equation} \label{ratio of jacobian}
\begin{aligned} 
\frac{\Jac \Psi^n_{\tip} (w)}{\Jac \Psi^n_{\tip}(F_nw)} = \frac{\Jac \Psi^n_0 ({\bf w}_n)}{\Jac \Psi^n_0({\bf F}_n{\bf w}_n)}
= \frac{1 + \di_x S^n_0({\bf w}_n)}{1 + \di_x S^n_0({\bf F}_n{\bf w}_n)}
\end{aligned} \ssk
\end{equation}
where ${\bf w}_n = T_n(w)$. 
By Theorem \ref{exponential convergence to 1d map}, the tip, $\tau_n$ converges to $ \tau_{\infty} = (f_*(c_*),\ c_*,\ 0)$ \,exponentially fast where $ c_* $ is the critical point of $ f_*(x) $. It implies the following limits
\begin{align*}
T_n & \ra  T_{\infty} \colon w \mapsto w- \tau_{\infty} \\
{\bf w}_n = T_n(w) & \ra  T_{\infty}(w) \\
{\bf F}_n{\bf w}_n  & \ra {\bf F}_* \circ T_{\infty}(w) = T_{\infty} \circ F_*(w) = (f_*(x) - f_*(c_*),\ x- c_*,\ 0) 
\end{align*}
and each convergence is exponentially fast. Hence, Lemma \ref{asymptotics of non linear part} implies that the following convergence 
\begin{align} \label{convergence of dS over x}
1+ \di_xS^n_0 \ra v_* '
\end{align}
is exponentially fast. The equations \eqref{ratio of jacobian}, \eqref{convergence of dS over x} and convergence of ${\bf F}_n{\bf w}_n$ to the ${\bf F}_* \circ T_{\infty} $ imply the following convergence
\begin{align} \label{universal 1dim limit}
\frac{\Jac \Psi^n_{\tip} (w)}{\Jac \Psi^n_{\tip}(F_nw)} \lra  
 \frac{v_* '\big( x- c_* \big)}{v_* ' \big( f_*(x)- f_*(c_*) \big)} \,\equiv a(x)
\end{align}
where $w=(x,y,z)$. 
Moreover, this convergence is exponentially fast. 
\ssk \\ 
The positivity of $a(x)$ comes from two facts. Firstly, Jacobian determinant of orientation preserving diffeomorphism is non-negative at every point and we assumed that each infinitely renormalizable map, $F \in \II(\bar \eps)$, is orientation preserving on each level. Secondly, renormalization theory of one dimensional maps at the critical value implies the non vanishing property of $v_* '$ with sufficiently small perturbation.
\end{proof}

\begin{rem}
 The universality of Jacobian determinant does not imply the universality of renormalized map $F_n$ because the Jacobian determinant of $F_n$, namely, $\di_y \eps_n \cdot \di_z \de_n - \di_z \eps_n \cdot \di_y \de_n $ does not imply universal expression of each element of Jacobian matrix, $ DF_n $.
\end{rem}
\bsk

\section{Toy model H\'enon-like map in three dimension} \label{Tangent bundle splitting}
\nin Let H\'enon-like map satisfying $ \eps(w) = \eps(x,y) $, that is, $ \di_{z} \eps \equiv 0 $ be {\em toy model H\'enon-like map}. Denote the toy model map by $ F_{\mod} $. Then the projected map $ \pi_{xy} \circ F_{\mod} = F_{2d} $ from $ B $ to $ \R^2 $ is exactly two dimensional H\'enon-like map. Let the horizontal-like diffeomorphism of $ F_{\mod} $ be $ H_{\mod} $. Thus $ \pi_{xy} \circ H_{\mod} $ is the horizontal map, $ H_{2d} $ of $ F_{2d} $. If $ F_{\mod} $ is renormalizable map, then renormalization of toy model map is a skew product of renormalization of two dimensional H\'enon-like map. In other words, we have $ \pi_{xy} \circ RF_{\mod} = RF_{2d} $. 
\msk

\begin{prop} \label{renormalization of toy model}
Let $ F_{\mod} = (f(x) - \eps(x,y),\ x,\ \de(w)) $ be a toy model map in $ \II(\bar \eps) $. Then $ n^{th} $ renormalized map $ R^nF_{\mod} $ is also a toy model map, that is,
$$ \pi_{xy} \circ R^nF_{\mod} = R^nF_{2d} $$
for every $ n \in \N $. Moreover, 
$ \eps_n(w) = \ b_1^{2^n}a(x)\:\! y (1+O(\rho^n)) $ where $ b_1 $ is the average Jacobian of two dimensional map, $ F_{2d} = \pi_{xy} \circ F_{\mod} $. 
\end{prop}

\ssk
\nin Let $ b $ be the average Jacobian of $ F_{\mod} \in \II(\bar \eps) $. Define another universal number, say $ b_2 $ is the ratio $ b/b_1 $. Then by the above Proposition $ \di_z \de_n \asymp b_2^{2^n} $ for every $ n \in \N $. 

\msk

\subsection{Tangent bundle splitting under $ DF_{\mod} $}
Let $ DF $ 
be the Fr\'echet derivative of $ F $. 
For the given point $ w = (x,\, y,\, z) $, let us denote $ w_i = (x_i,\, y_i,\, z_i) = F^i(x,\, y,\, z) $. The derivative of $ F $ has the block matrix form
\[
\renewcommand{\arraystretch}{1.3}
DF_{\mod} =  \left(
\begin{array} {cc|c}
\multicolumn{2}{c|}{\multirow{2}{*}{$DF_{2d}$}} & { 0}  \\ 
 & & { 0} \\    \hline  
\di_x \de & \di_y \de &\di_{z} \de
\end{array}
\right) .
\]
\ssk
Let us take a simpler notation below
\begin{equation*} \msk
\begin{aligned}
DF_{\mod}(x,y,z) =
\begin{pmatrix} 
A(w)& {\bf 0} \\
C(w)& D(w)
\end{pmatrix} \equiv
\begin{pmatrix}
A_w & {\bf 0} \\
C_w & D_w
\end{pmatrix} .
\end{aligned} 
\end{equation*} 
where $ A_w = DF_{2d}(x,y) $, 
${\bf 0} = \bigl(\begin{smallmatrix}
{ 0} \\ 
{ 0}
\end{smallmatrix} \bigr)$ , 
$ C_w = \big(\,\di_x \de(w) \ \, \di_y \de(w)\,\big) $ and $ D_w = \di_{z} \de(w) $. 
Since we assume that $ F_{\mod} $ and $ F_{2d} $ are diffeomorphisms, $ DF_{\mod} $ and $ A_w $ are invertible. It implies that $ D_w $ is invertible at each $ w $. Let $ w_N $ be $ F^N(w) $ and the derivative of the $ N^{th} $ iterated map $ F^N_{\mod} $ be $ DF^N_{\mod} $. We express $ DF^N_{\mod} $ as the block matrix form as follows \msk
\begin{equation} \label{block matrix of DF-mod infty}
\begin{aligned}
DF^N_{\mod}(x,y, z) = 
\begin{pmatrix}
A_N(w) & {\bf 0} \\
C_N(w) & D_N(w) 
\end{pmatrix} \equiv
\begin{pmatrix}
A_N& {\bf 0} \\
C_N& D_N
\end{pmatrix} .
\end{aligned} 
\end{equation}

\nin 
Then for each $ N \geq 1 $, 
\begin{equation}
\begin{aligned}
\begin{pmatrix}
A_N& {\bf 0} \\
C_N& D_N
\end{pmatrix} =
\begin{pmatrix}
A_1(w_{N-1}) & {\bf 0} \\
C_1(w_{N-1}) & D_1(w_{N-1})
\end{pmatrix} \cdot
\begin{pmatrix}
A_{N-1}& {\bf 0} \\
C_{N-1}& D_{N-1}
\end{pmatrix} .
\end{aligned} \msk
\end{equation}

\nin Let $ A_0 \equiv 1 $, $ C_0 \equiv 1 $, $ D_0 \equiv 1 $ and $ w = w_0 $ for notational compatibility. Then by the direct calculations, we obtain
\begin{equation} \label{component of DF-n infty}
\begin{aligned}
A_N &= A_1(w_{N-1})\;\! A_{N-1} = \prod_{i=0}^{N-1} A_1(w_{N-i-1}) \\
D_N &= D_1(w_{N-1})\;\! D_{N-1} = \prod_{i=0}^{N-1} D_1(w_{N-i-1}) \\[0.3em]
C_N &= C_1(w_{N-1})\;\! A_{N-1} + D_1(w_{N-1})\;\! C_{N-1} \\
&= \sum_{i=0}^{N-1} D_i(w_{N-1-i})\, C_1(w_{N-1-i}) \, A_{N-1-i} . \\
\end{aligned} 
\end{equation}

\nin We see that $ [DF^N_{\mod}(w)]^{-1} = DF^{-N}_{\mod}(F^{N}(w)) $ by inverse function theorem. Thus using block matrix expressions, $ [DF^N_{\mod}(w)]^{-1} $ is
\begin{equation} \label{block matrix of DF-1 infty}
\begin{aligned}
DF^{-N}_{\mod} = 
\begin{pmatrix}
A^{-1}_N & {\bf 0} \\[0.4em]
- D^{-1}_N \;\! C_N \;\! A_N^{-1} & D^{-1}_N
\end{pmatrix}
\end{aligned}
\end{equation}
at the point, $ F^{N}(w) $. 
\msk \\
Let us consider the tangent bundle, $ T_{\Gamma}B $ of $ DF_{\mod} $ over a compact invariant set $ \Gamma $. The lower triangular block matrix of the $ DF_{\mod} $ implies the existence of $ DF_{\mod} $ invariant tangent subbundle, say $ \mathbb{E}^1 $. Then if $ \| \:\! C_1 \| $ is small enough, then $ T_{\Gamma}B $ is decomposed to $ \mathbb{E}^1 \times \mathbb{E}^2 $ where $ \mathbb{E}^2 = T_{\Gamma}B / \mathbb{E}^1 $. Moreover, since $ B = \Dom(F) $ is contractible, the tangent bundle, $ T_{\Gamma}B $ is trivial. Thus decomposition of tangent space at each point, for instance, $ T_w(\R^2 \times \R)= T_w \R^2 \oplus T_w \R $, is still true for the whole tangent bundle, that is, $ \mathbb{E}^1 \times \mathbb{E}^2 = \mathbb{E}^1 \oplus \mathbb{E}^2 $.
\msk \\ Let the cone at $ w $ with some positive number $ \gamma $ be
\begin{align} \label{complement cone infty}
\CC(\gamma)_w = \{ u + v \; | \; u \in \R^2 \times \{0\},\; v \in \{0\} \times \R \ \ \textrm{and} \ \ \dfrac{1}{\gamma}\,\| \:\! u\| >  \|v\|  \; \} .
\end{align}
Cone field over a given compact invariant set $ \Gamma $ is the union of cones at every points in $ \Gamma $
\begin{align} \label{complement cone field infty}
\CC(\gamma) = \bigcup_{w \in \,\Gamma} \CC(\gamma)_w .
\end{align}
\msk
Let $ \| DF \| $ be the operator norm of $ DF $. The minimum expansion rate (or the strongest contraction rate) of $ DF $ is defined by the equation, $ \| DF^{-1} \| = \dfrac{1}{m(DF)} $. \footnote{The operator norm is defined on the linear operator at each point. For example, 
$$ \| DF_w \| = \sup_{ \| v \| =1} \{ \| DF_w v \| \} $$
The value $ \| DF \| $ is defines as $ \sup_{ w \in B } \| DF_w \| $. }

\ssk
\begin{lem} \label{upper bound of C-N A-N infty}
Let $ A_N $, $ {\bf 0} $, $ C_N $ and $ D_N $ be components of $ DF^N_{\mod} $ defined on \eqref{block matrix of DF-mod infty}. Suppose that $ \| D_1 \| < m(A_1) $. Then $ \| \;\! C_N \;\! A_N^{-1} \| < \kappa $ for some $ \kappa > 0 $ independent of $ N $.
\end{lem}
\begin{proof}
By \eqref{block matrix of DF-1 infty},
\begin{equation} \label{A-N -1 prod}
A_N^{-1}(w_N) = \prod_{i=0}^{N-1} A_1^{-1}(w_i) = \prod_{j=0}^{N-1-i} A_1^{-1}(w_j) \, \prod_{j=N-i}^{N-1} A_1^{-1}(w_j) = A_{N-1-i}^{-1} \, A^{-1}_i(w_{N-i})
\end{equation}
By \eqref{component of DF-n infty}, $ \| D_k \| \leq \| D_1 \|^k $ and by \eqref{A-N -1 prod} $ m(A_k) \geq m(A_1)^k $ for any $ k \in \N $.
Then \msk
\begin{equation}
\begin{aligned}
\| \;\! C_N \;\! A_N^{-1}(w_N) \|
&=  \ \Big\| \sum_{i=0}^{N-1} D_i(w_{N-1-i})\, C_1(w_{N-1-i}) \, A_{N-1-i}  \;\! A_N^{-1}(w_N) \Big\| \\
&=  \ \Big\| \sum_{i=0}^{N-1} D_i(w_{N-1-i})\, C_1(w_{N-1-i}) \,A^{-1}_i(w_{N-i}) \Big\| \\
&\leq \ \sum_{i=0}^{N-1} \| D_i \| \| C_1 \| \| A^{-1}_i \| = \ \| \:\! C_1 \| \sum_{i=0}^{N-1} \frac{\| D_i \|}{m(A_i)} \\
&\leq \ \| \:\! C_1 \| \sum_{i=0}^{N-1} \left( \frac{\| D_1 \|}{m(A_1)} \right)^i \leq \ \| \:\! C_1 \| \sum_{i=0}^{\infty} \left( \frac{\| D_1 \|}{m(A_1)} \right)^i \\[0.4em]
& = \ \frac{\| \:\! C_1 \| \cdot  m(A_1)}{m(A_1) - \| D_1 \| }
\end{aligned} \msk
\end{equation}
Then we can choose $ \kappa = \dfrac{\| \:\! C_1 \| \cdot  m(A_1)}{m(A_1) - \| D_1 \| } $ which is independent of $ N $.
\end{proof}

\msk
\begin{lem} \label{Invariance of cone field infty}
Let $ F_{\mod} \in \II(\bar \eps) $ with small enough $ \bar \eps >0 $. Suppose that $ \| D_1 \| \leq \frac{\ \rho}{\ 2} \cdot m(A_1) $ for some $ \rho \in (0,1) $. Let $ \CC(\gamma) $ 
be the cone field which is defined on \eqref{complement cone field infty} with cones in \eqref{complement cone infty}. Then $ \CC(\gamma) $ is invariant under $ DF^{-1}_{\mod} $ for all sufficiently small $ \gamma > 0 $. More precisely, $ DF_{\mod} $ has dominated splitting over any given invariant compact set $ \Gamma $. 
\end{lem}
\begin{proof}
Let us take a vector \footnote{The vector $ u \in \R^2 $ depends on each point $ w \in \R^2 \times \R $. Then $ w \mapsto u(w) $ is a map from $ \R $ to $ \R^2 $.} $ (u \ v) \in \R^2 \times \R $ in the cone field $ \CC(\gamma) $ with small enough $ \gamma > 0 $. Since $ \| u \| < \gamma \| v \| $, we may assume that $ \| v \| =1 $ and $ \| u \| < \gamma $. Thus for cone field invariance, it suffice to show that \footnote{Use the matrix form in \eqref{block matrix of DF-1 infty} with a vector on the cone field $ \CC(\gamma) $.
\begin{align*}
\begin{pmatrix}
A^{-1}_N & {\bf 0} \\[0.4em]
- D^{-1}_N \;\! C_N \;\! A_N^{-1} & D^{-1}_N
\end{pmatrix}
\begin{pmatrix}
u \\
1
\end{pmatrix}
\end{align*}    
The number one in the vector means the the number corresponding to the identity map.     }
\msk
\begin{equation*}
\big\| A_N^{-1} \;\! u \;\! \big( -D_N^{-1} \;\! C_N \;\! A_N^{-1} \;\! u + D_N^{-1} \big)^{-1} \big\| \leq \rho_1 \| u\|
\end{equation*}
for some $ \rho_1 \in (0,1) $ 
. Let us calculate the lower bound of $ m( -D_N^{-1} \;\! C_N \;\! A_N^{-1} \;\! u + D_N^{-1} ) $. The equation
\begin{equation*}
-D_N^{-1} \;\! C_N \;\! A_N^{-1} \;\! u + D_N^{-1} = D_N^{-1} \big( - C_N \;\! A_N^{-1} \;\! u + \Id \big) 
\end{equation*}
and Lemma \ref{upper bound of C-N A-N infty}, $ \| - C_N \;\! A_N^{-1} \;\! u \| < \kappa \gamma $. Then with the small enough $ \gamma > 0 $ and the definition of minimum expansion rate, 
\begin{equation*}
m( -D_N^{-1} \;\! C_N \;\! A_N^{-1} \;\! u + D_N^{-1} ) \geq m( D_N^{-1})\cdot(-\kappa \gamma +1 )
\end{equation*} 

\nin Then
\begin{equation}
\begin{aligned}
\big\| A_N^{-1} \;\! u \;\! \big( -D_N^{-1} \;\! C_N \;\! A_N^{-1} \;\! u + D_N^{-1} \big)^{-1} \big\| 
\leq & \  \| A_N^{-1} \;\! u \| \;\! \big\| \big( -D_N^{-1} \;\! C_N \;\! A_N^{-1} \;\! u + D_N^{-1} \big)^{-1} \big\| \\[0.3em]
\leq & \ \frac{\| A_N^{-1} \| \|u \| }{ m \big( -D_N^{-1} \;\! C_N \;\! A_N^{-1} \;\! u + D_N^{-1} \big) } 
 \\[0.3em]
\leq &
\ \frac{\| A_N^{-1} \| \|u \| }{ m ( D_N^{-1}) \;\! m(- C_N \;\! A_N^{-1} \;\! u + \Id ) } \\[0.3em]
\leq & \ \frac{\| A_N^{-1} \| \|u \| }{ (1 - \kappa \gamma) \, m(D_N^{-1})} \ =  \ \frac{ \|D_N \|}{ (1 - \kappa \gamma) \,m(A_N)} \, \|u \| \\[0.3em]
\leq & \ \frac{1}{1 - \kappa \gamma} \left( \frac{\| D_1 \|}{ m( A_1)} \right)^N \|u \| .
\end{aligned} \msk
\end{equation}

\nin Thus for all small enough $ \rho >0 $ and $ \gamma > 0 $ satisfying $ \kappa \gamma \ll 1 $, we see
\begin{equation} \label{contracting rate of cone}
\frac{1}{1 - \kappa \gamma} \; \frac{\| D_1 \|}{ m( A_1)} \leq \frac{\rho_1}{2}
\end{equation}
for some $ \rho_1 \in (0,1) $. Hence, there exists decomposition of the tangent bundle, $ T_{\Gamma}B $ which is invariant under $ DF_{\mod} $ and the invariant subbundles satisfies dominated splitting condition. Moreover, dominated splitting implies the continuity of invariant sections.
\end{proof}
\msk

\begin{rem}
In Lemma \ref{upper bound of C-N A-N infty}, \ssk the components of block matrix, $ A_1 $, $ D_1 $, $ A_N $ and $ D_N $ depend on each point $ w \in \Gamma $. Thus $ \dfrac{\| D_1(w) \|}{ m(A_1(w))} < \dfrac{\;1}{\;2}\,\rho_w $ for some positive $ \rho_w < 1 $. Then the actual assumption is that the set, $ \{ \, \rho_w > 0 \,|\; w \in \Gamma \, \} $ is totally bounded above by the number less than $ \frac{\,1}{\;2} $. However, since $ \Gamma $ is compact, the set $ \{\, \rho_w \,|\; w \in \Gamma \,\} $ is precompact. Then $ \rho $ can be chosen as $ \sup_{w \in \Gamma} \{\, \rho_w \,|\; w \in \Gamma \,\} $. 
Then $ \kappa $ in Lemma \ref{upper bound of C-N A-N infty} is independent of $ w \in \Gamma $ 
. Moreover, the cone field $ \CC(\gamma) $ in Lemma \ref{Invariance of cone field infty} is contracted in uniform rate by $ DF^{-1}_{\mod} $.
\end{rem}

\msk
\subsection{Tangent bundle splitting under a small perturbation of toy model map}
The existence of invariant cone field under $ DF_{\mod} $ is still true when a small perturbation of $ DF_{\mod} $ is chosen. Let us consider the block diagonal matrix of $ DF $. Let the following map be a {\em perturbation} of toy model map, $ F_{\mod} (w) = (f(x) - \eps_{2d}(x,y),\;x,\;\de(w)) $ 
\begin{equation}  \label{perturbation of model maps infty}
F(w) = (f(x) - \eps(w),\ x,\ \de(w))
\end{equation}
where $ \eps(w) = \eps_{2d}(x,y) + \widetilde{\eps}(w) $. Thus $ \di_{z} \eps(w) = \di_{z} \widetilde{\eps}(w) $. \msk
\begin{equation} \label{matrix symbolic expression of DF infty}
\begin{aligned}
DF  = 
\renewcommand{\arraystretch}{1.3}
 \left(
\begin{array} {cc|c}
\multicolumn{2}{c|}{\multirow{2}{*}{$D \widetilde {F}_{2d}$}} & \di_{z} \eps  \\ 
 & & { 0} \\    \hline  
\di_x \de & \di_y \de &\di_{z} \de
\end{array}
\right)
= \left( \renewcommand{\arraystretch}{1.3} \begin{array}{c|c}
A & B \\
\hline
C & D
\end{array} \right)
\end{aligned} \bsk
\end{equation}
where $ D\widetilde{F}_{2d} = \begin{pmatrix}
f'(x) - \di_x \eps(w) & -\di_y \eps(w) \\[0.3em]
1  & 0
\end{pmatrix} $.
Observe that if $ B \equiv {\bf 0} $, then $ F $ is $ F_{\mod} $. 
\ssk \\
Let us quantify a small perturbation keeping invariance of cone fields under the assumption $ b_1 \ll b_2 $. One of the sufficient condition is that $ \| \di_z \eps \| \ll b_F $ and $ \| \de \|_{C^1} \ll \| \di_y \eps \| $. See the lemma below.

\msk
\begin{lem} \label{Invariance of cone filed perturbation infty}
Let $ F $ be a perturbation of the toy model map $ F_{\mod} $ defined in \eqref{perturbation of model maps infty} and $ A $, $ B $, $ C $ and $ D $ are components of block matrix of $ DF $ defined in \eqref{matrix symbolic expression of DF infty}. Suppose that $ \| D_1 \| \leq \frac{\rho_1}{2} \cdot m(A_1) $ for some $ \rho_1 \in (0,1) $. Suppose \ssk also that $ \| B \| \| C \| \leq \rho_0 \cdot m(A) \cdot m(D) $ where $ \rho_0 < \frac{\kappa \gamma}{2} $ for sufficiently small $ \gamma > 0 $. Then the cone field $ \CC(\gamma) $ defined on \eqref{complement cone field infty} is invariant under $ DF^{-1} $ .
\end{lem}

\begin{proof}
The matrix form of $ DF^{-1} $ is
\begin{equation}  \label{DF-1 matrix form infty}
\begin{aligned}
\begin{pmatrix}
A^{-1} + \zeta_{11} & \zeta_{12} \\[0.3em]
-D^{-1}C (A^{-1} + \zeta_{11}) & D^{-1} \zeta_{22}
\end{pmatrix}
\end{aligned}  \msk
\end{equation}

\nin where $ \zeta_{12} = -(A- BD^{-1}C)^{-1}BD^{-1} $,\; $ \zeta_{11} = -\zeta_{12}\;\! CA^{-1} $\, and \, $ \zeta_{22} = \Id - C \;\! \zeta_{12} $. \footnote{If the bounded operator $ T $ has $ \| T \| <1 $, then $ \Id - T $ is invertible. Moreover,
\begin{equation*}
(\Id - T)^{-1} = \sum_{n=0}^{\infty} T^{\,n}
\end{equation*}
Since, $ \| T^n \| \leq \| T \|^n $ for every $ n \in \N $, $ \| (\Id - T)^{-1} \| \leq \dfrac{1}{1 - \| T \|}  $. Equivalently, we obtain the lower bound of the minimum expansion rate, $ m( \Id - T) \geq 1 - \| T \| $. } Thus \bsk
\begin{equation} \label{zeta-12 C norm upper bound}
\begin{aligned}
\| \:\! \zeta_{12} \| \| C \| & = \| -(A- BD^{-1}C)^{-1}BD^{-1} \| \| C \| \\
& \leq \| (\Id - A^{-1}BD^{-1}C)^{-1} \| \| A^{-1} \| \| B \| \| D^{-1} \| \| C \| \\[0.3em]
& \leq \frac{1}{1 - \| A^{-1}BD^{-1}C \| } \cdot \frac{\| B \| \|C \| }{ m(A)\;\! m(D)} \\
& \leq \frac{1}{1 - \rho_0}\cdot \rho_0 < \kappa \gamma .
\end{aligned} \msk
\end{equation}

\nin Let us calculate the upper bound of $ \| B \| $. Then \ssk
\begin{equation} \label{B norm upper bound}
\begin{aligned}
\| B \| \| C \| &< \frac{\kappa \gamma}{2} m(A)\;\! m(D) \\
\| B \| &< \frac{m(A)}{m(A) - \| D \|} \cdot m(A)\;\! m(D) \cdot \frac{\kappa \gamma}{2} \\
\frac{\| B \|}{ m(D)} &< \frac{[m(A)]^2}{m(A) - \| D \|} \cdot \frac{\kappa \gamma}{2} \\
&\leq \frac{m(A)}{ 2(1 - \frac{\rho_1}{2})} \cdot \gamma = \frac{m(A)}{2 - \rho_1} \cdot \gamma < m(A)\cdot \gamma
\end{aligned}
\end{equation}

\nin Thus by \eqref{zeta-12 C norm upper bound} and \eqref{B norm upper bound},
\begin{equation*}
\begin{aligned}
\| \;\!\zeta_{12} \| &< \frac{ 1}{1- \rho_0} \cdot \frac{ \| B \| }{m(A)\;\! m(D) } \cdot \gamma < \frac{\gamma}{m(A)} \cdot \frac{2}{2- \kappa \gamma} \cdot \frac{\| B \|}{m(D)} \\
&< \frac{\gamma}{m(A)} \cdot \frac{2 }{2- \kappa \gamma} \cdot m(A)\cdot \gamma = \frac{2 \gamma^2}{2 - \kappa \gamma} .
\end{aligned} \msk
\end{equation*}

\nin Take a vector $ (u \ v) \in \R^2 \times \R $ such that $ \| u \| < \gamma \| v \| $ in the cone $ \CC(\gamma) $ at a point $ w \in \Gamma $. We may assume that $ \| v \| = 1 $, that is, $ \| u \| < \gamma $. For the invariance of cone field under $ DF^{-1} $, it suffice to show that 
\begin{equation*}
\big\| \big[\,(A^{-1} + \zeta_{11}) u + \zeta_{12} \,\big] \big[\, -D^{-1}C (A^{-1} + \zeta_{11})\;\! u + D^{-1} \zeta_{22} \,\big]^{-1} \big\| < \rho_2 \gamma 
\end{equation*}
for some $ \rho_2 \in (0,1) $. Let us estimate the upper bound of norm of the first factor
\msk
\begin{equation} \label{numerator invariant cone infty}
\begin{aligned}
\| (A^{-1} + \zeta_{11})\;\! u + \zeta_{12} \| &\leq  \| A^{-1}\;\! u \| + \| \;\!\zeta_{11}\;\! u + \zeta_{12} \| \\[0.3em]
&=  \| A^{-1} u \| + \| \;\!\zeta_{12} \;\!(-CA^{-1} \;\! u + \Id ) \| \\
& \leq \frac{\gamma}{ m(A)} + 
 \frac{2 \gamma^2}{2 - \kappa \gamma}\cdot ( 1 + \kappa \gamma ) \\[0.3em]
& = \frac{\gamma}{ m(A)} \,\left[\,1 + \frac{2(1 + \kappa \gamma) \gamma}{2- \kappa \gamma} \cdot m(A)  \, \right] .
\end{aligned} \msk
\end{equation}
Let us consider the lower bound of the second factor 
\begin{equation} \label{denominator invariant cone infty}
\begin{aligned}
 m( - D^{-1}C(A^{-1} + \zeta_{11} )\;\!u + D^{-1} \zeta_{22} ) \geq & \ m(  D^{-1})\, m(CA^{-1} u + C \;\!\zeta_{11}\;\! u - \zeta_{22} ) \\
= & \ m(  D^{-1})\, m(CA^{-1} u - C \;\!\zeta_{12}\;\!CA^{-1} u - \Id + C\;\! \zeta_{12})   \\
= & \ m(  D^{-1})\, m( CA^{-1} u - C \;\!\zeta_{12}\;\! (CA^{-1} u - \Id ) - \Id) \\
= & \ m(  D^{-1})\, m( CA^{-1} u  - \Id - \,C \;\!\zeta_{12} (CA^{-1} u - \Id ))  \\
= & \ m(  D^{-1})\, m(\,[ \Id -  C \;\!\zeta_{12} ]\, [ CA^{-1} u  - \Id ]) \\
\geq & \ m(  D^{-1})\, m( \Id -  C \;\!\zeta_{12} ) \, m( CA^{-1} u  - \Id ) \\
\geq & \ \frac{ (1 - \kappa \gamma) ( 1- \kappa \gamma) }{\| D \|} = \frac{(1 - \kappa \gamma)^2 }{ \| D \| } .
\end{aligned} \bsk
\end{equation}

\nin Then the inequalities, \eqref{numerator invariant cone infty} and \eqref{denominator invariant cone infty} implies that \msk
\begin{equation*}
\begin{aligned}
& \quad \ \big\| \big[\,(A^{-1} + \zeta_{11}) u + \zeta_{12} \,\big] \big[\, -D^{-1}C (A^{-1} + \zeta_{11})\;\! u + D^{-1} \zeta_{22} \,\big]^{-1} \big\| \\[0.4em]
& \leq \frac{\| (A^{-1} + \zeta_{11})\;\! u + \zeta_{12} \|}{m( - D^{-1}C(A^{-1} + \zeta_{11} )\;\! u + D^{-1} \zeta_{22} ) } \\[0.4em]
& \leq \frac{\gamma}{ m(A)} \,\left[\,1 + \frac{2(1 + \kappa \gamma) \gamma}{2- \kappa \gamma} \cdot m(A) \, \right] \cdot \frac{ \| D \| }{(1 - \kappa \gamma)^2 } \\[0.5em]
& = \frac{1}{(1 - \kappa \gamma)^2} \, \left[\,1 + \frac{2(1 + \kappa \gamma) \gamma}{2- \kappa \gamma} \cdot m(A) \, \right] \cdot \frac{\| D \| }{m(A)} \; \gamma .
\end{aligned} \msk
\end{equation*}

\nin Thus for small enough $ \gamma > 0 $, the constant, $ \dfrac{1}{(1 - \kappa \gamma)^2} \, \left[\,1 + \dfrac{2(1 + \kappa \gamma) \gamma}{2- \kappa \gamma} \cdot m(A) \, \right] $ is less than two. Hence, the cone field $ \CC(\gamma) $ is invariant under $ DF $.

\end{proof}

\nin Then the tangent bundle $ T_{\Gamma}B $ has the splitting with subbundles $ E^1 \oplus E^2 $ such that

\ssk
\begin{enumerate}
\item $ T_{\Gamma}B = E^1 \oplus E^2 $. \msk
\item Both $ E^1 $ and $ E^2 $ are invariant under $ DF $. \msk
\item $ \| DF^n |_{E^1(x)} \| \|  DF^{-n} |_{E^2(F^{-n}(x))} \| \leq C \mu^n$ for some $ C>0 $ and $ 0 < \mu < 1 $ and $ n \geq 1 $.
\end{enumerate}
\msk
\nin Thus $ T_{\Gamma}B $ has {\em dominated splitting} over the compact invariant set $ \Gamma $. Moreover, the dominated splitting implies that invariant sections 
are continuous by Theorem 1.2 in \cite{New}. Then the maps, $ w \mapsto E^i(w) $ for $ i =1,2 $ are continuous.

\msk

\subsection{Non existence of continuous invariant line field over $ \OO_{F_{\mod}} $}
Toy model map $ F_{\mod} $ with invariant splitting over $ \OO_{F_{\mod}} $ satisfying Lemma \ref{Invariance of cone field infty} has the continuous invariant plane field. Moreover, the set of lines perpendicular to $ xy- $plane,
$$ \bigcup_{(x_0,y_0) \in I^x \times I^y} \{ (x_0,y_0,z) \ | \ z \in I^z \} $$
is forward invariant under $ F_{\mod} $. Thus $ F_{\mod} $ is semi-conjugate to $ F_{2d} \equiv \pi_{xy} \circ F_{\mod} $ by the projection, $ \pi_{xy} $. Sufficient condition of the existence of invariant plane field in Lemma \ref{Invariance of cone field infty} can be substituted by the condition $ b_2 \ll b_1 $ by Lemma \ref{estimation of 2d min-norm}. 
 
\msk
\begin{lem} \label{Perpendicular line ss manifold}
Let $ F_{\mod} $ be in $ \II(\bar \eps) $ with $ b_2 \ll b_1 $. If each line perpendicular to $ xy- $plane in $ \Dom(F) $ over $ \OO_{F_{\mod}} $ is the local strong stable manifold of $ w \in \OO_{F_{\mod}} $. Moreover, 
$$ W^{ss}(w) \cap \OO_{F_{\mod}} = \{w \}  $$ 
for each $ w \in \OO_{F_{\mod}} $. 
\end{lem}

\begin{proof}
By cone field construction in Lemma \ref{Invariance of cone field infty}, invariant direction represent $ \di_z \de $ is unique in each cones. However, each line perpendicular to $ xy- $plane in $ \Dom(F) $ over $ \OO_{F_{\mod}} $ has already invariant direction in each cones. Then these two invariant directions are equal by the uniqueness. It is the strongest contracting direction which is one dimensional. However, at each point of $ \OO_{F_{\mod}} $, the stable manifold is two dimensional. Thus strongest one dimensional direction is for the strong stable manifold. Moreover, the constant invariant direction implies that straight lines are the local strong stable manifolds.
\msk \\
If $ w $ and $ w' $ in $ \OO_{F_{\mod}} $ is in the same (strong) stable manifold, then $ \dist (F_{\mod}^n(w), F_{\mod}^n(w')) $ converges to zero as $ n \ra \infty $. However, each point of the critical Cantor set is the limit of nested sequence of boxes, for instance,
\begin{align*}
\{ p \} = \bigcap_{n \geq 0} B^n_{{\bf w}_n} 
\end{align*}
for every $ p \in \OO_{F_{\mod}} $. 
Thus if $ w \neq w' $, then each boxes $ B^n_{{\bf w}_n} $ and $ B^n_{{\bf w'}_n} $ for $ w $ and $ w' $ respectively are disjoint from each other for every big enough $ n $. Since box, $ B^n_{{\bf w}_n} $ with any word $ {\bf w} \in W^n $ is invariant under $ F^{2^{n+1}} $, $ \dist (F_{\mod}^n(w),\ F_{\mod}^n(w')) \geq c>0 $ for all sufficiently big $ n $. Then the fact that $ \dist (F_{\mod}^n(w),\ F_{\mod}^n(w')) \ra 0 $ as $ n \ra \infty $ implies that $ w = w' $.  
\end{proof}


\msk
\begin{lem}
Let $ F_{\mod} $ be in $ \II(\bar \eps) $ with $ b_2 \ll b_1 $. Then each strong stable manifold of the point, $ w \in \OO_{F_{\mod}} $ in the critical Cantor set, $ W^{ss}(w) $ meets a single point of the critical Cantor set, $ \OO_{2d} $. In particular, $ \pi_{xy} $ is a bijection between $ \OO_{F_{\mod}} $ and $ \OO_{F_{2d}} $. 
\end{lem}
\begin{proof}
Since $ F_{2d} = \pi_{xy} \circ F_{\mod} $ and $ \pi_{xy} \circ R^nF_{\mod} = R^nF_{2d} $ for each $ n \geq 1 $, we obtain that $ \pi_{xy} \big( B^n_{{\bf w}_n} \big) $ is the box for $ F_{2d} $ for every $ {\bf w}_n \in W^n $ and for every $ n \geq 1 $. Passing the limit, $ \pi_{xy}(w) $ for every $ w \in \OO_{F_{\mod}} $ is well defined. Moreover, $ \OO_{F_{2d}} $ is the limit of boxes which is the image of three dimensional boxes under $ \pi_{xy} $. Thus $ \pi_{xy} $ is onto map. Furthermore, the fact that $ \pi_{xy} $ is the map along $ W^{ss}(w) $ for each $ w \in \OO_{F_{\mod}} $ implies that $ \pi_{xy} $ is one-to-one by Lemma \ref{Perpendicular line ss manifold}.
\end{proof}

\nin The invariant splitting with uniform contraction implies the existence of invariant single surfaces contains the given invariant compact set.

\msk
\begin{defn}
A submanifold $ Q $ which contains $ \Gamma $ is {\em locally invariant} under $ f $ if there exists a neighborhood $ U $ of $ \Gamma $ in $ Q $ such that $ f(U) \subset Q $. 
\end{defn}

\nin The necessary and sufficient condition for the existence of these submanifolds, see \cite{CP} or \cite{BC}. 
\begin{thm}[\cite{BC}] \label{Existence of invariant submanifold} Let $ \Gamma $ be an invariant compact set with a dominated splitting $ T_{\Gamma}M = E^1 \oplus E^2 $ such that $ E^1 $ is uniformly contracted. Then $ \Gamma $ is contained in a locally invariant $ C^1 $ submanifold tangent to $ E^2 $ if and only if the strong stable leaves for the bundle $ E^1 $ intersect the set $ \Gamma $ at only one point.
\end{thm}

\nin It is remarkable that invariant submanifolds in Theorem \ref{Existence of invariant submanifold} are robust under $ C^1 $ perturbation by \cite{BC}. By Lemma \ref{Perpendicular line ss manifold} and the above Theorem \ref{Existence of invariant submanifold}, there exist $ C^1 $ single surfaces which contain $ \OO_{F_{\mod}} $ and tangent to invariant plane field at each point in $ \OO_{F_{\mod}} $. 
\ssk \\
\nin Let $ Q $ be a locally invariant surface which is tangent to invariant plane field and which contains $ \OO_{F_{\mod}} $. Since the lines perpendicular to $ xy- $plane is strong stable manifolds, $ Q $ be the graph of $ C^1 $ function $ \xi $ from a subset of $ I^x \times I^y $ to $ I^z $. The map $ (x,y) \mapsto (x,y, \xi(x,y)) $, say the graph map of $ \xi $, is a $ C^1 $ diffeomorphism. Moreover, the map, $ \pi_{xy}|_Q $ is the inverse of the graph map of $ \xi $. 
For notational simplicity, denote $ \OO_{F_{\mod}} $ by $ \OO_{\mod} $ and $ \OO_{F_{2d}} $ by $ \OO_{2d} $. Abusing notation let the graph map, $ (x,y) \mapsto (x,y, \xi(x,y)) $ be just $ \xi $ unless it makes any confusion.

\msk

\begin{thm}
Let $ F_{\mod} $ be in $ \II(\bar \eps) $ with $ b_2 \ll b_1 $. Let $ Q $ be a locally invariant surface tangent to the continuous invariant plane field, say $ E $ over $ \OO_{F_{\mod}} $. Then any invariant line field in $ E $ over $ \OO_{F_{\mod}} $ is discontinuous at the tip, $ \tau_F $. 
\end{thm} 
\begin{proof}
Denote $ \Dom(F_{2d}) $ by $ P $. Since the graph map $ \xi $ is a $ C^1 $ diffeomorphism, the following diagram is commutative.
\begin{displaymath}
\xymatrix @C=1.5cm @R=1.5cm
 {
T_{\OO_{2d}}P  \ar[d]_*+ {\pi}  \ar[r]^*+{(D\xi, \xi)}    & T_{\OO_{\mod}}Q \ar[d]^*+{\pi'}   \\
 \OO_{2d} \ar[r]^*+{ \xi}      & \OO_{\mod}
   }
\end{displaymath} \ssk \\
where the tangent map is $ (D\xi, \xi)(v, w) = (D\xi(w)\cdot v,\; \xi(w)) $ for each $ (v,w) \in T_{\OO_{2d}}P $ and both $ \pi $ and $ \pi' $ are the projections from the bundle to the base space, that is, for each $ (v, w) \in \text{bundle} $, $ \pi(v, w) = w $ and $ \pi'(v, w) = w $ respectively. 
\ssk \\
The image of any invariant tangent subbundle of $ T_{\OO_{2d}}P $ is an invariant subbundle of $ T_{\OO_{\mod}}Q $. Thus without loss of generality, we may assume that $ (D\xi, \xi)(E^1_{2d}) = E^1 $. Let $ \gamma $ and $ \gamma' $ be the invariant sections under $ F_{2d} $ and $ F|_{\,Q} $ respectively.
\begin{displaymath}
\xymatrix @C=1.5cm @R=1.5cm
 {
E^1_{2d}  \ar@{<-}[d]_*+ {\gamma}  \ar[r]^*+{(D\xi, \xi)}    & E^1 \ar@{<-}[d]^*+{\gamma'}   \\
 \OO_{2d} \ar[r]^*+{ \xi}      & \OO_{\mod}
   }
\end{displaymath}
Since $ \xi $ is $ C^1 $ function, the tangent map $ (D\xi, \xi) $ is continuous at $ (v, w) \in E^1_{2d} $. Then the section $ \gamma $ is continuous if and only if $ \gamma' $ is continuous because $ \xi $ is a diffeomorphism. However, any invariant line field under $ DF_{2d} $ on the Cantor set $ \OO_{2d} $ is not continuous at the tip, $ \tau_{F_{2d}} $ by Proposition 9.3 in \cite{CLM}. Hence, there is no continuous invariant line field under $ DF|_{\,Q} $ on any $ C^1 $ invariant surface $ Q $ under $ F $.
\end{proof}

\bsk


\titleformat{\section}[display]{\normalfont\Large\bfseries}{Appendix~\Alph{section}}{12pt}{\Large}

\begin{appendices}

\section{Lyapunov exponents and splitting of tangent bundle}
Three dimensional H\'enon-like map has two non zero Lyapunov exponents on its critical Cantor set, $ \OO_F $. 
Recall that $ \di_z \de \asymp b_2 $ for all $ w \in B $ by Proposition \ref{renormalization of toy model}. The number $ b_2 $ is defined as $ b/b_1 $ where $ b $ and $ b_1 $ is the average Jacobian of $ F_{\mod} $ and $ F_{2d} $ respectively. Actually two negative numbers, $ \log b_1 $ and $ \log b_2 $ are the exponents which affect the dynamical properties in two dimensional H\'enon-like map and contraction rate along $ z- $direction respectively. We would see that $ m(DF^{2^n}_{2d}) $ is bounded above by $ b_1^{2^n} $ for all $ w \in B $ but it is not much smaller than $ b_1^{2^n} $ (See Lemma \ref{estimation of 2d min-norm} below). 


\nin Recall that $ \Jac F_{2d} $ is $ \di_y \eps(x,y) $. Since $ F_{2d} $ is an orientation preserving diffeomorphism, 
$ \di_y \eps(x,y) $ has the positive infimum. Let this infimum be $ m_{2d} $. 
Similarly, define $ m_{2d,\,n} = \displaystyle\inf \left\{ \,\di_y \eps_n(x,y) \;| \; {(x,\,y) \in \,B(R^nF_{2d})} \, \right\} $ for $ n^{th} $ renormalized map, $ F_{2d,\,n} \equiv R^nF_{2d} $. 
Recall that $ \Psi^n_{0,\,2d} = ( \alpha_{n,\,0}\,(x+ S^n_0(w) + \si_{n,\,0}\, t_{n,\,0}\cdot y ,\ \si_{n,\,0}\cdot y  ) $. Thus the derivative of $ \Psi^n_{0,\,2d} $ at the tip is as follows
\msk
\begin{equation*}
\begin{aligned}
D\Psi^n_{0,\,2d} =
\begin{pmatrix}
\alpha_{n,\,0}\,(1+ \di_xS^n_0) & \alpha_{n,\,0}\cdot \di_yS^n_0 + \si_{n,\,0}\, t_{n,\,0} \\[0.3em]
0 & \si_{n,\,0}
\end{pmatrix} .
\end{aligned} \msk
\end{equation*}

\begin{lem} \label{estimation of 2d min-norm}
Let $ F_{\mod} \in \II(\bar \eps) $ with sufficiently small $ \bar \eps > 0 $. Then the infimum of the derivative of two dimensional map, $ m(DF^{2^n}_{2d}) \asymp \si^n b_1^{2^n} $ for every $ n \in \N $. 
\end{lem}
\begin{proof}
Firstly, let us show that $ m(DF^{2^n}_{2d}) \lesssim \si^n b_1^{2^n} $. If $ n=1 $, then we get the upper bound of $ m(DF_{2d}) $ as follows
\begin{equation*}
\begin{aligned}
\frac{1}{m(DF_{2d})} = \| DF_{2d}^{-1} \| \geq \left \| \, \frac{1}{\di_y \eps} 
\begin{pmatrix}
0 & \di_y \eps \\
-1 & f'(x) - \di_x \eps
\end{pmatrix}
\begin{pmatrix}
1 \\
0
\end{pmatrix} \right \| 
& \geq \frac{1}{\di_y \eps(x,y)} . 
\end{aligned} \msk
\end{equation*}
Then $ m(DF_{2d}) \leq \di_y \eps(x,y) = \Jac F_{2d}$ for every point $ (x,y) $. Similarly, since $ F_{2d} $ is infinitely renormalizable, $ m(DF_{2d,\,n}) \leq \di_y \eps_n(x,y) = \Jac F_{2d,\,n}$ for every $ n \in \N $. Then $ m(DF_{2d,\,n}) \leq m_{2d,\,n} $.   
Let us generalize the above idea for the estimating the norm of $ DF_{2d}^{-2^n} $. For example,
\msk
\begin{equation*}
\begin{aligned}
\sqrt{a^2 + c^2} = 
\begin{Vmatrix} 
\begin{pmatrix}
a \\ c
\end{pmatrix}
\end{Vmatrix} =
\begin{Vmatrix}
\begin{pmatrix}
a & b \\
c & d
\end{pmatrix}
\begin{pmatrix}
1 \\ 0
\end{pmatrix}
\end{Vmatrix}
\leq \sup_{ \| (v_1 \, v_2 ) \| =1} \left\{
\begin{Vmatrix}
\begin{pmatrix}
a & b \\
c & d
\end{pmatrix}
\begin{pmatrix}
v_1 \\ v_2
\end{pmatrix}
\end{Vmatrix} \right\} .
\end{aligned} \msk
\end{equation*}
Thus a lower bound of $ \| DF^{-2^n}_{2d} \| $ is determined by the length of the first column of $ DF^{-2^n}_{2d} $. Recall that $ F^{2^n} \circ \Psi^n_{\tip} = \Psi^n_{\tip} \circ F_{n} $, that is,  
$$ F^{-2^n} = \Psi^n_{\tip} \circ F_{n}^{-1} \circ (\Psi^n_{\tip})^{-1} . $$
Since composition of linear functions does not affect the norm of derivative, let us use $ \Psi^n_0 $ instead of $ \Psi^n_{\tip} $ in order to simplify notations. Let us choose points, $ w'' $, $ w' $ and $ w $ with following relation
\begin{equation*}
w'' \xmapsto{ (\Psi^n_{\tip})^{-1} } w' \xmapsto{ \quad F_{n}^{-1} \ \ } w .
\end{equation*}
Recall that
\begin{equation*}
\begin{aligned}
\Psi^n_{0,\,2d}(w) &= \ ( \alpha_{n,\,0} \cdot (x + S^n_0(w)) + \si_{n,\,0}\,t_{n,\,0}\cdot y,\ \si_{n,\,0}\cdot y) \\[0.3em]
F_{n,\,2d}(w) &= \ (f_n(x) - \eps_n(w), \ x) .
\end{aligned}
\end{equation*}
Then
\begin{equation}
\begin{aligned}
& \quad \ \ DF^{-2^n}_{2d}(w'') \\[0.3em]
&= \ D\Psi^n_{0,\,2d}(w) \cdot DF_{n}^{-1}(w') \cdot D(\Psi^n_{0,\,2d})^{-1}(w'') \\[0.3em]
&= \ D\Psi^n_{0,\,2d}(w) \cdot \big[\,DF_{n}^{-1}(w) \,\big]^{-1} \cdot \big[\,D(\Psi^n_{0,\,2d})^{-1}(w') \,\big]^{-1} \\[0.4em]
&= \ \begin{pmatrix}
\alpha_{n,\,0} \cdot (1 + \di_x S^n_0(w)) & \alpha_{n,\,0}\, \di_y S^n_0(w) + \si_{n,\,0}\,t_{n,\,0} \\[0.2em]
0 & \si_{n,\,0}
\end{pmatrix} 
\cdot \frac{1}{\di_y \eps_n(w)} \,
\begin{pmatrix}
0 & \di_y \eps_n(w) \\[0.2em]
-1 & f_n'(x) - \di_x \eps_n(w)
\end{pmatrix} \\[0.4em]
& \qquad \cdot \frac{1}{\si_{n,\,0}\,\alpha_{n,\,0} \cdot (1 + \di_x S^n_0(w))} \,
\begin{pmatrix}
\si_{n,\,0} & - \alpha_{n,\,0} \cdot \di_y S^n_0(w) - \si_{n,\,0}\,t_{n,\,0} \\[0.3em]
0 & \alpha_{n,\,0} \cdot (1 + \di_x S^n_0(w)) .
\end{pmatrix}
\end{aligned} \bsk
\end{equation}
Let us choose temporary expression of $ D\Psi^n_{0,\,2d}(w) \cdot \big[\,DF_{n}^{-1}(w) \,\big]^{-1} \cdot \big[\,D(\Psi^n_{0,\,2d})^{-1}(w') \,\big]^{-1} $ as follows
\begin{equation*}
\begin{aligned}
DF^{-2^n}_{2d}(w'') &= \ 
\begin{pmatrix}
A_1 & B_1 \\
0 & \si_{n,\,0}
\end{pmatrix} \cdot
\begin{pmatrix}
0 & 1 \\
C_2 & D_2
\end{pmatrix} \cdot
\begin{pmatrix}
A_3 & B_3 \\
0 & \si_{n,\,0}^{-1}
\end{pmatrix} \\[0.5em]
&= \ \begin{pmatrix}
B_1 C_2 A_3 & B_1 C_2 B_3 + (A_1 + B_1 D_2)\, \si_{n,\,0}^{-1} \\[0.2em]
\si_{n,\,0}\, C_2 A_3 & \si_{n,\,0}\,C_2 B_3 + D_2
\end{pmatrix} .
\end{aligned} \msk
\end{equation*}
Then
\begin{equation*}
\begin{aligned}
B_1 C_2 A_3 = &\ - \frac{1}{ \di_y \eps_n(w)} \cdot \frac{\alpha_{n,\,0} \cdot \di_y S^n_0(w) + \si_{n,\,0}\,t_{n,\,0} }{\alpha_{n,\,0} \cdot (1 + \di_x S^n_0(w'))} \\[0.2em]
=& \ - \frac{1}{b_1^{2^n}} \cdot \frac{1}{(-\si)^n} \cdot \frac{t_{*,\,0} + O(\si^n)}{v_*(x - \pi_x(\tau_{\infty}))} \, \big(1 + O(\rho^n) \big) 
\\[0.5em]
 \si_{n,\,0}\, C_2 A_3 =& \ \si_{n,\,0} \cdot \left( - \frac{1}{ \di_y \eps_n(w)} \right) \cdot \frac{ 1}{\alpha_{n,\,0} \cdot (1 + \di_x S^n_0(w'))} \\
 =& \ - \frac{1}{b_1^{2^n}} \cdot \frac{1}{(-\si)^n} \cdot \frac{1}{v_*(x - \pi_x(\tau_{\infty}))} \, \big(1 + O(\rho^n) \big) .
\end{aligned} \msk
\end{equation*}

\nin By the universality of two dimensional H\'enon-like maps, $ 1+ \di_xS^n_0(w) = v_*'(x) + O(\rho^n) $ where $ v_*(x) $ is a diffeomorphism on its domain and $ \di_y \eps_n \asymp b_1^{2^n} $. Moreover, $ | \, \si_{n,\,0} | \asymp \si^n $, $ \alpha_{n,\,0} \asymp \si^{2n} $, $ | \,t_{n,\,0} | = O(\bar \eps) $, and $ \di_y S^n_0 = a_Fy + O(\rho^n) $ for $ a_F = O(\bar \eps) $. 
For the detailed proof about the above asymptotic of two dimensional H\'enon-like maps, see Section 7 in \cite{CLM}. Then
\begin{equation*}
\frac{C_0}{\si^n b_1^{2^n}} \leq \| DF^{-2^n}_{2d} \| 
\end{equation*}
for some $ C_0 > 0 $ independent of $ n $. 
\ssk \\
Secondly, let us estimate an upper bound of the norm of $ DF_{2d}^{-2^n} $. Let us observe the following fact for later use. The unit vector $ \bigl(\begin{smallmatrix}
v_1 \\ 
v_2
\end{smallmatrix} \bigr)$ 
, that is, $ v_1^2 + v_2^2 =1 $, satisfies the following inequality by Cauchy-Schwarz inequality.
\begin{align} \label{matrix 2norm 1}
\left\| \begin{pmatrix}
a& b \\
c& d
\end{pmatrix}
\begin{pmatrix}
v_1 \\
v_2
\end{pmatrix} \right\| \leq \sqrt{a^2 + b^2 + c^2 +d^2} .
\end{align}
Moreover, if $ ad-bc \neq 0 $, then
\begin{align} \label{matrix 2norm 2}
\left\| \begin{pmatrix}
a& b \\
c& d
\end{pmatrix}^{-1}
\begin{pmatrix}
v_1 \\
v_2
\end{pmatrix} \right\| \leq \frac{1}{| ad-bc |}\sqrt{a^2 + b^2 + c^2 +d^2} .
\end{align}
Recall the equation
 $$ F_{2d}^{-2^n} = \Psi^n_0 \circ F_{2d,\,n}^{-1} \circ (\Psi^n_0)^{-1} . $$ 
 Then
\begin{align*}
\| DF^{-2^n}_{2d} \| \leq \| D\Psi^n_{0,\,2d} \| \cdot \| DF_{2d,\,n}^{-1} \| \cdot \| D(\Psi^n_{0,\,2d})^{-1} \| .
\end{align*}
By \eqref{matrix 2norm 1} and \eqref{matrix 2norm 2}, the upper bounds of norms are as follows
\begin{equation*}
\begin{aligned}
\| D\Psi^n_{0,\,2d} \|^2 = & \left\|
\begin{pmatrix}
\alpha_{n,\,0}\,(1+ \di_xS^n_0) & \alpha_{n,\,0}\, \di_yS^n_0 + \si_{n,\,0}\, t_{n,\,0} \\
0 & \si_{n,\,0}
\end{pmatrix} \right\|^2 \\[0.5em]
\leq & \ \sup_{w \in \,B(R^nF) } \{ \alpha_{n,\,0}^2(1+ \di_xS^n_0)^2 + (\alpha_{n,\,0}\, \di_yS^n_0 + \si_{n,\,0}\, t_{n,\,0} )^2 + \si_n^2 \} \\
\| DF_{2d,\,n}^{-1} \|^2 = & \left\|
\, \frac{1}{\di_y \eps_n} 
\begin{pmatrix}
0 & \di_y \eps_n \\
-1 & f_n'(x) - \di_x \eps_n
\end{pmatrix} \right\|^2 \\[0.5em]
\leq & \ \sup_{w \in \,B(R^nF) } \left\{ \frac{1}{(\di_y \eps_n)^2} \big( (\di_y\eps_n)^2 + 1 + (f_n'(x) - \di_x\eps_n)^2 \big) \right\} .
\end{aligned}
\end{equation*}
Hence, 
\begin{align*}
& \| DF_{2d}^{-2^n} \|^2  \leq  
 \sup_{w \in \,B(R^nF) } \{ \alpha_{n,\,0}^2(1+ \di_xS^n_0)^2 + (\alpha_{n,\,0}\, \di_yS^n_0 + \si_{n,\,0}\, t_{n,\,0} )^2 + \si_n^2 \} \\
&\qquad  \cdot \sup_{w \in \,B(R^nF) } \left\{ \frac{1}{(\di_y \eps_n)^2} \big( (\di_y\eps_n)^2 + 1 + (f_n'(x) - \di_x\eps_n)^2 \big) \right\} \\
&\qquad  \cdot \sup_{w \in \,B(F) } \left\{ \frac{1}{\alpha_{n,\,0}^2(1+ \di_xS^n_0)^2 \cdot \si_{n,\,0}^2}\, \Big[ \,\alpha_{n,\,0}^2(1+ \di_xS^n_0)^2 + (\alpha_{n,\,0}\, \di_yS^n_0 + \si_{n,\,0}\, t_{n,\,0} )^2 + \si_n^2 \Big] \right\} \\[0.5em]
&\qquad \qquad \ \leq \frac{C_2}{b_1^{2^{n+1}}  \si^{2n}}
\end{align*}
for some $ C_2>0 $. Then 
\begin{align*}
\| DF^{-2^n}_{2d} \|  \leq \frac{C_1}{b_1^{2^{n}}  \si^{n}} 
\end{align*}
where $ C_1 > 0 $ is independent of $ n $ for each $ n\in \N $. Hence, $ m(DF_{2d}^{2^n}) \asymp b_1^{2^{n}}  \si^{n} $.
\end{proof}
\msk
\nin By the above lemma, we obtain the estimation of $ m(DF_{2d}^{N}) $ for each $ N $. Each natural number $ N $ can be expressed as a dyadic number
. Let us assume that 
$$ N = \sum_{j = 0}^k 2^{m_j} $$
where $ m_k > m_{k-1} > \cdots > m_1 > m_0 \geq 0 $. Then we estimate the minimum expansion rate as follows
$$ m(DF_{2d}^{N}) \geq C b_1^N \si^{\sum_{j=0}^k m_j}  $$
for some $ C > 0 $ independent of $ n $. Observe that $ \log_2 N \asymp \sum_{j=0}^k m_j $ for each big enough $ N $. Let us choose a number smaller than $ \si $. For example, let us take $ \frac{1}{4} < \si $. Thus
$$ \si^{\sum_{j=0}^k m_j} \asymp \si^{\log_2 N} \geq \left( \frac{1}{4} \right)^{ \log_2 N} = \frac{1}{N^2} $$ 
Then the minimum expansion rate has the lower bound as follows
$$ m(DF_{2d}^{N}) \geq C \frac{b_1^N}{N^{\alpha}}  $$
where $ -\alpha < \log_2 \si < 0 $ \,for some $ C>0 $.

\end{appendices}

\bsk


\bibliographystyle{alpha}

\end{document}

%% file: macrosH3D1st.tex
\newtheorem{thm}{Theorem}[section]

\newtheorem{cor}[thm]{Corollary}
\newtheorem{lem}[thm]{Lemma}
\newtheorem{prop}[thm]{Proposition}

\theoremstyle{remark}

\theoremstyle{definition}
\newtheorem{defn}{Definition}[section]
\newtheorem{rem}{Remark}[section]

\newtheorem*{acknowledgement}{Acknowledgement}

\numberwithin{equation}{section}
\numberwithin{figure}{section}

\font\nt=cmr7

\def\note#1
{\marginpar
{\nt $\leftarrow$
\par
\hfuzz=20pt \hbadness=9000 \hyphenpenalty=-100 \exhyphenpenalty=-100
\pretolerance=-1 \tolerance=9999 \doublehyphendemerits=-100000
\finalhyphendemerits=-100000 \baselineskip=6pt
#1}\hfuzz=1pt}

\renewcommand{\Im}{\operatorname{Im}}

\newcommand{\di}{\partial}

\newcommand{\ra}{\rightarrow}

\def\lra{\longrightarrow}

\def\ssk{\smallskip}
\def\msk{\medskip}
\def\bsk{\bigskip}

\def\nin{\noindent}

\newcommand{\diam}{\operatorname{diam}}
\newcommand{\dist}{\operatorname{dist}}

\newcommand{\inter}{\operatorname{Int}}
\renewcommand{\mod}{\operatorname{mod}}

\newcommand{\id}{\operatorname{id}}

\newcommand{\Id}{\operatorname{Id}}

\newcommand{\Orb}{\operatorname{Orb}}

\newcommand{\tip}{\operatorname{tip}}

\newcommand{\loc}{\operatorname{loc}}

\newcommand{\const}{\mathrm{const}}
\def\loc{{\mathrm{loc}}}

\newcommand{\eps}{{\varepsilon}}

\newcommand{\de}{{\delta}}
\newcommand{\la}{{\lambda}}
\newcommand{\La}{{\Lambda}}
\newcommand{\si}{{\sigma}}

\newcommand{\CC}{{\mathcal C}}

\newcommand{\II}{{\mathcal I}}

\newcommand{\GG}{{\mathcal G}}

\newcommand{\OO}{{\mathcal O}}

\newcommand{\TT}{{\mathcal T}}

\newcommand{\UU}{{\mathcal U}}

\newcommand{\N}{{\mathbb N}}

\newcommand{\R}{{\mathbb R}}

\newcommand{\Z}{{\mathbb Z}}

\def\BPhi{{\boldsymbol{\BPhi}}}
\def\B0{{\mathbf{0}}}

\newcommand{\Jac}{\operatorname{Jac}}

\newcommand{\Dom}{\operatorname{Dom}}

\catcode`\@=12

\def\Empty{}
\newcommand\oplabel[1]{
  \def\OpArg{#1} \ifx \OpArg\Empty {} \else
  	\label{#1}
  \fi}
		
\newcommand{\comm}[1]{}
\newcommand{\comment}[1]{}